\documentclass{article}   
\usepackage{amsthm}
\usepackage[margin=1in]{geometry}

\usepackage{amsmath,verbatim,amssymb,amsfonts,amscd,mathtools,xcolor}
\usepackage{algorithm}
\usepackage[noend]{algpseudocode}
\usepackage{mathrsfs}
\usepackage{wrapfig}
\usepackage{caption} 
\usepackage{subfigure}
\usepackage{lscape}
\usepackage{graphicx,tabularx,tabulary}
\usepackage{booktabs}
\usepackage[breaklinks=true, colorlinks=true, linkcolor={blue!90!black}, citecolor={blue!90!black}, urlcolor={blue!90!black}]{hyperref}

\usepackage{tikz-cd}
\usepackage{mhchem} 

\newcommand{\cI}{\mathcal{I}}
\newcommand{\cL}{\mathcal{L}}
\newcommand{\cD}{\mathcal{D}}
\newcommand{\cE}{\mathcal{E}}
\newcommand{\cG}{\mathcal{G}}
\newcommand{\cN}{\mathcal{N}}
\newcommand{\cP}{\mathcal{P}}
\newcommand{\cR}{\mathcal{R}}
\newcommand{\cS}{\mathcal{S}}
\newcommand{\cU}{\mathcal{U}}
\newcommand{\cV}{\mathcal{V}}
\newcommand{\f}{f_{u^*,v^*,i^*,j^*}}
\newcommand{\bs}[1]{\boldsymbol{#1}}
\newcommand{\wt}[1]{\widetilde{#1}}

\newcommand{\ol}{\overline}
\newcommand{\ul}{\underline}
\newcommand{\mtc}{\mathcal}

\newcommand{\tn}{\textnormal}
\newcommand{\Frdp}[1]{F^{(#1)}_{\tn{\ref{opt:RDP}}}}
\newcommand{\Fopt}[1]{F^{*(#1)}_{\tn{\ref{opt:RDP}}}}

\newcommand{\Cp}{C^{\textnormal{plat}}}
\newcommand{\Cph}{\widetilde{C}^{\textrm{plat}}}

\newcommand{\Psp}{P^{\textrm{SP}}}
\newcommand{\Psprel}{P^{\textrm{SP}}_{1\textrm{-rel}}}
\newcommand{\mP}{\mathcal{P}}
\newcommand{\mL}{\mathcal{L}}
\newcommand{\mI}{\mathcal{I}}

\newcommand{\minimize}{\operatornamewithlimits{minimize}}

\usepackage[normalem]{ulem}
 
\newcommand{\replace}[2]{{{\color{black}#2\color{black}}}} 
\newcommand{\replacemath}[2]{{{\color{black}#2\color{black}}}} 
\newcommand{\jcolor}{black}

\newcommand{\rank}{\textrm{rank}}
\newcommand{\beq}{\begin{equation}}
\newcommand{\eeq}{\end{equation}}
\newcommand{\bdm}{\begin{displaymath}}
\newcommand{\edm}{\end{displaymath}}
\newcommand{\ba}{\begin{aligned}}
\newcommand{\ea}{\end{aligned}}
\newtheorem{theorem}{Theorem}
\theoremstyle{definition}
\newtheorem{definition}{Definition}
\newtheorem{proposition}{Proposition}
\newtheorem{remark}{Remark}
\newtheorem{lemma}{Lemma}

\newtheorem{observation}{Observation}[section]

\DeclarePairedDelimiterX\Set[2]{\lbrace}{\rbrace}%
 { #1 \colon\mathopen{} #2 }

\usepackage{enumitem}

\graphicspath{{../../figures/platoon_illustration/},{../../figures/ITHM_scheme/},{../../figures/maps/},{../../figures/RSHM_iters_plot/},{../../figures/fuel_vs_idleTime/},{../../figures/act-constr-alg-demo/}}
\usepackage{endnotes}
\let\footnote=\endnote

\usepackage{natbib}
 \bibpunct[, ]{(}{)}{,}{a}{}{,}%
 \def\bibfont{\small}%
\providecommand{\KEYWORDS}[1]
{
  \small	
  \textbf{\textit{Keywords---}} #1
}
\begin{document}




\title{A Repeated Route-then-Schedule Approach to Coordinated Vehicle Platooning:
Algorithms, Valid Inequalities and Computation}

\author{Fengqiao Luo\footnote{Northwestern University, Dept. of Industrial Engineering and Management Science, \texttt{fengqiaoluo2014@u.northwestern.edu}}
~and Jeffrey Larson\footnote{Argonne National Laboratory, Mathematics and Computer Science Division, \texttt{jmlarson@anl.gov}}
}

\date{} 
\maketitle

\begin{abstract}
Platooning of vehicles is a promising approach for reducing fuel consumption,
increasing vehicle safety, and using road space more efficiently. 
\replace{The}{We consider the important but} difficult problem of assigning optimal routes and departure schedules to a
collection of vehicles\replace{is therefore important}{}. 
We propose an iterative route-then-schedule \replace{approach}{heuristic} for centralized
planning that quickly converges to high-quality solutions.
We also propose and analyze a collection of valid inequalities for the individual
problems of assigning vehicles to routes and scheduling the times that vehicles
traverse their routes. These inequalities are shown to reduce the computational
time or optimality gap of solving the routing and scheduling problem instances.
Our approach uses
the valid inequalities in both the routing and scheduling portions of each
iteration; numerical experiments highlight the speed of the approach for routing
vehicles on a real-world road network.
\end{abstract}%




%


\section{Introduction}
Coordinated platooning of vehicles is a proposed fuel-saving approach where
vehicles travel in a single file
with short intervehicle distances. Control of the vehicles in the platoon may be
facilitated by on-board sensors or vehicle-to-vehicle communications
\citep{Bergenhem2012,tsugawa2013overview,Serizawa2019,Dey2016}. 
%
Researchers have shown platoons to more efficiently use road space, improve
network throughput, decrease the risk of crashes, and reduce greenhouse gas
emissions~\citep{eritco2016}. Because platooning vehicles have reduced aerodynamic drag,
such environmental benefits directly correspond to reduced fuel use and
ultimately financial savings to vehicle owners. 
In this manuscript we focus on platooning as a fuel-saving (and therefore
cost-saving) technology. Such savings are not trivial.
ERITCO reports a 7--16\% reduction in \ce{CO2} emissions for trailing vehicles,
as well as a 1--8\% reduction for the vehicle leading the
platoon~\citep{eritco2016}. 
\replace{}{The leading vehicle's savings arise from a reduction in the
low-pressure air bubble behind it; a trailing vehicle's savings arise from a
decrease in the high-pressure air bubble in front of it.} 
See \cite{Tsugawa2016} for a review of platooning technologies, observed energy
saving values, and potential issues that have been reported by notable platooning projects. 

While much of the recent literature has been devoted to the control of
individual platoons \citep{Zhang2019,Gong2016,Ye2019}, the technology has also
inspired research in many \replace{}{other} areas. Studies have been devoted to increasing the acceptance of
platooning among professional drivers~\citep{Castritius2020}, using platoons
optimally when unloading container ships~\citep{You2020}, and controlling
platoons to improve pavement longevity~\citep{Gungor2020}.
Many works in the recent literature seek to maximize the benefits of platooning
by adjusting vehicle operations. These works include adjustments to the times
vehicles traverse their routes \citep{Boysen2018},
the routes that vehicles take \citep{Baskar2013},
and the speed at which vehicles travel
\citep{Hoef2015:itsc:15,Liang2016,luo2018-veh-platn-mult-speeds}.
Such coordination can be in a centralized framework where a single coordinator
decides vehicle operations (e.g., \cite{Abdolmaleki2019}) or in a distributed
approach where platooning opportunity information is conveyed to vehicles
already traveling in the network (e.g., \cite{Larson2014d,Sokolov2017}).

\replace{}{This manuscript also considers the problem of routing vehicles to
facilitate platooning.}
Because \replace{the general platoon routing problem}{determining a fuel-optimal
routing of platoonable vehicles} has been shown to be NP-hard
\citep{Larsson2015}, various approaches for solving these optimization
problems have been considered: evolutionary methods
\citep{Nourmohammadzadeh2016}, mixed-integer programming techniques
\citep{2016-coord-platn-routing}, and network-flow formulations
\citep{Abdolmaleki2019}. Nearly all of the literature considers a few dozen vehicles
to be the limit of their method. Notable exceptions are the works of
\citet{luo2018-veh-platn-mult-speeds} and \citet{Abdolmaleki2019}, which 
present approaches that are capable of routing 
thousands of vehicles. 

The recent work of \cite{Abdolmaleki2019} is especially related to our work. They
consider a time-discretized approximation and assume that more than two
vehicles that enter an edge within a given time segment will platoon.
A time-expanded network is constructed under
this time-discretized approximation. The coordinated vehicle platooning problem (CVPP) is formulated as a concave minimization problem
with mixed-integer variables and is reformulated into a convex mixed-integer
nonlinear program after introducing an artificial strictly convex term in the objective.
They identify high-quality solutions of this nonlinear program using an outer-approximation technique and local search.
In contrast to their work, which expands the
number of vertices using the time discretization, our work treats
time as a continuous variable.
We focus on formulating a tighter mixed-integer linear program (MILP) (but not nonlinear reformulations)
by decoupling the routing and scheduling subproblems.
We note that
the two approaches are not exclusive to each other.
The time-discretization approximation from \cite{Abdolmaleki2019} can also be applied to 
reformulate the scheduling subproblem proposed in this paper.

\replace{}{Many investigations into vehicle-routing 
problems have been undertaken using time-extend networks and time-discretized approximations
to handle space-time coordination.
\cite{Boland2019} investigates how the time-interval size
affects model complexity and approximation accuracy. 
 In some cases, the approximation gap between the time-discretized and the time-continuous models
 can be greater than 20\%.
 \cite{Fischer2012} establishes a general dynamic time-extend graph generation framework
 to control the complexity of the time-extended network
 without sacrificing the quality of approximation when the time horizon is
 long or the time discretization is increasingly refined. 
 This approach can be applied to many routing
 problems in which shorter mission-completion time is preferred. 
 In some situations, 
 carefully designed methods keep time as a continuous parameter
 without sacrificing computational tractability.
 \cite{Boland2017} design a service network
that consolidates carriers and transport shipments where
the paths of the carriers need to be coordinated in both space and time.
They develop an iterative refinement algorithm using a partially time-expanded
network for solving this problem in a continuous-time situation,
rather than applying a time-discretization approach. 
Other approaches, such as \cite{Contardo2015} and \cite{Pecin2017}, solve 
time-dependent vehicle routing problems without explicitly constructing a
time-extended network; instead, time windows are used to select the feasible
portions of the network.
\cite{Skutella2009} and \cite{Gendreau2015} provide overviews of network
optimization problems in time-dependent networks. 
}

This work assumes travel times are known deterministically and that a
centralized controller determines vehicle routes and departure times. 
(Platooning coordination in the deterministic case is already difficult;
accounting for uncertainty in travel times is even more challenging 
\citep{Zhang2017,Li2017,Li2017a}.) Note that we do not allow vehicles to stop
traveling once they have started traveling. See the work of
\citet{Bhoopalam2018} for a thorough review of approaches for optimal platoon
routing approaches and their assumptions.

The research reported in this paper benefits from existing techniques of identifying 
valid and facet-defining inequalities for strengthening a mixed-integer linear program.
These techniques include 
constructing disjunctive valid inequalities to cut fractional solutions
\citep{1998-balas_disj-prog-conv-hull-feasb-pt,Jeroslow1977,
letchford2001_disj-cuts-for-combinatorial-opt,Luedtke2017_lift-and-proj-cuts-for-convex-minlp},
identifying valid inequalities for clique-partition polytopes 
\citep{grotschel1990_clique-partition-polytope,wolsey1996_form-valid-ineq-cap-graph-partition,
oosten2001_clique-partition-facets,bandelt1999_lift-facet-character-clique-partition}, 
using lift-and-projection methods to construct the convex hull of a mixed binary set
\citep{au2016_analysis-poly-lift-proj-methd,
kocuk2016_cycle-base-form-valid-ineq-DC-power-transmission,
Luedtke2017_lift-and-proj-cuts-for-convex-minlp,burer2005_solve-lift-proj-relax,
aguilera2004_lift-proj-relax-matching-related,
balas1997_mod-lift-proj-proced},
projecting out undesired variables by using Fourier-Motzkin elimination technique
\citep{dantzig1972_FM-elimination-and-dual,williams1976_FM-elim-int-prog,
schechter1998_integration-polyhedron-appl-of-EM-elim},
and using lifting techniques to extend facet-defining inequalities valid only 
for lower-dimensional polytopes to be valid for higher-dimensional polytopes
\citep{2007wolsey_lift-superadd-single-node-flow,nemhauser2003_lifted-ineq-mip-thy-alg,atamturk2003_facets-of-mixed-int-knapsack-poly,
gu1999_liftted-cover-ineq-complexity,gu1999_liftted-flow-cover-ineq,
gu1998_liftted-cover-ineq-computing}.
Furthermore, our proofs utilize a variety of techniques for verifying that
a certain valid inequality is also facet-defining, which include the
construction and counting of affinely independent points
\citep{conforti-IP-2014} and coefficient determination techniques that are
facilitated by substituting feasible points
\citep{grotschel1990_clique-partition-polytope}. Valid and facet-defining
inequalities are widely used in strengthening the mixed-integer
formulations of combinatorial optimization problems in operations research
\citep{pochet1993_lot-sizing-form-valid-ineq,pochet1995_cap-fac-loc-valid-ineq,perboli2010_new-valid-ineq-two-echelon-veh-routing,coelho2014_inventory-routing-with-valid-ineq,
baldacci2009_valid-ineq-fleet-size-mix-veh-routing,
huygens2006_two-edge-hop-constr-netwk-design}.
Valid inequalities derived for a particular polyhedral structure can be applied to a variety of
problems that have that structure in their formulations. For instance, flow-cover inequalities
can be applied to network design and resource management problems \citep{gu1999_liftted-flow-cover-ineq},
cover inequalities can be applied to problems with knapsack-type constraints 
\citep{gu1999_liftted-cover-ineq-complexity},
and path-cover-path-pack inequalities \citep{atamturk2017_path-cover-pack-ineq-cap-netwk-flow} 
can be applied to capacitated network flow problems
arising from telecommunication, facility location, production planning, and supply chain management.
In a recent trend of research on mixed-integer programming (MIP) theory,
valid inequalities are derived for a variety of problems arising from optimization
under uncertainty \citep{wang2019_dr-chance-constr-assignment,gade2012_decomp-alg-param-gomory-cuts-TSSIP,kucukyavuz2012_mix-sets-chance-constr-prog,koster2011_rob-netwk-des-valid-ineq}.  
Valid inequalities are powerful in reducing 
the computational time or optimality gap in many applications.
However, there should be a balance between strengthening the formulation
and computational cost, since numerical construction of valid inequalities incurs 
additional computational time, and adding too many valid inequalities
can slow down solving node-relaxation problems 
\citep{botton2013_benders-decomp-for-hop-constrained-surv-netwk-des}.

\subsection{\textcolor{\jcolor}{Description of the Coordinated Vehicle Platooning Problem}}
\replace{}{
In the CVPP, we are given a road network and a set of vehicles, each with an
origin node and destination node, as well as origin/departure times for each vehicle.
We seek a time-feasible path for each vehicle from its origin to its destination. Fuel
saving will be incurred when vehicles form a platoon, that is, when they
traverse a shared edge in their paths at the same time. The benefits of
platooning should be incurred only if they outweigh the cost of having each
vehicle take its assigned route. The CVPP objective is to minimize the 
collective total fuel consumed by the vehicles.}

\subsection{Contributions and Organization of the Paper}
We propose a repeated route-then-schedule heuristic method (RSHM) to 
solve the CVPP. Our approach is motivated \replace{}{by}
the fact that vehicle routes that can enable platooning
and fuel saving are limited in practice: such routes must be ``close'' to the shortest path. 
This fact can help reduce the size of MIP models
and also indicates that it may be computationally beneficial 
to decompose the routing and scheduling procedures within a novel route-then-schedule
approach where a route is assigned to each vehicle
and then a scheduling problem is solved for the given routes. 
In the framework we present here, the route-then-schedule method is 
repeated to regenerate routes that are more promising for saving fuel
after time constraints are imposed. 
\replace{}{Structural properties of the CVPP subproblems 
are also identified.} 
 
We develop valid inequalities that can strengthen the MILP formulations of the routing and scheduling 
problems that are decoupled in the RSHM scheme. 
The polyhedral properties of the routing and scheduling formulations 
identified in this work are of independent interest. We perform systematic and intensive numerical
experiments in order to understand the performance of the RSHM scheme, 
the advantage of using this scheme over former MILP formulations 
\citep{2016-coord-platn-routing,luo2018-veh-platn-mult-speeds} 
of the CVPP that combine the routing and scheduling, and the merits of adding the 
developed valid inequalities. We are unaware of any other work that
investigates valid inequalities for the routing or scheduling subproblems
associated with the CVPP.

Section~\ref{sec:prob-description} defines the CVPP.
Sections~\ref{sec:routing-problem} and \ref{sec:form-scheduling} decouple the
CVPP into routing and scheduling problems, respectively.
Section~\ref{sec:edge-contraction} discusses a procedure that can reduce
the size of the scheduling problem. Section~\ref{sec:two-stage-heuristc}
presents the RSHM scheme that is based on repeatedly solving the routing and
scheduling subproblems and updating the presumed fuel cost. This culminates in
Theorem~\ref{thm:RSHM-alg-converge}, which highlights properties of the RSHM
method. Section~\ref{sec:valid-ineq-route} and
Section~\ref{sec:valid-ineq-schedule} develop valid inequalities to strengthen
the routing and scheduling formulations, respectively.
Section~\ref{sec:num-study} presents the numerical investigation of the
algorithms and valid inequalities developed in this paper. Proofs and
additional numerical results are given in the electronic supplement.

\section{Problem Formulation and Solution Methodology}
\label{sec:form-sol-method}
\replace{}{Section~\ref{sec:prob-description} gives a mathematical description
of the CVPP, which consists of a collection of vehicle routing and scheduling
decisions. Simultaneously addressing both types of decisions can
lead to highly intractable problems with loose relaxations because of a large number of
big-M coefficients~\citep{luo2018-veh-platn-mult-speeds}. We propose 
a route-then-schedule approach that 
decomposes the CVPP into a routing subproblem and a scheduling subproblem
and repeatedly solves the two subproblems with adjusted input parameters.
Section~\ref{sec:prob-form} formulates the routing and scheduling subproblems,
and Section~\ref{sec:two-stage-heuristc} outlines how our
method integrates the two subproblems.
}

\subsection{A Formal Problem Description}\label{sec:prob-description}
We first describe the CVPP.
Let $\cG(\cN,\cE)$ be a graph representing the highway network of an area, 
where $\cN$ is the set of nodes and $\cE$ is the set of highway segments (or edges).
Let $C_{i,j}$ denote the fuel cost of traversing edge $(i,j)\in \cE$.
Let $\cV$ be the set of vehicles to be routed, each with origin node, $O_v \in
\cN$; destination node, $D_v \in \cN$; earliest time it can \replace{departure}{depart} from
$O_v$, $T_v^O \in \mathbb{R}_+ $; and latest time it can arrive at $D_v$, $T_v^D \in \mathbb{R}_+$. 
\replace{}{We assume that $v$ travels continuously from $O_v$ and $D_v$ without 
making intermediate stops (e.g., at a parking or rest area)
to wait for other vehicles to coordinate platoon formation.}
Nevertheless, if multiple vehicles arrive at some node $i$ simultaneously, they may be
able to platoon along some edge $(i,j) \in \cE$. 
When vehicles form a platoon, the fuel-saving rates for the lead vehicle and 
each trailing (following) vehicle are $\sigma^l$ and $\sigma^f$, respectively, where $\sigma^l<\sigma^f$.   
Ultimately, we seek a set of optimal routes and departure times for all
vehicles that minimizes the total fuel consumed while completing all transport
missions. We assume that each vehicle drives at free-flow speed along 
each edge (allowing for this speed to be edge dependent).
\replace{}{In practice, 
vehicles can make slight adjustments to their speed in order to
facilitate platoon formation. We do not consider such intervehicle coordination but
focus exclusively on the high-level strategy of routing and scheduling.}

\replace{}{Note that modeling the highway network as a directed graph means
edges $(i,j)$ and $(j,i)$ may correspond to two directions of the same 
road segment. In general, $\cE$ can correspond to a connected network or a
collection of disconnected components. In the latter case, the CVPP
can be decomposed into independent smaller problems
corresponding to each disconnected component. The formulations below are valid
for either case. Although the road network is not assumed to be sparse in what
follows, most road networks correspond to relatively sparse planar graphs.}

\subsection{Problem Formulation}\label{sec:prob-form}
Every iteration of the RSHM consists of two phases (or subproblems).
In phase one, we solve a route-design problem (RDP) 
to identify optimal routes for all vehicles assuming that there
are no time constraints on vehicles' departure and arrival times. 
In phase two, we solve a scheduling problem (SP) with time constraints
assuming that vehicles travel on their RDP routes. The SP solution is used to
update the costs $C_{i,j}$ in the next RDP. The solution of the 
routing and scheduling problems is repeated until a termination criterion is reached.

We note that the RSHM's first RDP problem is identical
to the unlimited vehicle platooning (UVP) problem investigated by \cite{Larsson2015}.
In the UVP problem, it is assumed that all vehicles that are assigned to travel through 
a common highway segment (represented by an edge in the transportation network) 
can form a single platoon on this edge. Based on this assumption, the UVP problem
searches for an optimal route assignment to all vehicles to minimize the total fuel cost. 
The fuel cost adjustment can result in a different objective in subsequent
iterations, but the constraints of the RDP problem are unchanged.

\subsubsection{Formulation of the Route-Design Problem.}\label{sec:routing-problem}
\citet{Larsson2015} prove that the RDP on general 
and on planar graphs is NP hard. 
They provide a mixed 0-1 formulation that is applicable to the RDP, but their
formulation
involves more logical constraints than necessary (e.g.,~\citet[Definition~9]{Larsson2015}) 
to determine whether two vehicles are
platooned when traveling on an edge. 
We propose a new formulation that uses
vehicle-number indicator variables to note when 
only one vehicle or more than two vehicles traverse on each edge. 
\replace{This RDP formulation uses the following notation.}{}
\begin{table}[t]
  \footnotesize
  \caption{Sets, parameters, and variables for the RDP formulation. \label{tab:RDP}}
  \begin{tabularx}{\textwidth}{lX}
    \toprule
    Set & Definition\\
    \midrule
			$\cN$ & set of nodes\\
			$\cE$ & set of road edges, $(i,j)$, for $i,j \in \cN$\\
			$\cV$ & set of vehicles \\
    \midrule\\[-.35cm]
    Parameter & Definition\\
    \midrule
			$C_{i,j}$ & fuel cost of traversing edge $(i,j)\in \cE$ for each vehicle\\
			$\sigma^l$, $\sigma^f$ & fuel saving rate for the lead and following vehicle in a platoon\\
			$O_v, D_v$ & origin and destination nodes for vehicle $v\in\cV$  \\
			$T_v^O$ & earliest allowable departure time of vehicle $v$ from its origin\\
			$T_v^D$ & latest allowed arrival time of vehicle $v$ at its destination\\
      $T_{i,j}$ & time cost of traversing edge $(i,j)\in \cE$ for all vehicles\\
    \midrule\\[-.35cm]
    Variable & Definition\\
    \midrule
			$x_{v,i,j}$ & 1 if vehicle $v$ traverses edge $(i,j)\in \cE$, 0 otherwise \\
			$y_{i,j}$ & 1 if any vehicle traverses edge $(i,j)\in \cE$, 0 otherwise\\
			$y'_{i,j}$ & 1 if two or more vehicles traverse edge $(i,j)\in \cE$, 0 otherwise\\
			$w_{i,j}$ & number of extra vehicles taking $(i,j)$, 
			if more than one vehicle takes $(i,j)$. That is, \linebreak[4]$w_{i,j}=\max\{0,m-1\}$, 
			where $m$ is the number of vehicles assigned to traverse $(i,j)$\\
    \bottomrule
  \end{tabularx}
\end{table}
	
The phase-one RDP \replace{}{uses the notation in Table~\ref{tab:RDP} and} is formulated as follows:
\begin{subequations}
  \makeatletter
  \def\@currentlabel{RDP}
  \makeatother
  \label{opt:RDP}
  \renewcommand{\theequation}{RDP.\arabic{equation}}
  \begin{alignat}{3}
    \minimize_{x,y,y',w}     &\quad \sum_{v\in\cV}\sum_{(i,j)\in \cE} C_{i,j}x_{v,i,j} - \sum_{(i,j)\in \cE} \sigma^l C_{i,j} y'_{i,j} - \sum_{(i,j)\in \cE} \sigma^f C_{i,j} w_{i,j} \label{eqn:RDP_obj}\\
    \text{subject to:}  &  \sum_{k\in \cN: (j,k)\in \cE} x_{v,j,k} - \sum_{i\in \cN: (i,j)\in \cE} x_{v,i,j}  =
    \begin{cases}
        1  &  \textrm{if } j= O_v \\
        -1 &  \textrm{if } j=D_v \\
        0  &  \textrm{otherwise } \\
      \end{cases}                         & \forall v\in\cV,\;\forall j\in\cN,  \label{eqn:RDP_1} \\
  &\sum_{(i,j)\in\cE}T_{i,j}x_{v,i,j}\le T^D_v-T^O_v & \forall v\in\cV, \label{eqn:RDP_1extra} \\
  &\sum_{v\in\cV} x_{v,i,j} \ge 2 y'_{i,j}       & \forall (i,j)\in \cE,  \label{eqn:RDP_2} \\
  &w_{i,j} - \sum_{v\in\cV} x_{v,i,j} + y_{i,j} \le 0  & \forall (i,j)\in \cE,  \label{eqn:RDP_3}\\
  &x_{v,i,j} \le y_{i,j}             & \forall v\in\cV,\; \forall (i,j)\in \cE,  \label{eqn:RDP_4}\\
  & y'_{i,j}\le y_{i,j}           & \forall (i,j)\in \cE, \label{eqn:RDP_5} \\
  &x_{v,i,j}, y_{i,j},  y'_{i,j}\in\{0,1\}, w_{i,j}\ge 0    & \forall  v\in\cV,\; \forall (i,j)\in \cE. \label{eqn:RDP_6}
  \end{alignat}
\end{subequations}
The objective \ref{eqn:RDP_obj} contains three terms. The first term
represents the cost of fuel used to traverse edges, without any savings from
platooning. The second and third terms represent the cost of fuel saved
from lead and trailing vehicles, respectively.
Constraint~\ref{eqn:RDP_1} ensures that vehicles 
travel from $O_v$ to $D_v$.
Constraint \ref{eqn:RDP_1extra} ensures that the route is feasible for the destination time
requirement. (Time constraints to ensure that formation of platoons
is feasible appear
in the \ref{opt:SPF}.)
Constraints~\ref{eqn:RDP_2}--\ref{eqn:RDP_4} ensure that $y'_{i,j}$,
$y_{i,j}$, and $w_{i,j}$ represent their respective quantities.
Although $w_{i,j}$ is specified only as a nonnegative real variable, the integral values
of $x_{v,i,j}$ and $y_{i,j}$ with constraints \ref{eqn:RDP_3} \replace{}{and
objective \ref{eqn:RDP_obj}} ensure that $w_{i,j}$ takes only integral values
\replace{}{at an optimal solution}.

\subsubsection{Formulation for Scheduling.}
\label{sec:form-scheduling}
After solving the \ref{opt:RDP} to obtain routes for all vehicles, 
we solve a scheduling problem in phase two to determine vehicle departure times.
\replace{}{The decisions on the formation of platoons will be determined along with vehicle schedules
when solving the scheduling problem. Since we assume the order of vehicles in a platoon
does not affect its fuel savings, we can label all vehicles using positive integers 
and simplify our model by assuming that vehicles in a platoon follow a canonical order with
vehicles platooning in order of their labels. 
}
The scheduling problem uses the following notation.
\begin{table}[t]
  \footnotesize
  \caption{Sets, parameters, and variables for the SP formulation. \label{tab:SP}}
  \begin{tabularx}{\textwidth}{lX}
    \toprule
    Set & Definition\\
    \midrule
			$\cR_v$ & route of vehicle $v\in\cV$ (an ordered subset of edges in the route)\\
			$\cN_v$ & set of nodes on $\cR_v$, $v\in\cV$\\
      $\cV_{i,j}$ & set of vehicles taking the edge $(i,j) \in \cup_{v\in\cV}\cR_v$ \\
    \midrule\\[-.35cm]
    Parameter & Definition\\
    \midrule
    $T_v^O,T_v^D,T_{i,j}$ & defined in Table~\ref{tab:RDP}\\
			$\lambda$ & maximum number of vehicles permitted in a platoon  \\
    \midrule\\[-.35cm]
    Variable & Definition\\
			$t_{v,O_v}$ & departure time of vehicle $v$ from node $O_v$ \\
			$t_{v,i}$ & arrival time of vehicle $v$ at node $i$ such that there exists a $j$ such that $(j,i) \in \cR_v$ \\
      $f_{u,v,i,j}$ & 1 if vehicle $u$ follows vehicle $v$ at edge $(i,j)\in\cR_u\cap\cR_v$, 0 otherwise. \replace{}{(We require $u>v$ by the rule of canonical ordering.)} \\
			$\ell_{v,i,j}$ & 1 if vehicle $v$ leads a platoon at edge $(i,j)\in\cR_v$, 0 otherwise\\
    \bottomrule
  \end{tabularx}
\end{table}
\replace{The scheduling problem uses the following notation.}{}
The second-stage model to assign schedules for each vehicle \replace{}{uses the notation in Table~\ref{tab:SP} and} can be formulated as
\begin{subequations}
  \makeatletter
  \def\@currentlabel{SP}
  \makeatother
  \label{opt:SPF}
  \renewcommand{\theequation}{SP.\arabic{equation}}
  \begin{alignat}{3}
    \text{maximize }     &\sum_{(i,j)\in\cup_v\cR_v} \Big(  \sum_{v\in\cV_{i,j}} \sigma^l C_{i,j} \ell_{v,i,j} + \sum_{u,v\in\cV_{i,j} u>v} \sigma^f C_{i,j} f_{u,v,i,j} \Big) \hspace{114pt}  \label{eqn:SPF_1}
  \end{alignat}
  \vspace{-25pt}
  \begin{alignat}{3}
    \text{subject to: }  &  t_{v,O_v} \ge T_v^O    & \forall v\in\cV,  \label{eqn:SPF_2} \\
& t_{v,D_v} \le T_v^D  & \forall  v\in\cV,   \label{eqn:SPF_3}  \\
& t_{v,j} = t_{v,i} + T_{i,j} & \forall v\in\cV,\; \forall (i,j)\in\cR_v, \label{eqn:SPF_4} \\
&  t_{u,i} - t_{v,i} \le M_{u,v,i,j}(1-f_{u,v,i,j}) & \forall (i,j)\in\bigcup_{v\in\cV} \cR_v, \; \forall u > v\in\cV_{i,j},  \label{eqn:SPF_5}  \\
&  t_{u,i} - t_{v,i} \ge -M_{u,v,i,j}(1-f_{u,v,i,j}) & \hspace{1in} \forall (i,j)\in\bigcup_{v\in\cV} \cR_v, \; \forall u > v\in\cV_{i,j},  \label{eqn:SPF_6}  \\
&\sum_{w\in\cV_{i,j}: w<v} f_{v,w,i,j} \le 1-\ell_{v,i,j} & \forall v\in \cV,\; (i,j)\in \cR_v,  \label{eqn:SPF_7}  \\
&\sum_{u\in\cV_{i,j}: u>v} f_{u,v,i,j} \le (\lambda-1)\ell_{v,i,j} & \forall v\in \cV,\; (i,j)\in \cR_v,  \label{eqn:SPF_8}  \\
&\sum_{u\in\cV_{i,j}: u>v} f_{u,v,i,j}\ge \ell_{v,i,j}  & \forall v\in \cV,\; (i,j)\in \cR_v,   \label{eqn:SPF_9} \\
&f_{u,v,i,j}\in\{0,1\} & \forall u,v\in\cV, u>v, \forall (i,j)\in\cR_v,  \label{eqn:SPF_11}  \\
&\ell_{v,i,j}\in\{0,1\} & \forall v\in\cV, \forall (i,j)\in\cR_v,  \label{eqn:SPF_12} 
  \end{alignat}
\end{subequations}
where $M_{u,v,i,j}$ in constraint \ref{eqn:SPF_7} is defined for 
all $(i,j)\in\cup_v\cR_v$ and $u,v\in\cV_{i,j}$ with $u>v$, 
such that $-M_{u,v,i,j}\le t_{u,i}-t_{v,i}\le M_{u,v,i,j}$ 
holds for any feasible values of $t_{u,i}$ and $t_{v,i}$.

The objective \ref{eqn:SPF_1} consists of two terms corresponding to 
fuel saving of lead and trailing vehicles on each edge, respectively.  
Note that maximizing the total fuel saving is equivalent to minimizing the total fuel cost. 
Constraints \ref{eqn:SPF_2}--\ref{eqn:SPF_4} ensure that vehicle travel times are accurate.
Constraints \ref{eqn:SPF_5}--\ref{eqn:SPF_6} ensure that if vehicle $u$ 
follows vehicle $v$ in a platoon on edge $(i,j)$, they must traverse the edge simultaneously.
Constraint~\ref{eqn:SPF_7} ensures that if vehicle $v$ is a lead vehicle 
on an edge, it cannot follow any other vehicles.
If $v$ is not a lead vehicle,
it can follow no more than one vehicle (the lead vehicle) on an edge. 
(Note that our formulation does not allow any trailing vehicle to lead another
vehicle.)
Constraint~\ref{eqn:SPF_8} enforces that
a vehicle can lead at most $\lambda-1$ other vehicles. 
Constraint~\ref{eqn:SPF_9} states that if $v$ is a lead vehicle
on an edge, the number of vehicles following $v$ on that edge must be at least one.

In this problem,
one can remove variables $t_{v,i}$ other than for $i =
O_v$ because the vehicle routes are fixed and there is no
waiting at any node other than the starting node. Therefore, once a vehicle leaves
$O_v$, the time to travel from $O_v$ to $D_v$ is $\sum_{(i,j) \in \cR_v} T_{i,j}$.
To make constraint declaration cleaner, however,
we use $t_{v,i}$ for intermediate nodes $i$.

\subsubsection{Reducing the Scheduling Problem Size by Edge Contraction.}
\label{sec:edge-contraction}
\replace{}{Consider the optimal schedule obtained from \eqref{opt:SPF} based on any route assignment.
The following proposition states that in an optimal solution, 
if two vehicles are platooned on a shared 
edge $e$, they will be platooned on all shared edges preceding and following
$e$.} 
\begin{proposition}
\label{prop:edge-contraction}
\replace{}{Let $\cup_v\cR_v$ be any route assignment.
There exist an optimal schedule and platooning decisions
from \eqref{opt:SPF} based on $\cup_v\cR_v$
such that for any path $L:\;n_1\to n_2\to\ldots\to n_k$
satisfying $\cV_{n_i,n_{i+1}}=\cV_{n_j,n_{j+1}}$ for all $i,j\in\{1,\ldots,k-1\}$,
the platooning decisions are the same on every edge of $L$,
where $\cV_{n_i,n_{i+1}}$ is the subset of vehicles sharing the
edge $(n_i,n_{i+1})$ induced by $\cup_v\cR_v$.}
\end{proposition}  
Proposition~\ref{prop:edge-contraction} is proven in Section~\ref{sec:edge_contraction_proof} and
implies that the \ref{opt:SPF} problem size can be reduced by preprocessing \ref{opt:RDP} routes.
This procedure (Algorithm~\ref{alg:edge-contraction}) contracts consecutive
edges into a single edge if the sets of vehicles sharing these edges are identical.
Table~\ref{tab:edge-contraction} in Section~\ref{app:add-num-res}
shows how this procedure can reduce the number of variables and constraints and problem loading time.
\begin{algorithm}[t]
{\footnotesize
\caption{Edge-contraction preprocess applied to the \ref{opt:RDP} routes.}\label{alg:edge-contraction}
\begin{algorithmic}[1]
\State{\bf Input}: A set of routes $\Set*{\cR_v}{v\in\cV}$, 
and the sets of time costs $\Set*{T_{i,j}}{(i,j)\in\cup_{v\in\cV}\cR_v}$,
fuel costs $\Set*{C_{i,j}}{(i,j)\in\cup_{v\in\cV}\cR_v}$, 
and the vehicles $\Set*{\cV_{i,j}}{(i,j)\in\cup_{v\in\cV}\cR_v}$ on each edge
of the input routes.
\State{\bf Output}: A set of routes $\Set*{\cR^{\prime}_v}{v\in\cV}$,
the set $\Set*{T^{\prime}_{i,j}}{(i,j)\in\cup_{v\in\cV}\cR_v}$ of time cost,
the set $\Set*{C^{\prime}_{i,j}}{(i,j)\in\cup_{v\in\cV}\cR_v}$ of fuel costs, 
and the set $\Set*{\cV^{\prime}_{i,j}}{(i,j)\in\cup_{v\in\cV}\cR_v}$ of vehicle subsets on each edge
of the new routes. 
\State{\bf Initialization:} Set $\cR^{\prime}_v\gets\cR_v\;\forall v\in\cV$,
 $T^{\prime}_{i,j}\gets T_{i,j}$, $C^{\prime}_{i,j}\gets C_{i,j}$ and 
$\cV^{\prime}_{i,j}\gets\cV_{i,j}$ for all $(i,j)\in\cup_{v\in\cV}\cR^{\prime}_v$.
\While{$\exists$ a vehicle $v\in\cV$, and two consecutive edges 
$(i,j),(j,k)\in\cR^{\prime}_v$ such that $\cV^{\prime}_{i,j}=\cV^{\prime}_{j,k}$.}
	\For{$v^{\prime}\in\cV^{\prime}_{i,j}$}
		\State{Delete two edges $(i,j),(j,k)$ from $\cR^{\prime}_{v^{\prime}}$,
			 and add $(i,k)$ to $\cR^{\prime}_{v^{\prime}}$.}
	\EndFor
	\State{Delete $T^{\prime}_{i,j}$, $T^{\prime}_{j,k}$ and create 
		  $T^{\prime}_{i,k}\gets T^{\prime}_{i,j}+T^{\prime}_{j,k}$ in the set of time costs.}
      \State{Delete $C^{\prime}_{i,j}$, $C^{\prime}_{j,k}$ and create 
		  $C^{\prime}_{i,k}\gets C^{\prime}_{i,j}+C^{\prime}_{j,k}$ in the set of fuel costs.}
	\State{Delete the sets $\cV^{\prime}_{i,j}$, $\cV^{\prime}_{j,k}$ and create 
	           a set $\cV^{\prime}_{i,k}\gets\cV^{\prime}_{i,j}$.}
\EndWhile
\State{\Return{$\Set*{\cR_v}{v\in\cV}$, $\Set*{T^{\prime}_{i,j},C^{\prime}_{i,j},\cV^{\prime}_{i,j}}{(i,j)\in\cup_{v\in\cV}\cR^{\prime}_v}$}}
\end{algorithmic}
}
\end{algorithm}

\subsection{A Repeated Route-Then-Schedule Heuristic Using a Feedback Mechanism}
\label{sec:two-stage-heuristc}
The \ref{opt:RDP} disregards the time constraints and assumes
that all vehicles that share the same edge on their routes can form a single platoon
when traversing that edge. The optimal value of the \ref{opt:RDP} is therefore a lower bound
on the optimal fuel consumption.
When the \ref{opt:RDP} routes are fixed for the scheduling problem,
the optimal solution of the \ref{opt:SPF} may indicate that, because of the time constraints, 
only a subset of the vehicles sharing an edge can be platooned. 
Therefore, multiple platoons may need to be formed on some edges. 
Based on the \ref{opt:RDP} routes,
if the platooning configuration generated from the \ref{opt:SPF} 
is much worse than the presumed case (without time constraints),
the \ref{opt:RDP} routes are not likely 
to be the optimal routes for the CVPP.    

Each iteration of the RSHM adjusts the \ref{opt:RDP} objective coefficients
to generate routes with enriched information on the time constraints. 
Specifically, we develop a strategy to adjust the 
objective function of the \ref{opt:RDP} based on 
the optimal solution of the previous scheduling problem.
The RSHM is novel in the way that it establishes a learning process. 
The platooning configuration produced by solving the \ref{opt:SPF}
gives feedback to help the \ref{opt:RDP} learn how to generate
routes that better address the time constraints in the next iteration. 

\subsubsection{RSHM Notation.}
\begin{definition}\label{def:notations}
Below is notation used by the RSHM, Algorithm~\ref{alg:RSHM}.
\begin{itemize} 
  \item \ref{opt:RDP}-$(n)$, \ref{opt:SPF}-$(n)$: routing problem and the scheduling problem
    at iteration $n$ of the RSHM. (The exact definition of these two problems will be specified later)
  \item $x^{(n)}_{v,i,j}$: the $x_{v,i,j}$ component of an optimal solution of \ref{opt:RDP}-$(n)$
  \item $\ell^{(n)}_{v,i,j},f^{(n)}_{v,i,j}$: the $\ell_{v,i,j}$ and $f_{v,i,j}$ component of an optimal solution of \ref{opt:SPF}-$(n)$
  \item $\cR^{(n)}_v$: \emph{route} of vehicle $v$ obtained from solving \ref{opt:RDP}-$(n)$:
        $\cR^{(n)}_v = \left\{ (i,j)\in \cE: x^{(n)}_{v,i,j} = 1 \right\}$
      \item $\cV^{(n)}_{i,j}$: set of vehicles on edge $(i,j)$ in routes $\cup_v\cR^{(n)}_v$:
   $ \cV^{(n)}_{i,j} = \left\{ v\in \cV : x^{(n)}_{v,i,j} = 1 \right\}$
  \item For a set of routes $\cup_v\cR^{(n)}_v$ obtained from \ref{opt:RDP}-$(n)$ and a
    solution of \ref{opt:SPF}-$(n)$, \emph{the platoon $\mP^{(n)}_{v,i,j}$ for vehicle $v$ on an
    edge $(i,j)\in\cup_v \cR^{(n)}_v$} is the set of vehicles containing $v$ and 
    \begin{itemize}
      \item if $\ell^{(n)}_{v,i,j} = 1$, all vehicles $u \in \cV$ such that $f^{(n)}_{u,v,i,j}=1$, or
      \item if $f^{(n)}_{v,v',i,j} = 1$, $v'$ and all vehicles $u \in \cV$ such that $f^{(n)}_{u,v',i,j}=1$.
    \end{itemize}
    A platoon such that $\cP^{(n)}_{v,i,j} = \left\{ v \right\}$ is called a \emph{trivial platoon}.
  \item $\cL^{(n)}_{i,j}$: set of leading vehicles for platoons formed at edge $(i,j)$ at iteration $n$,
  	that is, $\cL^{(n)}_{i,j}=\Set*{v}{v\in\cV^{(n)}_{i,j},\;v\textrm{ is the leading vehicle of }\cP^{(n)}_{v,i,j}}$ 
  \item $z^{(n)}$: total fuel cost obtained at the end of iteration $n$ 
  	   by solving \ref{opt:RDP}-$(n)$ and \ref{opt:SPF}-$(n)$
  \item $C^{(n)}_{v,i,j}$: presumed fuel cost for vehicle $v$ on edge $(i,j)$
			 input into \ref{opt:RDP}-$(n)$, the value of which will be given later 
\end{itemize}
\begin{remark}
Based on this definition, if vehicles $u$ and $v$ are in the same platoon on
an edge $(i,j)$ at iteration $n$, then $\cP^{(n)}_{u,i,j}=\cP^{(n)}_{v,i,j}$.
By convention, a vehicle $v$ is considered to be the leading vehicle in the platoon $\cP^{(n)}_{v,i,j}$
if $\cP^{(n)}_{v,i,j}=\{v\}$.
\end{remark}

\begin{definition}\label{def:extr-idx}
For a given iteration index $n\ge 3$ in Algorithm~\ref{alg:RSHM},
a vehicle $v$, and an edge $(i,j)\in\cup_v\cR^{(k)}_v$,
the \textit{set $\mtc{I}(n,v,i,j)$ of iteration indices with platooning configuration similarity}
is the set of iteration \replace{index}{indices} $k$ satisfying the following conditions:
\begin{displaymath}
  (1)\; 1\le k\le n-2, \quad (2)\; v\in\cV^{(k+1)}_{i,j}, \quad
  (3)\; \Set*{\mP^{(k)}_{u,i,j}}{u\in\cV^{(k)}_{i,j}}=\Set*{\mP^{(n)}_{u,i,j}}{u\in\cV^{(n)}_{i,j}}.
\end{displaymath}
When $\mtc{I}(n,v,i,j)$ is nonempty, we define \textit{the maximal iteration index
with platooning configuration similarity} as the largest index in $\mtc{I}(n,v,i,j)$.
This maximal iteration index is denoted as $I(n,v,i,j)$ in the rest of the paper.
\end{definition}
\end{definition}

\begin{definition}\label{def:platn-fuel-cost}
  Let $\Cp_{i,j}(l)$ denote the total fuel cost of a platoon of size-$l$ ($l\ge 0$) when
traversing an edge $(i,j)\in\cE$. 
\end{definition}
Based on this definition, 
the value $\Cp_{i,j}(l)$ is calculated by using the following formula:
\begin{equation}\label{eqn:C-plat}
  \Cp_{i,j}(l)=\left\{\def\arraystretch{1.5}
		    \begin{array}{ll}
          lC_{i,j} & \textrm{ if } l=0, 1, \\
          (1-\sigma^l)C_{i,j}+(1-\sigma^f)(l-1)C_{i,j} & \textrm{ if } l\ge 2.
		    \end{array}
		    \right.
\end{equation}

\begin{figure}
\centering
\subfigure[Routing problem assuming all vehicles that share 
an edge can platoon (without considering the time constraints).]{ \includegraphics[width=0.47\linewidth,trim=1cm 14.5cm 0.3cm 4cm]{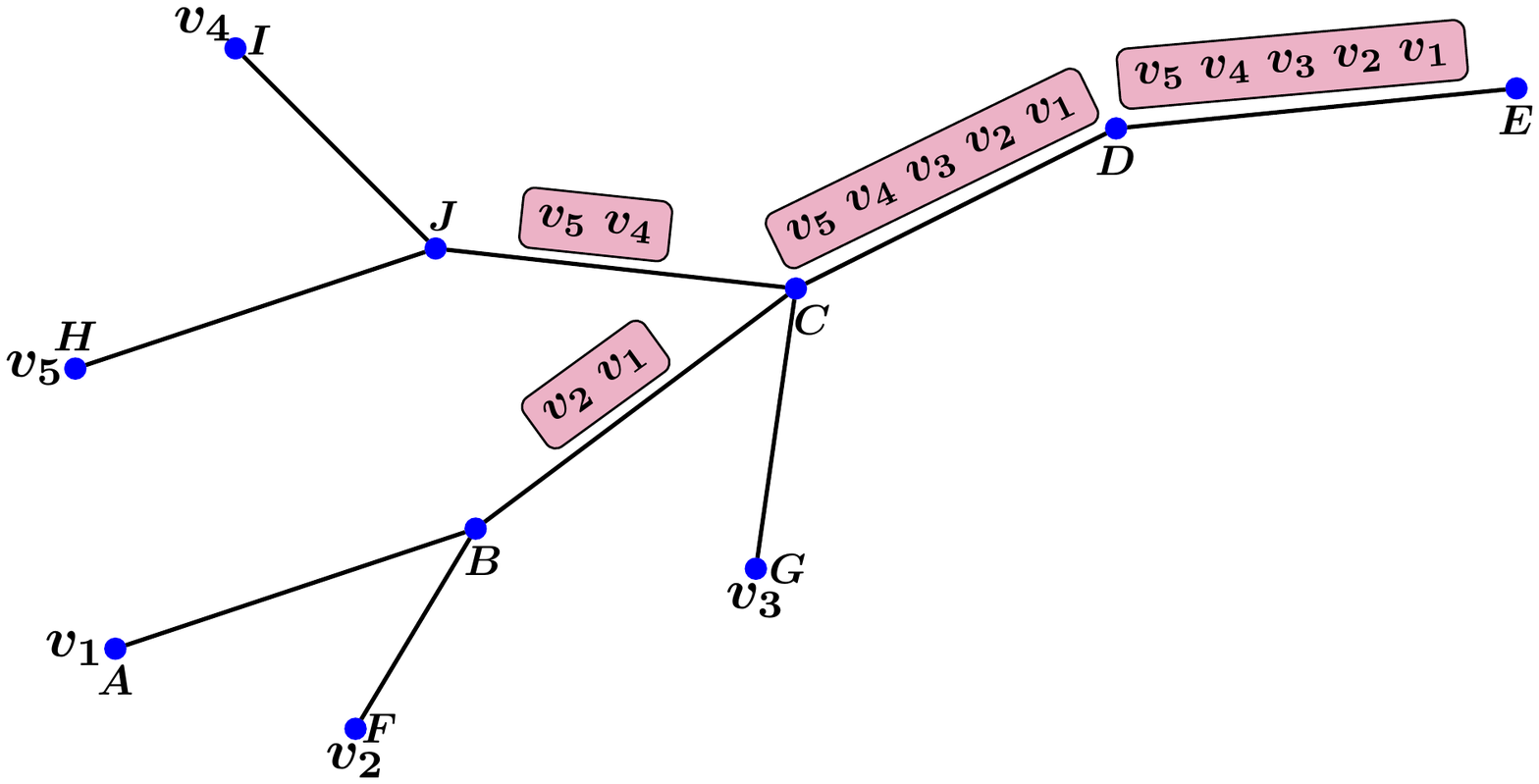} }
\hfill 
\subfigure[After the scheduling problem is solved, 
the vehicles can form only two separate platoons on edges $(C,D)$ and $(D,E)$,
for example, if $v_1,v_2,v_3$ must reach their destinations before vehicles $v_4,v_5$.]{ \includegraphics[width=0.47\linewidth,trim=1cm 14.5cm 0.3cm 4cm]{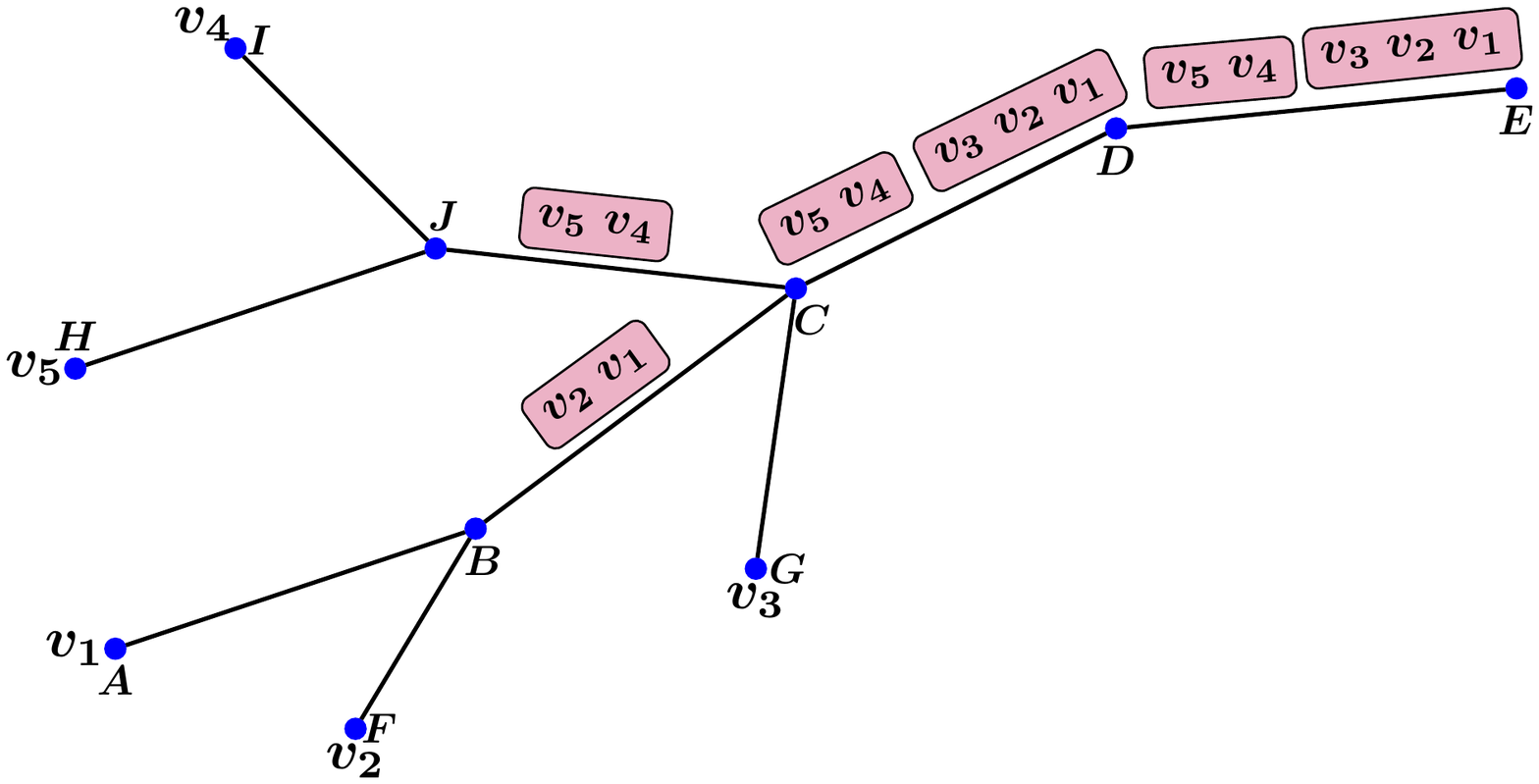} }
\caption{Illustrative figure of platoons formed by five vehicles $v_1,v_2,v_3,v_4,v_5$
given by a typical iteration of Algorithm~\ref{alg:RSHM}. Every vehicle is labeled 
beside its origin node, and the node $E$ is their common destination node.
The blocks plotted on top of some edges represent platoons on those edges.}\label{fig:platoon_illustration}
\end{figure}
\begin{algorithm}
{\footnotesize
	\caption{Route-then-schedule heuristic method}
	\label{alg:RSHM}
	\begin{algorithmic}
      \State {\bf Input}: Model parameters, transportation network, and positive integer $I$ as a threshold termination frequency.
	\State {\bf Output}: Routes and schedules for the CVPP instance.
	\State {\bf Initialization}: Set the $n\gets 1$ and $\widehat{z}\gets\infty$.
			Let $routesFreq$ store the number of times a route assignment is identified.
			Let $\{\cup_v\widehat{\cR}_v,\hat{f},\hat{\ell}, \hat{t}\}$
      and $\widehat{z}$ denote the current best solution and objective value identified.
                \While {$\Call{GetMaxFreq}{routesFreq}<I$} 
                \State{\bf Step 1}: Solve \ref{opt:RDP}-$(n)$ to get the \ref{opt:RDP} routes $\cup_v\cR^{(n)}_v$.
                	\State\hspace{3.5em} $\Call{AddRoutes}{\cup_v\cR^{(n)}_v, routesFreq}$.		
                  \State{\bf Step 2}:  Input the \ref{opt:RDP} routes $\cup_v\cR^{(n)}_v$ into \ref{opt:SPF}-$(n)$. 
                  \State\hspace{3.5em}   Solve \ref{opt:SPF}-$(n)$ to get the optimal
         platooning variables $(f^{(n)}, \ell^{(n)}, t^{(n)})$, and objective value $z^{(n)}$. 
    \If{$z^{(n)}<\widehat{z}$}
    					\State $\widehat{\cR}_v\gets\cR^{(n)}_v$ for all $v\in\cV$.
						  $\hat{f}\gets f^{(n)}$, $\hat{\ell}\gets \ell^{(n)}$, $\hat{t}\gets t^{(n)}$, $\widehat{z}\gets z^{(n)}$.
    				\EndIf	       					 
    \State{\bf Step 3}: Compute the size $|P^{(n)}_{v,i,j}|$ using \eqref{eqn:P^n_ve}.  
		           Update the presumed fuel cost using \eqref{eqn:C^n_ve}.
               \State{\bf Step 4}: Build the objective function $F^{(n+1)}_{\tn{\ref{opt:RDP}}}$ 
		of \ref{opt:RDP}-$(n+1)$ using \eqref{eqn:RDP-obj-update}.	
    		
		\If{$\cR^{(n)}_v=\cR^{(n-1)}_v$ for all $v\in\cV$}  
    			\Return the current best solution $\{\cup_v\widehat{\cR}_v,\hat{f},\hat{\ell},\hat{t}\}$. 
		\Else\hspace{0.3em} Set $n\gets n+1$. 
		\EndIf
	\EndWhile 	 
  	\State \Return the current best solution $\{\cup_v\widehat{\cR}_v,\hat{f},\hat{\ell},\hat{t}\}$.
       	\end{algorithmic}	
       \vspace{-8pt}\noindent\makebox[\linewidth]{\rule{\linewidth}{0.4pt}}\vspace{-5pt}
	\begin{algorithmic}
		\Function{GetMaxFreq}{$routesFreq$}
			\State{$maxCount\gets0$}
			\For{$routes$ in $routesFreq$}
				\If{$routesFreq[routes]>maxCount$}
					\State{$maxCount \gets maxCount+1$.}
				\EndIf
			\EndFor
			\State{\Return $maxCount$.}
		\EndFunction
	\end{algorithmic} 
       \vspace{-8pt}\noindent\makebox[\linewidth]{\rule{\linewidth}{0.4pt}}\vspace{-5pt}
	\begin{algorithmic}
		\Procedure{AddRoutes}{$routes$,$routesFreq$}			
			\If{$routes$ is in $routesFreq$}
				\State{$routesFreq[routes]\gets routesFreq[routes]+1$.}
			\Else
				\State{$routesFreq[routes]\gets 1$.}
			\EndIf
		\EndProcedure
	\end{algorithmic}
}
\end{algorithm}

\subsubsection{RSHM Discussion.}
The first RSHM iteration solves the route design problem \ref{opt:RDP}-(1)
to produce routes $\cR^{(1)}_v$ for each vehicle $v$.
Based on these routes, \ref{opt:SPF}-(1) is solved
to determine the platooning configuration at each edge. 
The second RSHM iteration solves the
route design problem \ref{opt:RDP}-(2) with constraints that  
are the same as for the \ref{opt:RDP} but with the objective:
\begin{equation}\label{eqn:RDP-obj-2}
\begin{aligned}
  {F}_{\tn{\ref{opt:RDP}}}^{(2)}=&\sum_{v\in\cV}\sum_{(i,j)\in \cE\setminus\cup_u\cR^{(1)}_u} C_{i,j} x_{v,i,j} 
  - \sum_{(i,j)\in \cE\setminus\cup_u\cR^{(1)}_u} \sigma^l C_{i,j} y'_{i,j} 
  - \sum_{(i,j)\in \cE\setminus\cup_u\cR^{(1)}_u} \sigma^fC_{i,j} w_{i,j} \\
  &+ \sum_{v\in\cV}\sum_{(i,j)\in \cup_u\cR^{(1)}_u} C^{(2)}_{v,i,j} x_{v,i,j},
 \end{aligned}
\end{equation}
where the adjusted fuel cost $C^{(2)}_{v,i,j}$ for $v\in\cV$ and 
$(i,j)\in\cup_{u\in\cV}\cR^{(1)}_u$ is 
\begin{equation}\label{eqn:C^2_ve}
  C^{(2)}_{v,i,j}=\left\{\def\arraystretch{1.7}
		    \begin{array}{ll}
          \displaystyle \frac{\Cp_{i,j}(|\cP^{(1)}_{v,i,j}|)}{|\cP^{(1)}_{v,i,j}|} &  \textrm{ if }  (i,j)\in\cR^{(1)}_v,\\
          \displaystyle (1-\sigma^f)C_{i,j} & \textrm{ if }  (i,j)\in\cup_{u\in\cV}\cR^{(1)}_u\setminus\cR^{(1)}_v.
		    \end{array}
		    \right.
\end{equation}
We provide the intuition behind the adjustment of the objective in \eqref{eqn:RDP-obj-2}.
All the edges of the network are partitioned into two parts: the explored edges (by some vehicles)
$\cup_{u\in\cV}\cR^{(1)}_u$ and the unexplored edges $\cE\setminus\cup_{u\in\cV}\cR^{(1)}_u$. 
For the unexplored edges, fuel costs are represented by the first three terms
in \eqref{eqn:RDP-obj-2}, and they are the same as in the objective of the \ref{opt:RDP}. 
When a vehicle $v$ is assigned to an explored edge $(i,j)\in\cE\setminus\cup_{u\in\cV}\cR^{(1)}_u$, 
the adjusted objective \eqref{eqn:RDP-obj-2} assumes that the vehicle will
incur $C^{(2)}_{v,i,j}$ amount of fuel cost, which is the last term in \eqref{eqn:RDP-obj-2}.
The adjusted fuel cost $C^{(2)}_{v,i,j}$ is defined differently for two cases in \eqref{eqn:C^2_ve}:
(a) the edge $(i,j)$ is in the route $\cR^{(1)}_v$ and (b) the edge $(i,j)$ is not 
in the route $\cR^{(1)}_v$ but in the route of some other vehicle.
For case (a), the adjusted objective assumes that the vehicle $v$
if assigned again to the edge $(i,j)$
will save at a rate that is equal to the averaged amount of fuel saving
in the platoon $\cP^{(1)}_{v,i,j}$ obtained at the first iteration of RSHM,
where the average is taken over the number of vehicles in $\cP^{(1)}_{v,i,j}$.  
For case (b), the adjusted objective assumes that the vehicle 
can receive the maximum amount of fuel saving if assigned to $(i,j)$.

In general, suppose we are at the beginning of the iteration $n \ge 2$.
We solve the current vehicle routing problem \ref{opt:RDP}-$(n)$
and obtain the \ref{opt:RDP} routes $\cup_{u\in\cV}\cR^{(n)}_u$. 
We input the routes $\cup_{u\in\cV}\cR^{(n)}_v$ into \eqref{opt:SPF}-$(n$)
in order to obtain the size of platoon $|\mP^{(n)}_{v,i,j}|$ 
for each $v\in\cV$ and $(i,j)\in\cR^{(n)}_v$. In particular, $|\mP^{(n)}_{v,i,j}|$
can be computed based on the solution of \ref{opt:RDP}-$(n)$ as
\begin{equation}\label{eqn:P^n_ve}
  |\mP^{(n)}_{v,i,j}|=\left\{\def\arraystretch{1.5}
			\begin{array}{ll}
        \displaystyle 1+\sum_{u\in\cV(i,j):\;u>v} f^{(n)}_{u,v,i,j} & \textrm{ if } \ell^{(n)}_{v,i,j}=1, \\
        \displaystyle 1+\sum_{u\in\cV(i,j):\;u>w} f^{(n)}_{u,w,i,j} & \textrm{ if } f^{(n)}_{v,w,i,j}=1 \textrm{ for a }w\in\cV^{(n)}_{i,j}, \\
			\displaystyle 1 & \textrm{ otherwise}.
			\end{array} 
		  \right.
\end{equation}
In the first case of \eqref{eqn:P^n_ve}, $v$ is the lead vehicle of $\mP^{(n)}_{v,i,j}$.
In the second case, $w$ is the lead vehicle.
In the third case, $v$ forms a single-vehicle platoon.
As in \eqref{eqn:RDP-obj-2}, we update the objective function (for the next iteration) 
based on the set $\cup^n_{k=1}\cup_{u\in\cV}\cR^{(k)}_u$ of explored edges before iteration $n+1$ 
and the set $\cE\setminus\cup^n_{k=1}\cup_{u\in\cV}\cR^{(k)}_u$ of unexplored edges.
The updated objective for the routing problem \ref{opt:RDP}-$(n+1)$ 
at iteration $n+1$ is
\begin{equation}\label{eqn:RDP-obj-update}
\begin{aligned}
  & {F}_{\tn{\ref{opt:RDP}}}^{(n+1)}=\sum_{v\in\cV}\sum_{(i,j)\in \cE\setminus\cup^n_{k=1}\cR^{(k)}} C_{i,j} x_{v,i,j} 
  - \sum_{(i,j)\in \cE\setminus\cup^n_{k=1}\cup_{u\in\cV}\cR^{(k)}_u} \sigma^l C_{i,j} y'_{i,j} 
  - \sum_{(i,j)\in \cE\setminus\cup^n_{k=1}\cup_{u\in\cV}\cR^{(k)}_u} \sigma^fC_{i,j} w_{i,j}  \\
 &\qquad\qquad + \sum_{v\in\cV}\sum_{(i,j)\in \cup^n_{k=1}\cup_{u\in\cV}\cR^{(k)}_u} C^{(n+1)}_{v,i,j} x_{v,i,j},
 \end{aligned}
\end{equation}
and constraints of \ref{opt:RDP}-$(n+1)$ are the same as those of the \ref{opt:RDP}.
The presumed fuel cost for assigning a vehicle $v\in\cV$ to an edge 
$(i,j)\in\cup^n_{k=1}\cup_{u\in\cV}\cR^{(k)}_u$
at iteration $n+1$ is updated based on the following recurrence equation:
\begin{equation}\label{eqn:C^n_ve}
  C^{(n+1)}_{v,i,j}=\left\{\def\arraystretch{1.7}
		    \begin{array}{ll}
          \displaystyle \frac{\Cp_{i,j}(|\mP^{(n)}_{v,i,j}|)}{|\mP^{(n)}_{v,i,j}|} &  \textrm{ if }  (i,j)\in\cR^{(n)}_v,\\
          \displaystyle (1-\sigma^f)C_{i,j}  &  \textrm{ if }  (i,j)\in \cup^n_{k=1}\cup_{u\in\cV}\cR^{(k)}_u\setminus\cR^{(n)}_v\textrm{ and }\mtc{I}(n,v,i,j)=\emptyset,  \\
          \displaystyle C^{(I(n,v,i,j)+2)}_{v,i,j} & \textrm{ if } (i,j)\in \cup^n_{k=1}\cup_{u\in\cV}\cR^{(k)}_u\setminus\cR^{(n)}_v\textrm{ and }\mtc{I}(n,v,i,j)\neq\emptyset.
		    \end{array}
		    \right.
\end{equation}
The intuition of updating the presumed fuel cost in \eqref{eqn:C^n_ve}
is similar to \eqref{eqn:C^2_ve} but with an additional ingredient.
In the first case of \eqref{eqn:C^n_ve} when $v$ is assigned 
to an edge $(i,j)\in\cR^{(n)}_v$ in the next iteration,
the adjusted objective assumes that $v$ will incur the platoon-averaged fuel cost.
In the second and third cases of \eqref{eqn:C^n_ve} when $v$ is 
assigned to an explored (up to iteration $n$) edge $(i,j)$ that is not in the current route of $v$,
the adjusted model updates the $C^{(n+1)}_{v,i,j}$ depending on whether a ``similar''
situation in previous iterations exists.
Specifically, if there exists a
previous iteration $k$  in which the platoon 
configuration at edge $(i,j)$ is the same as in the current iteration,
that is, $\Set*{\mP^{(k)}_{u,i,j}}{u\in\cV^{(k)}_{i,j}}=\Set*{\mP^{(n)}_{u,i,j}}{u\in\cV^{(n)}_{i,j}}$,
and $v$ is assigned to $(i,j)$ in iteration $k+1$, then we set the costs for
$v$ to traverse $(i,j)$ in the next iteration (iteration $n+1$),
to be the same fuel cost as in iteration $k+1$. 
The set of iteration indices satisfying these conditions is exactly
the set $\mtc{I}(n,v,i,j)$ introduced in Definition~\ref{def:extr-idx}.
Every index in $\mtc{I}(n,v,i,j)$ can be viewed as having a local 
similarity with iteration $n$ at edge $(i,j)$.
If $\mtc{I}(n,v,i,j)$ is empty, it means that no previous iteration
has a local similarity with iteration $n$ at edge $(i,j)$,
and then the adjusted objective makes the most optimistic estimation
of the adjusted fuel cost $C^{(n+1)}_{v,i,j}\gets (1-\sigma^f)C_{i,j}$ (case 2 in \eqref{eqn:C^n_ve}). 
If $\mtc{I}(n,v,i,j)$ is nonempty,
we pick the most recent one in it, which is $I(n,v,i,j)$. Notice that 
in this case, vehicle $v$ has been assigned to $(i,j)$ at iteration 
$I(n,v,i,j)+1$ and the adjusted fuel cost obtained at the end of iteration 
$I(n,v,i,j)+1$ corresponding to this assignment is $C^{(I(n,v,i,j)+2)}_{v,i,j}$.
Thus, by analogy to the most recent local similarity, one can reasonably 
estimate $C^{(n+1)}_{v,i,j}$ as $C^{(n+1)}_{v,i,j}\gets C^{(I(n,v,i,j)+2)}_{v,i,j}$
(case 3 in \eqref{eqn:C^n_ve}). 
The algorithm then goes to iteration $n+1$, and the steps are repeated. 
The sequence of updating key quantities involved in the RSHM is illustrated
in Figure~\ref{fig:RSHM_scheme}.
Note that at iteration $n$, if the presumed total fuel cost
obtained from solving \ref{opt:RDP}-$(n)$ is very different from the real
total fuel cost after imposing the time constraints, the fuel cost adjustment
at the end of iteration $n$ is likely to encourage vehicles to explore different routes
in the next routing problem \ref{opt:RDP}-$(n+1)$.

\begin{wrapfigure}{r}{0.5\textwidth}
\centering
\vspace{-20pt}
\includegraphics[width=\linewidth,trim=0cm 16cm 0cm 4.6cm,clip]{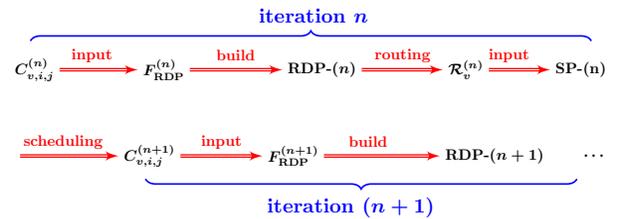}
\caption{Updating of key RSHM quantities in two consecutive iterations.}
\label{fig:RSHM_scheme}
\vspace{-20pt}
\end{wrapfigure}

We remark that the goal of the RSHM is to encourage vehicles to explore routes 
that are potentially more fuel economic. This intention can be seen from
the optimistic estimation of fuel cost when a vehicle is assigned to an unexplored edge.
On the other hand, if such an optimistic updating of presumed fuel cost results in  
the same routes in every iteration, the routes are likely to be true optimal routes.
In Theorem~\ref{thm:RSHM-alg-converge}, we show that the performance 
of the RSHM is guaranteed in special cases when the sequence
of route assignment generated by the algorithm satisfies the following condition.
For a given optimal route assignment $\cup_v\cR^*_v$,
we define the notations $\cV^*_{i,j}$, $\mP^*_{v,i,j}$,
$\cL^*_{i,j}$, and $z^*$ as in Definition~\ref{def:notations}. 

\begin{definition}\label{def:nest-cond}
Let $\cup_v \cR^*_v$ be
the route assignment corresponding to an optimal solution of a CVPP instance.  
Consider an edge $(i,j)\in\cup_v\cR^*_v$
and a platoon $\mP^*_{u,i,j}$ for some $u\in\cV^*_{i,j}$. 
A route assignment $\cup_v\cR^{(k)}_v$ of Algorithm~\ref{alg:RSHM} satisfies the 
\emph{local nested condition (i) at $\mP^*_{u,i,j}$} 
if
\begin{itemize}
  	\item[($\romannumeral 1$)] \replace{}{There exists a platoon $\mP^{(k)}_{w,i,j}$ 
  		(for some $w\in\cV^{(k)}_{i,j}$) such that $\mP^{(k)}_{w,i,j}\cap\mP^*_{u,i,j}$ is nonempty,
		$|\mP^{(k)}_{w,i,j}|\ge 1 + \bs{1}\{|\mP^*_{u,i,j}|>1\}$,
   		and $(\mP^{(k)}_{w,i,j}\setminus\mP^*_{u,i,j})\cap\cV^*_{i,j}=\emptyset$.}
 		 \replace{}{Furthermore, $\mP^{(k)}_{w^{\prime},i,j}\cap \mP^*_{u,i,j}$ is empty for 
  		any platoon $\mP^{(k)}_{w^{\prime},i,j}$ with $w^{\prime}\in\cV^{(k)}_{i,j}\setminus\mP^{(k)}_{w,i,j}$.}
\end{itemize}
The route assignment $\cup_v\cR^{(k)}_v$ satisfies the 
\emph{local nested condition (ii) at $\mP^*_{u,i,j}$} 
if 
\begin{itemize}
	\item[($\romannumeral 2$)] \replace{}{$\cV^{(k)}_{ij}\cap\mP^*_{u,i,j}=\emptyset$.}
\end{itemize}
\end{definition}

\begin{definition}\label{def:noncrossing-cond}
Let $\cup_v\cR^*_v$ be an optimal route assignment. 
A sequence of route assignments $\mathcal{S}=\{\cup_v\cR^{(k)}_v\}^n_{k=1}$ 
generated from Algorithm~\ref{alg:RSHM}
satisfies the \emph{noncrossing condition with respect to $\cup_v\cR^*_v$} 
if for every $(i,j)\in\cup_v\cR^*_v$ and every platoon $\mP^*_{u,i,j}$ ($u\in\cV^*_{i,j}$),
at least one of the following conditions is satisfied:
\begin{itemize}
  \item[($\romannumeral 1$)] \replace{}{The route assignment $\cup_v\cR^{(n)}_v$ 
  satisfies the local nested condition (i) or (ii) at $\mP^*_{u,i,j}$.
  Furthermore, for every $v\in \mP^*_{u,i,j}\setminus\mP$ 
    the set $\mI(n,v,i,j)$ is empty, where 
    \begin{equation}
    \mP=\left\{  
      \begin{array}{ll}
      \textrm{the platoon in $\cV^{(n)}_{i,j}$ 
       satisfying Definition~\ref{def:nest-cond}(i)}  & \textrm{if local nested cond.(i) is met}, \\
     \emptyset  &  \textrm{if local nested cond.(ii) is met}.
     \end{array}
    \right.
    \end{equation}
    }
  \item[($\romannumeral 2$)] For any $\cup_v\cR^{(k)}_v$ in $\mathcal{S}$, 
  if there exists a platoon $\mP^{(k)}_{w,i,j}$ for some $w\in\cV^{(k)}_{i,j}$ 
  satisfying $\mP^{(k)}_{w,i,j}\cap\mP^*_{u,i,j}\neq\emptyset$, 
  then $\mP^*_{u,i,j}\subseteq\mP^{(k)}_{w,i,j}$.
\end{itemize} 
\end{definition}

The proof of Theorem~\ref{thm:RSHM-alg-converge} appears in Section~\ref{proof_of_RSHM_theorem}. 
\begin{theorem}\label{thm:RSHM-alg-converge}
Let $\{\cup_v\cR^{(k)}_v\}_{k\ge1}$ be the sequence of \ref{opt:RDP} routes generated by the RSHM. 
Let $\cup_v\cR^*_v$ be an optimal route assignment. 
Then the following properties hold:
\begin{itemize}
  \item[\emph{(a)}] The algorithm will terminate in finitely many iterations.
      \item[\emph{(b)}] Let $n$ be the last iteration of Algorithm~\ref{alg:RSHM}.
         If $\cR^{(n)}_v=\cR^{(n-1)}_v$ for all $v\in\cV$,
				then 
				\begin{equation}\label{eqn:gap-Fn-F*}
				\begin{aligned}
          z^{(n)}-z^* \le&\sum_{(i,j)\in\cup_v\cR^*_v\setminus\cup_v\cR^{(n-1)}_v}-\sigma^f(|\cL^*_{i,j}|-1) \\
          &+ \sum_{(i,j)\in(\cup_v\cR^*_v)\cap\left(\cup_v\cR^{(n-1)}_v\right)}
          \sum_{u\in\cL^*_{i,j}} \sum_{v\in \cU^{(n-1)}_{u,i,j}}
          \left(\frac{C^{\emph{plat}}_{i,j}(|\cP^{(n_v)}_{v,i,j}|)}{|\cP^{(n_v)}_{v,i,j}|} - \frac{\widetilde{C}^{\emph{plat}}_{i,j}(|\cU^{(n-1)}_{u,i,j}|)}{|\cU^{(n-1)}_{u,i,j}|}\right),
				\end{aligned}
				\end{equation}
        where $\cU^{(n-1)}_{u,i,j}:=\cP^*_{u,i,j}\cap\left(\cV^{(n-1)}_{i,j}\cup\Set*{w}{\cI(n-1,w,i,j)\neq\emptyset} \right)$,
        and $\widetilde{C}^{\emph{plat}}_{i,j}(l)$ is defined as
        \begin{displaymath}
        \widetilde{C}^{\emph{plat}}_{i,j}(l)=\left\{\def\arraystretch{1.7}
        		\begin{array}{ll}
		(1-\sigma^l)C_{i,j} & \textrm{ if } l=1, \\
		C^{\emph{plat}}_{i,j}(l) & \textrm{ if } l=0,\textrm{ or }l\ge 2,
		\end{array}
        		\right.
	\end{displaymath}
	and where index $n_v$ is defined as
	\bdm
	n_v=\left\{
		\begin{array}{ll}
		n-1 & \textrm{ if } v\in\cV^*_{i,j}\cap\cV^{(n-1)}_{i,j}, \\
		I(n-1,v,i,j)+1 & \textrm{ if } v\in\cV^*_{i,j}\setminus\cV^{(n-1)}_{i,j}\textrm{ and }\mI(n-1,v,i,j)\neq\emptyset.
		\end{array}
		\right.
	\edm
	Recall that the definitions of $z^{(n)},\cV^{(n)}_{i,j},\cP^{(n)}_{v,i,j}$ are given in Definition~\ref{def:notations}
	and that $z^*,\cV^*_{i,j},\cP^*_{v,i,j}$ are their counterparts in an optimal route assignment of the CVPP.
      \item[\emph{(c)}] At any iteration $n$, if the sequence $\{\cup_v\cR^{(k)}_v\}^{n-1}_{k=1}$ 
        satisfies the noncrossing condition with respect to $\cup_v\cR^*_v$, then
        $F^{(n)}_{\tn{\ref{opt:RDP}}}(\cup_v\cR^*_v)\le z^*$, 
        where $F^{(n)}_{\tn{\ref{opt:RDP}}}(\cup_v\cR^*_v)$ is the objective value of
        $F^{(n)}_{\tn{\ref{opt:RDP}}}$ evaluated at routes $\cup_v\cR^*_v$. 
      \item[\emph{(d)}] Let $n$ be the last iteration of Algorithm~\ref{alg:RSHM}.
      		If the sequence $\{\cup_v\cR^{(k)}_v\}^{n-1}_{k=1}$ 
		satisfies the noncrossing condition with respect to $\cup_v\cR^*_v$ 
		and if $\cR^{(n)}_v=\cR^{(n-1)}_v$ for all $v\in\cV$, 
		then the RSHM returns an optimal solution of the CVPP. 
  \item[\emph{(e)}] Let $n$ be the last iteration of Algorithm~\ref{alg:RSHM}.
       If $\cR^{(n)}_v=\cR^{(n-1)}_v$ for all $v\in\cV$
       and if the set $\cI(n-1,v,i,j)$ is empty for all $\cup^n_{k=1}\cup_{u\in\cV}\cR^{(k)}_u$, then
       \beq
       z^{(n)}-z^*\le \sum_{(i,j)\in\cup_u\cR^{(n)}_u}\sum_{v\in\cL^{(n)}_{i,j}}\left[\left(1-\frac{|\cP^{(n)}_{v,i,j}|}{\lambda} \right)(\sigma_f-\sigma_l)+\bs{1}\{|\cP^{(n)}_{v,i,j}|=1\}\sigma_l\right]C_{i,j}.
       \eeq
\end{itemize}
\end{theorem}

\section{Valid Inequalities of the Routing and Scheduling Problems} \label{sec:sol-route}
We now present valid and facet-defining inequalities for
the feasible sets in the routing problem
and the scheduling problem, respectively.
Adding these inequalities to
these problems 
\replace{}{c}an tighten the 
two formulations and may be beneficial in the computational performance.
Specifically, for the routing problem
the goal is to identify inequalities that
define the convex hull of the feasible set induced by \ref{eqn:RDP_2}--\ref{eqn:RDP_6}.
For the scheduling problem
we identify two families of valid inequalities.
Inequalities in the first family cut fractional solutions due to
the introduction of $M_{u,v,i,j}$ coefficients in the constraints that ensure
platooning vehicles travel edges simultaneously.
Inequalities in the second family improve the representation of the 
platooning polytope defined by the feasible points satisfying the 
constraints \ref{eqn:SPF_7}--\ref{eqn:SPF_9}, \ref{eqn:SPF_11}, and \ref{eqn:SPF_12}. 
The valid inequalities for the routing problem and scheduling problem are presented in
Sections~\ref{sec:valid-ineq-route} and \ref{sec:valid-ineq-schedule}, respectively.
The improvement in computational performance due to these valid inequalities 
is analyzed in Sections~\ref{sec:num-routing} and \ref{sec:num-sched}.

\subsection{Valid Inequalities for Tightening the Routing Problem Formulation}
\label{sec:valid-ineq-route}
We first study polyhedral properties of the polytope 
defined by \ref{eqn:RDP_2}--\ref{eqn:RDP_6}. Specifically,
polytope $P^{\textrm{RDP}}_{i,j}$ for $(i,j)\in \cE$ is defined as 
\begin{equation}\label{def:P_RDP_ij}
P^{\textrm{RDP}}_{i,j}:=\textrm{conv}
\Set*{
\begin{array}{l}
[x_{v,i,j}, y_{i,j}, w_{i,j}, y'_{i,j}]\\
\forall v\in\cV
\end{array}
}
{
\begin{aligned}
& \sum_{v\in\cV} x_{v,i,j} \ge 2 y'_{i,j},\\
& w_{i,j} - \sum_{v\in\cV} x_{v,i,j} + y_{i,j} \le 0, \\
& x_{v,i,j}\le y_{i,j}\; \forall v\in\cV,\;  y'_{i,j}\le y_{i,j}, \\
& x_{v,i,j},y_{i,j}, y'_{i,j}\in\{0,1\},\; w_{i,j}\ge 0 
\end{aligned}
}.
\end{equation}
The polytope $P^{\textrm{RDP}}_{i,j}$ is the convex hull of some integer points.
One can give a linear system representation of $P^{\textrm{RDP}}_{i,j}$
by adding three facet-defining inequalities to \eqref{def:P_RDP_ij}.
The full representation of $P^{\textrm{RDP}}_{i,j}$ is given by the following
theorem, proven in Section~\ref{proof_of_RDP_thrm}.

\begin{theorem}\label{thm:P_RDP_ij_full}
Let $|\cV|\ge 3$. For any $(i,j)\in \cE$, consider the polytope $P^{\textrm{RDP}}_{i,j}$. 
\emph{(a)} The following inequality is facet defining for $P^{\textrm{RDP}}_{i,j}$:
$\sum_{v\in\cV}x_{v,i,j}\ge y_{i,j} +  y'_{i,j}$.
\emph{(b)} The polytope $P^{\textrm{RDP}}_{i,j}$ can be represented by the linear system
\begin{equation}\label{eqn:P_RDP_ij_full}
P^{\textrm{RDP}}_{i,j}=\Set*{\begin{aligned} &x_{v,i,j}\; \forall v\in\cV, \\ & y_{i,j},  y'_{i,j}, w_{i,j}   \end{aligned}}{
	\begin{aligned}
		&\sum_{v\in\cV} x_{v,i,j}\ge 2 y'_{i,j},\;\; x_{v,i,j}\le y_{i,j} \;\;\forall v\in\cV,\;\;y_{i,j}\ge  y'_{i,j}, \\
		& w_{i,j} - \sum_{v\in\cV} x_{v,i,j} +y_{i,j} \le 0,\;\; \sum_{v\in\cV}x_{v,i,j}\ge y_{i,j} +  y'_{i,j}, \\ 
		&0\le x_{v,i,j}\le 1\;\; \forall v\in\cV,\;\; 0\le y_{i,j}\le 1, \;\; 0\le y'_{i,j}\le 1, \;\; w_{i,j}\ge 0
	\end{aligned}
  }.
\end{equation}

\end{theorem}

\subsection{Valid Inequalities for Tightening the Scheduling Formulation}
\label{sec:valid-ineq-schedule}
Let $\Psp$ be the convex hull of the feasible set of the \ref{opt:SPF}.
For simplicity, we use shorthand notations such as $t$, $f$, or $\ell$
to denote the decision variables in the \ref{opt:SPF}
with the missing indices enumerated over their possible values.
For example, $t$ represents the subset of decision variables
$\{t_{u,i}\;|\;u\in V,\; i\in\cR_u\}$, and $f$ represents the collection
of decision variables $\{f_{u,v,i,j}\;|\;u,v\in\cV,\; (i,j)\in\cR_u\cap\cR_v\}$.

We divide our polyhedral study of the $\Psp$ into two parts. 
We first investigate the polytope $\Psp_1$, which is a relaxation of $\Psp$ 
that consists of constraints involving time decision variables $t$:
\begin{equation}
\begin{aligned}
\Psp_1:=\textrm{conv}\Set*{ t, y}{\textrm{ s.t. } 
\tn{\ref{eqn:SPF_2}}-\tn{\ref{eqn:SPF_6}},\;
\tn{\ref{eqn:SPF_11}}}.
\end{aligned}
\end{equation}   
Then we investigate the vehicle-counting polytope $\Psp_2$, defined as
\begin{equation}\label{eqn:Psp_2}
\Psp_2:=\textrm{conv}\Set*{y, z}{ \textrm{ s.t. }  
\tn{\ref{eqn:SPF_7}}-\tn{\ref{eqn:SPF_9}}, \; \tn{\ref{eqn:SPF_11}}, \; \tn{\ref{eqn:SPF_12}}
				      }.			
\end{equation} 
We also consider the polytope $\Psp_{2,i,j}$ defined as
\begin{equation}
\Psp_{2,i,j}=\textrm{conv}\Set*{
\begin{array}{l}
f_{u,v,i,j}\;\forall u,v\in\cV_{i,j}\;u>v,\\
\ell_{v,i,j}\; \forall v\in\cV_{i,j}
\end{array}
}{ \textrm{ s.t. }  
\tn{\ref{eqn:SPF_7}}-\tn{\ref{eqn:SPF_9}}, \; \tn{\ref{eqn:SPF_11}}, \; \tn{\ref{eqn:SPF_12}}
				      \textrm{ for }(i,j) }.
\end{equation}
Note that since $ \Psp_2=\bigotimes_{(i,j)\in\cup_v\cR_v} \Psp_{2,i,j}$,
it suffices to understand the structure of $\Psp_{2,i,j}$ in order to understand
the structure of $\Psp_2$.
We define the following notations to ease our discussion later. 
\begin{definition}\label{def:Rv-t}\
\begin{enumerate}
	\item Let $\cR_v(a,b)$ denote the subpath from node $a$ to node $b$ on $\cR_v$, 
and let $N_v(a,b)$ be the set of nodes on $\cR_v(a,b)$ (including $a$ and $b$). 
For any node $i\in\cR_v$, let $i^-$ ($i^+$) denote
the preceding (succeeding) node of $i$ on $\cR_v$, if it exists.
	\item $\underline{t}_{u,i}$ and $\overline{t}_{u,i}$ ($\forall u\in\cV,\forall i\in N_u$): 
        the lower and upper bound of $t_{u,i}$,
		 where $\underline{t}_{u,i}=T_u^{O_u}+\sum_{(r,s)\in\cR_u(O_u,i)}T_{rs}$
		 and $\overline{t}_{u,i}=T_u^{D_u}-\sum_{(r,s)\in\cR_u(i,D_u)}T_{rs}$.
	\item $f_{\ol{u,v},i,j}$ ($\forall u,v\in\cV,\forall (i,j)\in\cR_u\cap\cR_v$): If $u>v$, $f_{\ol{u,v},i,j}$ is $f_{u,v,i,j}$.
    If $v>u$, $f_{\ol{u,v},i,j}$ represents $f_{v,u,i,j}$. 
\end{enumerate}
\end{definition}
One can easily check that a necessary \emph{platoonable condition} is satisfied
for vehicles $u$ and $v$ 
to be platooned when traversing the edge $(i,j)\in\cR_u\cap\cR_v$ is
$\overline{t}_{u,i}-\underline{t}_{v,i}>0$ and $\ol{t}_{v,i}-\underline{t}_{u,i}>0$.
Therefore, at a preprocessing stage, constraints \ref{eqn:SPF_5}
and \ref{eqn:SPF_6} will be removed for $u,v,i,j$ that violate the platoonable condition, 
and we set the corresponding platooning variable $f_{u,v,i,j}=0$.  
For $u,v,i,j$ that satisfies this condition we can set 
$M_{u,v,i,j}=\textrm{max}\big\{\big\{ \overline{t}_{u,i}-\underline{t}_{v,i}\big\}^+,
\big\{\overline{t}_{v,i}-\underline{t}_{u,i}\big\}^+ \big\}$.

\subsubsection{Scenario-Based Disjunctive Cutting Plane Generation for $\Psp_1$.}
\label{sec:disj-cut}
Consider the root-relaxation problem of the \ref{opt:SPF} in a branch-and-bound
approach for solving the \ref{opt:SPF}. In the optimal solution of this relaxation problem,
some platooning variables $f_{u,v,i,j}$ may 
be fractional even though vehicles $u,v$ are not 
platoonable because of inconsistent entering time at node $i$.
We develop Algorithm~\ref{alg:act-constr-coll}
and derive disjunctive inequalities based on the active constraints identified 
by this algorithm to cut fractional solutions 
with respect to the polytope $\Psp_1$ when solving the mixed 0-1 linear program that defines \ref{opt:SPF}.
A key step in Algorithm~\ref{alg:act-constr-coll} is to recursively search and collect active constraints in the \ref{opt:SPF}  
that produce the current fractional solution. 
Let $\Psprel$ be the polytope after relaxing
the integral constraints of variables $f$ in $\Psp_1$, that is, 
$\Psprel :=\Set*{ t, f}{\textrm{ s.t. } 
				  \tn{\ref{eqn:SPF_1}}-\tn{\ref{eqn:SPF_6}},\; 
          \bs{0}\le f\le \bs{1}}$. 
More specifically, Algorithm~\ref{alg:act-constr-coll} takes a fractional point $[\hat{t},\hat{f}]$
(i.e., at least one entry of $\hat{f}$ is fractional) as the input. If the point $[\hat{t},\hat{f}]$ is an extreme point of 
$\Psprel$, it returns two subsets $\cV_1$ and $\cV_2$ of vehicles that induce a subset of constraints
in (the linear system definition of) $\Psprel$ that are active at the input point.
If $[\hat{t},\hat{f}]$ is not an extreme point of $\Psprel$, the algorithm returns ``None.''
An example demonstrating this algorithm is given in Section~\ref{app:disj-cut-example}.
Suppose the algorithm terminates with nonempty sets of $\cV_1$ and $\cV_2$.
Then the following constraints are active at the fractional point $[\hat{t},\hat{f}]$ input to the algorithm
(the sets $\mtc{U}_1,\mtc{U}_2,\mtc{F}$ are defined in \eqref{eqn:U1-U2-F}):
\begin{align}
& t_{u,j} = t_{u,i} + T_{i,j}  &  \forall u\in \cV_1\cup\cV_2,\;\forall (i,j)\in\cR_u, \label{eqn:ac1}\\
& t_{u,O_u}\ge \underline{t}_{u,O_u}  &  \forall u\in \mathcal{U}_1, \label{eqn:ac2}\\
&  t_{u,O_u}\le \ol{t}_{u,O_u} & \forall u\in \mathcal{U}_2, \label{eqn:ac3}\\
& t_{u,i}-t_{v,i} \le M_{u,v,i,j}(1-f_{u,v,i,j})   & \forall (u,v,i,j)\in\mathcal{F}, \label{eqn:ac4}\\
& t_{u,i}-t_{v,i} \ge -M_{u,v,i,j}(1-f_{u,v,i,j})   & \forall (u,v,i,j)\in\mathcal{F}, \label{eqn:ac5}
\end{align}
\begin{align}
& \left\{
\begin{array}{ll}
t_{u^*,i^*}-t_{v^*,i^*} \le M_{u^*,v^*,i^*,j^*}(1-f_{u^*,v^*,i^*,j^*}) & \textrm{ if } \hat{t}_{u^*,i^*}-\hat{t}_{v^*,i^*}> 0 \\
t_{u^*,i^*}-t_{v^*,i^*} \ge -M_{u^*,v^*,i^*,j^*}(1-f_{u^*,v^*,i^*,j^*}) & \textrm{ if } \hat{t}_{u^*,i^*}-\hat{t}_{v^*,i^*}< 0 ,
\end{array}
\right.\label{eqn:ac6}
\end{align}
where $u^*,v^*,i^*,j^*$ are determined in Line~\ref{line:y} of Algorithm~\ref{alg:act-constr-coll}
and the sets $\mathcal{U}_1$, $\mathcal{U}_2$, and $\mathcal{F}$ are 
\begin{equation}\label{eqn:U1-U2-F}
\begin{aligned}
&\mathcal{U}_1=\Set*{u\in \cV_1\cup \cV_2}{\hat{t}_{u,O_u}=\underline{t}_{u,O_u}}, 
\quad\mathcal{U}_2=\Set*{u\in \cV_1\cup \cV_2}{\hat{t}_{u,O_u}=\ol{t}_{u,O_u}}, \\
&\mathcal{F}=\bigcup_{k=1,2}\Set*{(u,v,i,j)}{u,v\in \cV_k,\;u>v,\;(i,j)\in\cR_u\cap\cR_v,\; \hat{f}_{u,v,i,j}=1}.
\end{aligned}
\end{equation}
The way of generating a disjunctive cut toward a given fractional solution is
given in Theorem~\ref{thm:disj-cut}, which relies on
Proposition~\ref{prop:constrList}; both are proven in
Section~\ref{proofs_for_disj_cuts}.

\begin{algorithm}
{
\footnotesize
	\caption{Collecting a set of vehicles with constraints that yield a fractional solution.}
	\label{alg:act-constr-coll}
	\begin{algorithmic}[1]
		\State{\bf Input}: An extreme point $[\hat{t},\hat{f}]$ of $\Psprel$ 
					that contains fractional values in $\hat{f}$.
		\State{\bf Output}: Two subsets of vehicles or None.  
		\State Find $u^*,v^*\in\cV$ (with $u>v$) and $(i^*,j^*)\in\cR_{u^*}\cap\cR_{v^*}$ such that 
			   $\hat{f}_{u^*,v^*,i^*,j^*}$ is fractional, 
			   and $|\hat{t}_{u^*,i^*}-\hat{t}_{v^*,i^*}|=M_{u^*,v^*,i^*,j^*}(1-\hat{f}_{u^*,v^*,i^*,j^*})$. \label{line:y}
		\If{such $u^*,v^*,i^*,j^*$ do not exist}
			\State \textbf{stop} and \Return None. \label{line:empty1}
		\EndIf  
		\State Set $\cV_1\gets\{u^*\}$, $\cV_2\gets\{v^*\}$, $flag\gets \textrm{None}$.  \label{line:init-V1-V2}
		\If{$\underline{t}_{u^*,O_{u^*}}<\hat{t}_{u^*,O_{u^*}}<\ol{t}_{u^*,O_{u^*}}$}
			\State \Call{DeepSearch}{$flag$,\;$\cV_1$}.  \label{line:deep-search}
		\EndIf
		\State \textbf{if} $flag=0$, \textbf{then} \textbf{stop} and \Return None.   \label{line:empty2}
		\State $flag\gets \textrm{None}$. 
		\If{$\underline{t}_{v^*,O_{v^*}}<\hat{t}_{v^*,O_{v^*}}<\ol{t}_{v^*,O_{v^*}}$}
			\State \Call{DeepSearch}{$flag$,\;$\cV_2$}.
		\EndIf
		\State \textbf{if} $flag=0$, \textbf{then} \textbf{stop} and \Return None.   \label{line:empty3}
		\State \Return $[\cV_1,\;\cV_2]$. \label{line:terminate}
       \end{algorithmic}
       \vspace{-8pt}\noindent\makebox[\linewidth]{\rule{\linewidth}{0.4pt}}\vspace{-5pt}
       \begin{algorithmic}[1]
		\Procedure{DeepSearch}{$flag$,\;$vehSet$}
			\State $continue\gets 1$
			\While{$continue=1$}
				\If{$\exists$ $v\in vehSet$, $\exists$ $u\in\cV\setminus vehSet$, 
					and $\exists$ $(r,s)\in\cR_{v}\cap\cR_u$ such that $\hat{f}_{\ol{u,v},r,s}=1$} \label{line:cond-add-u}
					\State $vehSet.add(u)$. \label{line:add-u}
					\If{$\hat{t}_{u,O_{u}}=\underline{t}_{u,O_u}$ or $\hat{t}_{u,O_{u}}=\ol{t}_{u,O_u}$} \label{line:cond}
						\State $continue\gets 0$. $flag\gets 1$. \label{line:cont=0}
					\Else
						\State \Call{DeepSearch}{$flag,\;vehSet$}.
					\EndIf
				\Else
					\State $flag\gets 0$. \textbf{Stop}. \label{line:signal0} 
		 		\EndIf
			\EndWhile
		 \EndProcedure
	\end{algorithmic}
 }	
\end{algorithm}

\begin{proposition}\label{prop:constrList}
Suppose $[\hat{t},\hat{f}]$ is an extremal point of $\Psprel$ that contains at
least one fractional entry in $\hat{f}$. If $[\hat{t},\hat{f}]$ is input into
Algorithm~\ref{alg:act-constr-coll}, then the algorithm will return nonempty
$\cV_1$ and $\cV_2$.
\end{proposition}

\begin{theorem}\label{thm:disj-cut}
Suppose $[\hat{t},\hat{f}]$ is a fractional point in $\Psprel$.
Suppose Algorithm~\ref{alg:act-constr-coll} returns nonempty vehicle sets $\cV_1$ and $\cV_2$.
Let $\omega$ denote the variables involved in the active constraints 
\eqref{eqn:ac1}--\eqref{eqn:ac5}
\emph{(}$\omega$ is a subset of variables from $[t,f]$\emph{)}    
representing the constraints \eqref{eqn:ac1}--\eqref{eqn:ac5} as $A\omega\ge b$,
where $A$ and $b$ are appropriate constant matrix and vector that can represent
those constraints.
Let $\hat{\omega}$ be values of $[\hat{t},\hat{f}]$ corresponding to variables $\omega$.
Then the following holds: \newline
\emph{(a)} \emph{(}Validness of the algorithm\emph{)}  If $\hat{t}_{u^*,i^*}-\hat{t}_{v^*,i^*}> 0$,
 the point $[\hat{\omega},\hat{f}_{u^*,v^*,i^*,j^*}]$ is the unique solution to the linear equation system 
\begin{equation}
 \left\{
	\begin{array}{l}
		A\omega=b \\
		t_{u^*,i^*}-t_{v^*,i^*} = M_{u^*,v^*,i^*,j^*}(1-f_{u^*,v^*,i^*,j^*}).
	\end{array}
\right.
\end{equation}
If $\hat{t}_{u^*,i^*}-\hat{t}_{v^*,i^*}< 0$,
 the point $[\hat{\omega},\hat{f}_{u^*,v^*,i^*,j^*}]$ is the unique solution to the linear equation system 
\begin{equation}
 \left\{
	\begin{array}{l}
		A\omega=b \\
		t_{u^*,i^*}-t_{v^*,i^*} = -M_{u^*,v^*,i^*,j^*}(1-f_{u^*,v^*,i^*,j^*}).
	\end{array}
\right.
\end{equation}
\emph{(b)} \emph{(}Disjunctive cut generation\emph{)}
Let $\widetilde{A}\omega+\tilde{p} f_{u^*,v^*,i^*,j^*}\ge\tilde{b}$ represent the system of linear inequalities:
\begin{displaymath}
\begin{aligned}
&A\omega\ge b, & \\
&t_{u^*,i^*}-t_{v^*,i^*} \le M_{u^*,v^*,i^*,j^*}(1-f_{u^*,v^*,i^*,j^*}), & \\
&t_{u^*,i^*}-t_{v^*,i^*} \ge -M_{u^*,v^*,i^*,j^*}(1-f_{u^*,v^*,i^*,j^*}), & \\
&\ul{t}_{u,O_u}\le t_{u,O_u}\le\ol{t}_{u,O_u} &\forall u\in\cV_1\cup\cV_2, \\
&0\le f_{u,v,i,j}\le 1\; & \forall(u,v,i,j)\in\mathcal{F}, \\
&0\le f_{u^*,v^*,i^*,j^*}\le 1, &
\end{aligned}
\end{displaymath}
where $\tilde{A}$, $\tilde{p}$, and $\tilde{b}$ are appropriate matrix and vectors 
that can represent the above constraints.
Consider the following linear program:
\begin{equation}\label{opt:disj-cut-gen}
\begin{aligned}
\minimize\;\;& \alpha^{\top}\hat{\omega}
+(\beta^{0\top}\tilde{b}-\beta^{1\top}\tilde{b}
+\beta^{1\top}\tilde{p}+\gamma^0-\gamma^1)\hat{f}_{u^*,v^*,i^*,j^*}
-\beta^{0\top}\tilde{b}+\gamma^1 \\
\emph{subject to:}\;\;& \widetilde{A}^{\top}\beta^{0}-\alpha=0, 
\;\; \widetilde{A}^{\top}\beta^{1}-\alpha=0,\;\;
\beta^0\ge\bs{0},\;\;\beta^1\ge\bs{0},\;\;\gamma^0\ge0,
1\ge\gamma^1\ge0,\;\;\alpha\in\mathbb{R}^{\dim(\omega)}.
\end{aligned}
\end{equation}
The inequality 
\begin{equation}\label{eqn:disj-valid-ineq}
\hat{\alpha}^{\top}\omega
+(\hat{\beta}^{0\top}\tilde{b}-\hat{\beta}^{1\top}\tilde{b}
+\hat{\beta}^{1\top}\tilde{p}+\hat{\gamma}^0-\hat{\gamma}^1)f_{u^*,v^*,i^*,j^*}
-\hat{\beta}^{0\top}\tilde{b}+\hat{\gamma}^1\ge 0
\end{equation}
is valid for $\Psp_{1}$, 
and it strongly separates the point $[\hat{t},\hat{f}]$ from $\Psp_{1}$,
where the coefficient vector $[\hat{\alpha},\hat{\beta}^0,\hat{\beta}^1,
\hat{\gamma}^0,\hat{\gamma}^1]$ is  
the optimal solution of \eqref{opt:disj-cut-gen}.   
\end{theorem}

\subsubsection{Valid Inequalities of $\Psp_{2,i,j}$.}
\label{sec:valid-ineq-SPF2}
The polytope $\Psp_{2,i,j}$ is defined by constraints that specify the rules of 
how platoons can be formed by vehicles sharing a same edge.
We can create a graph that represents the set of vehicles that share a given edge,
where every vertex represents a vehicle and where an edge is added to every pair of a lead
vehicle and a trailing vehicle within a platoon. This graph can be decomposed into
(vertex- and edgewise) disjointed star subgraphs (see Definition~\ref{def:star-partition});
the size of each star subgraph is at most $\lambda$. 
\begin{definition}\label{def:star-partition}
(a) For $(i,j)\in\cup_v\cR_v$, a set of values 
$S=\{\hat{f}_{u,v,i,j}:\;u,v\in\cV_{i,j},\;u>v\}$
\textit{specifies a star partition} if there exists some disjoint subsets 
$\{U_1,\ldots,U_k\}$ of $\cV_{i,j}$ and vehicles $\{u_s\}^k_{s=1}$ with $u_s\in U_s$ 
for every $s\in\replacemath{[k]}{ \left\{ 1,\ldots, k\right\}  }$ such that the following condition holds:
$\hat{f}_{u,u_s,i,j}=1$ for every $u\in U_s\setminus\{u_s\}$ and every $s\in \replacemath{[k]}{ \left\{ 1, \ldots, k \right\}}$, 
and other values in $S$ are equal to zero.  
\end{definition}
A star partition graph has a topological structure similar to that of a clique partition graph.
The polytope structure of the clique partition graph has been investigated 
\citep{grotschel1990_clique-partition-polytope,bandelt1999_lift-facet-character-clique-partition,oosten2001_clique-partition-facets},
but we have not found related references on the polytope structure of the star
partition graph.
We find that facet inequalities for the star-partition polytope are simpler 
than those of the clique-partition polytope.
Consider the following polytopes:
\begin{equation}
\begin{aligned}
&P_{3,i,j}=\textrm{conv}\Set*{\widehat{S}}{ \widehat{S}
\textrm{ specifies a star partition in which the size of each star is at most }Q}, \\
&\widetilde{P}_{3,i,j}=\textrm{conv}\Set*{ \widehat{S} }{ \widehat{S}
\textrm{ specifies a star partition}},
\end{aligned}
\end{equation}
where $S=\Set*{f_{u,v,i,j}}{u,v\in\cV_{i,j},\;u>v}$ represents the set of $f_{u,v,i,j}$ variables
and $\widehat{S}$ denotes a value assignment to $S$.
Note that we have $P_{3,i,j}=\textrm{Proj}_{S}(\Psp_{2,i,j})$ by definition
and that $\widetilde{P}_{3,i,j}$ is defined by removing the star-size (size of a platoon)
constraints from $P_{3,i,j}$. The sets of valid inequalities for these polytopes 
obey the following relation:
$\{\textrm{valid ineqs. of }\widetilde{P}_{3,i,j}\}\subseteq
\{\textrm{valid ineqs of }P_{3,i,j}\}\subseteq
\{\textrm{valid ineqs. of }\Psp_{2,i,j}\}$.
Theorem~\ref{thm:facet-star-partition}, proven in Section~\ref{proof_of_star_part}, provides sets of valid inequalities 
for $\widetilde{P}_{3,i,j}$ and $P_{3,i,j}$.
\begin{theorem}\label{thm:facet-star-partition}
\emph{(a)} The polytope $\widetilde{P}_{3,i,j}$ can be represented by 
the following inequalities:
\begin{equation}\label{eqn:claim1}
\begin{aligned}
\widetilde{P}_{3,i,j}=\Set*{
\def\arraystretch{1.5}
\begin{array}{l}
f_{u,v,i,j},\\
\forall u,v\in\cV_{i,j},\\
u>v
\end{array}
}
{
\def\arraystretch{1.5}
\begin{array}{ll}
  \displaystyle \sum_{v\in\cV_{i,j}<v_{\emph{max}}}f_{v_{\emph{max}},v,i,j}\le 1, \\
  \displaystyle f_{u,v,i,j}+\sum_{w\in\cV_{i,j}<v}f_{v,w,i,j}\le 1 & \forall u,v\in\cV_{i,j}, u>v,\;v>v_{\emph{min}} \\
  \displaystyle 0\le f_{u,v,i,j}\le 1 & \forall u,v\in\cV_{i,j}, u>v
\end{array}
},
\end{aligned}
\end{equation} 
where $v_{\emph{min}}$ and $v_{\emph{max}}$ 
are the vehicles in $\cV_{i,j}$ with the smallest and largest index, respectively.

\emph{(b)} The following are facet defining for $P_{3,i,j}$:
\begin{equation}\label{eqn:sum_y<Q-1}
  \sum_{u,v\in U:u>v}f_{u,v,i,j}\le \lambda-1 \qquad \forall U\subset\cV_{i,j},\;|U|=\lambda+1.
\end{equation}
\end{theorem}

\section{Numerical Investigation}
\label{sec:num-study}
We now study the performance of the RSHM (Algorithm~\ref{alg:RSHM}), 
the \ref{opt:RDP}, and the \ref{opt:SPF}.
Section~\ref{sec:inst-gen} describes the considered CVPP instances. 
In Sections~\ref{sec:num-routing} and \ref{sec:num-sched},
we present the computational performance of solving the routing problem
and scheduling problem,
respectively. 
Section~\ref{sec:num-RSHM} presents the performance of the RSHM approach
on the CVPP instances and demonstrates the advantage of the RSHM over
solving the CVPP as a single MILP. Our code is available at the following.
\begin{center}
  \url{https://www.mcs.anl.gov/~jlarson/Platooning/}
\end{center}

\begin{wrapfigure}{r}{0.4\textwidth}
\centering
\vspace{-20pt}
\includegraphics[scale=0.4]{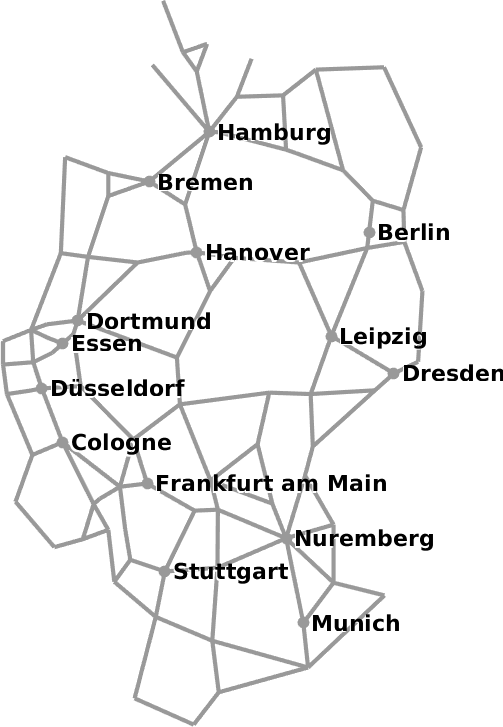}
\caption{Network topology of the Germany highway network system, noting the 14 largest cities.
\label{fig:germany}}
\vspace{-10pt}
\end{wrapfigure}
\subsection{Problem Generation}
\label{sec:inst-gen}
All numerical instances use the Germany highway network system in Figure~\ref{fig:germany}.
Three ingredients are involved in the generation of an instance: the number
of vehicles, the origin and destination nodes of each vehicle ($O_v,D_v$), and 
the start-and-destination time window ($T^O_v,T^D_v$) of each vehicle. 
We consider the number of vehicles to be in $\{50,100,150,200,250,400,800\}$.
For a given number of vehicles, we use two approaches to generate $O_v$ and $D_v$ for 
of each vehicle: a distributed model and a two-cluster model. 
In the distributed model,
most trips in the road network have origins and destinations that are near
major cities. From the 14 German cities with
more than 500,000 people, we drew (without replacement) two vertices. We then
chose an origin node $O_v$ randomly within a 50~km radius of the first vertex
and a destination node $D_v$ randomly within a 50~km radius of the second vertex. 
Since Germany is approximately 75\% urban/25\% rural, this procedure was performed for 75\%
of the vehicles considered. The remaining 25\% of the vehicles'
origin/destination pairs were drawn uniformly over the all possible nodes. 
Such an approach does not mirror reality exactly but generates a
collection of realistic trips for people largely traveling between major
destinations, while still accounting for trips with rural origins or rural
destinations (or both). In the two-cluster model, 
we created a list $L$ of distance values
of any two nodes in the network.
We randomly selected two hub nodes $H_1,H_2$ that satisfy 
the condition that the distance between $H_1$ and $H_2$ is greater than at least 70\% of the
elements in $L$. For each vehicle, we randomly selected a node within 
$r_0$ radius from the hub node $H_1$/$H_2$ as the origin/destination node of the vehicle,
where $r_0$ is a value that is greater than 20\% of the elements in $L$.

In both models, the origin time $T^O_v$ was selected uniformly throughout the
day, and the destination time $T^D_v$ was selected such that $T^D_v-T^O_v$ is twice 
the shortest-path traveling time from $O_v$ to $D_v$. Such a time window
means that a vehicle can have a waiting time equal to the time of travel. 
This waiting time can be utilized for determining a departure time in order to form platoons
with other vehicles. This waiting time calculation follows the intuition
that a vehicle will admit a longer waiting time if it is traveling a longer distance.
The platooning savings rates 
are set as $\sigma^l=0.02$ for a leading vehicle and $\sigma^f=0.1$
for a trailing vehicle in a platoon.
\replace{}{These values are safely within the respective saving ranges of
1--8\% and 7--16\% reported in the literature~\citep{eritco2016}.}

A numerical instance G100-0 represents, for example,
the CVPP instance containing 100 vehicles with the origin/destination nodes
generated by using the distributed model and with a random seed set to 0.
Similarly, a numerical instance GC200-1 represents
the CVPP instance containing 200 vehicles with the origin/destination nodes
generated by using the two-cluster model and the random seed set to 1.
The algorithms and MILP models are implemented in Python 3.7, and the MILP
models are solved by using Gurobi 8.1.

\subsection{Computational Performance of Solving the Routing Problem}
\label{sec:num-routing}
Given the origin and destination nodes of a vehicle, we can determine 
a set of possible edges it can traverse in an optimal route for the CVPP
based on the physical saving rate of forming platoons. We observe
that for any vehicle $v$,
the length of an optimal route of $v$ in the CVPP
cannot be greater than $d(O_v,D_v)/(1-\sigma^f)$,
where $d(O_v,D_v)$ is the length of the shortest path between $O_v$ and $D_v$;
otherwise vehicle $v$ choosing its shortest path without platooning
uses less fuel. This observation is proved in
Proposition~\ref{prop:bound-on-path}. 
The bound also provides a condition to select a set of edges for a vehicle $v$
that possibly can be involved in an optimal route, namely,
\begin{equation}\label{eq:E_v}
\cE_v=\Set*{(i,j)\in\cE}{d(O_v,i)+d(i,j)+d(j,D_v)\le d(O_v,D_v)/(1-\sigma^f)}.
\end{equation}

Table~\ref{tab:routing} gives the computational performance of solving
\ref{opt:RDP}-(1) for 40 CVPP instances. 
All valid inequalities in \eqref{eqn:P_RDP_ij_full} that give a complete description
of $P^{\textrm{RDP}}_{i,j}$ are incorporated in the definition of the \ref{opt:RDP} being solved.
All the instances are considered solved when the relative optimality
gap is less than $10^{-4}$. 
The results show that with the valid inequalities, the routing problem
is effectively solved before the start of the branch-and-bound procedure because the 
number of branch-and-bound nodes is either 0 or 1 for all the instances.
Specifically, seven instances are solved during Gurobi's preprocessing, 
and the remaining instances are solved during root relaxation. 
For large instances (e.g., with 800 vehicles), however, loading the model
can take longer.
\begin{table}
\centering
\caption{Computational performance of solving \ref{opt:RDP}-(1)
for 40 CVPP instances. All instances are stopped when they reach a relative
optimality gap of $10^{-4}$.
CPU(s) is the computational time for solving \ref{opt:RDP}-(1),
not including the problem loading time and solution extraction time.
Fuel$_0$ is the total fuel consumed
when each vehicle takes its shortest path without platooning,
Nodes is the number of branch-and-bound nodes needed to solve the given
instance. DetourVs is the number of vehicles for which the \ref{opt:RDP} route is not the shortest path.
 }\label{tab:routing}
{\scriptsize
\begin{tabular}{cccccc|cccccc}
\hline\hline
Inst.	&	Fuel$_0$	&	Obj	&	CPU(s)	&	Nodes	&	DetourVs &
Inst.	&	Fuel$_0$	&	Obj	&	CPU(s)	&	Nodes	&	DetourVs
\\
\hline
G100-0	&	2217.7	&	2049.0	&	0.25	&	1	&	14	&	GC100-0	&	2684.6	&	2438.5	&	0.51	&	1	&	6	\\
G100-1	&	2164.1	&	2003.3	&	0.29	&	0	&	12	&	GC100-1	&	3859.0	&	3503.1	&	0.79	&	1	&	1	\\
G100-2	&	2384.4	&	2199.9	&	0.26	&	1	&	7	&	GC100-2	&	2699.1	&	2458.8	&	0.47	&	1	&	3	\\
G100-3	&	2191.2	&	2022.3	&	0.32	&	1	&	5	&	GC100-3	&	3695.1	&	3355.3	&	0.73	&	0	&	1	\\
G100-4	&	2384.3	&	2201.1	&	0.30	&	1	&	9	&	GC100-4	&	2903.8	&	2638.6	&	0.52	&	1	&	0	\\
G200-0	&	4692.0	&	4281.6	&	0.74	&	0	&	12	&	GC200-0	&	5461.0	&	4942.4	&	1.36	&	1	&	2	\\
G200-1	&	4603.4	&	4206.8	&	0.65	&	1	&	14	&	GC200-1	&	7865.4	&	7110.0	&	2.42	&	1	&	2	\\
G200-2	&	4693.6	&	4290.7	&	0.73	&	1	&	8	&	GC200-2	&	5392.5	&	4885.7	&	1.43	&	1	&	3	\\
G200-3	&	4448.0	&	4062.7	&	0.64	&	1	&	4	&	GC200-3	&	7303.1	&	6604.5	&	2.04	&	0	&	1	\\
G200-4	&	4730.0	&	4318.2	&	0.72	&	1	&	4	&	GC200-4	&	5578.4	&	5047.2	&	1.33	&	1	&	3	\\
G400-0	&	9067.4	&	8232.2	&	1.85	&	1	&	6	&	GC400-0	&	10870.5	&	9812.3	&	3.87	&	1	&	3	\\
G400-1	&	9100.1	&	8260.7	&	2.09	&	1	&	6	&	GC400-1	&	15572.0	&	14046.7	&	6.87	&	1	&	5	\\
G400-2	&	9804.6	&	8898.3	&	1.93	&	1	&	15	&	GC400-2	&	10793.9	&	9749.9	&	4.41	&	0	&	4	\\
G400-3	&	9240.1	&	8385.7	&	2.00	&	0	&	3	&	GC400-3	&	14675.0	&	13240.5	&	6.57	&	1	&	1	\\
G400-4	&	9509.2	&	8627.1	&	2.02	&	1	&	9	&	GC400-4	&	11333.9	&	10229.1	&	4.10	&	0	&	7	\\
G800-0	&	19111.3	&	17278.9	&	5.57	&	1	&	3	&	GC800-0	&	21657.4	&	19522.2	&	11.20	&	1	&	2	\\
G800-1	&	18608.3	&	16824.1	&	6.24	&	1	&	9	&	GC800-1	&	31044.2	&	27972.8	&	18.22	&	1	&	6	\\
G800-2	&	19142.6	&	17308.9	&	6.40	&	1	&	1	&	GC800-2	&	21416.2	&	19310.2	&	10.88	&	1	&	1	\\
G800-3	&	18708.6	&	16913.9	&	6.04	&	1	&	2	&	GC800-3	&	29690.0	&	26754.6	&	14.44	&	1	&	7	\\
G800-4	&	19257.2	&	17408.9	&	6.16	&	1	&	2	&	GC800-4	&	23034.1	&	20761.3	&	11.29	&	1	&	9	\\
\hline
\end{tabular}
}
\end{table}

\subsection{Computational Performance of Solving the Scheduling Problem}
\label{sec:num-sched}
Once the \ref{opt:RDP} routes are obtained, we solve the \ref{opt:SPF}
to decide an optimal schedule of each vehicle for the given \ref{opt:RDP} routes.
Before solving the \ref{opt:SPF}, we can apply the edge-contraction
procedure to reduce the problem size. 
Table~\ref{tab:edge-contraction} of Section~\ref{app:add-num-res}
compares the problem size
(measured by the number of variables and constraints
and the loading time of the problem)
before and after applying the edge-contraction procedure.
Indeed, we see that the problem size is reduced significantly after applying the 
edge-contraction procedure:
the number of variables decreases 20.6\%--74.7\%, 
the number of constraints decreases 20.7\%--85.6\%,
and the amount of loading time decreases 66.7\%--99.6\%.
The results in the remainder of this subsection are 
after the edge-contraction preprocess.

We now show the power of adding 
(i) the disjunctive inequalities \eqref{eqn:disj-valid-ineq} based on Theorem~\ref{thm:disj-cut}
and (ii) the platooning inequalities \eqref{eqn:claim1} based on Theorem~\ref{thm:facet-star-partition}(a).
Note that the platooning inequalities can be added directly to the \ref{opt:SPF}
while the disjunctive inequalities are generated to cut a target fractional solution,
but only  
of the root-relaxation linear program. 
We do not generate
disjunctive inequalities at branch-and-bound nodes.
First, we study whether the added cuts improve
the root-relaxation bound. Second, we highlight that adding cuts can lead to a better
objective value or a better MILP gap in a certain amount of time at the branch-and-bound
process. 

Table~\ref{tab:LP_bd} 
shows the differences in the root-relaxation bound after
adding disjunctive cuts and platooning cuts for 40 \ref{opt:RDP} solutions described in
the preceding subsections.
The table shows that both the disjunctive cuts and platooning cuts
improve the linear relaxation bound. 
For most instances the platooning cuts are more powerful than the disjunctive cuts,
in the sense that adding platooning cuts leads to a more significant improvement 
(3.51\% on average) 
on the linear relaxation bound than just adding disjunctive cuts 
(0.66\% on average).  
The different levels of improvement in the relaxation bound with respect
to the two families of inequalities may be because there are more 
platooning cuts than disjunctive cuts. 
The platooning cuts \eqref{eqn:claim1} are easier to add because they have a simple
form and they can be added directly to the original problem.
Adding the disjunctive cuts \eqref{eqn:disj-valid-ineq}  
requires solving the linear program \eqref{opt:disj-cut-gen}, and
generating these disjunctive cuts can require between 0.1 and 13 
seconds depending on the problem size. 

Table~\ref{tab:mipgap} shows the differences in the objective value and the relative MIP gap
before and after adding the cuts. 
We use SP-plain and SP-cuts
to denote the scheduling problem without adding the cuts and the scheduling problem
after adding the cuts (including the disjunctive cuts and platooning cuts), respectively.
We first compare SP-plain and SP-cuts at the one-hour time limit.
Table~\ref{tab:mipgap} shows that SP-plain and SP-cuts both can solve 20 instances to optimality,
but the 20 instances they can solve are not exactly the same.
For problem GC100-0, SP-cuts solves it in 2084 seconds but SP-plain
does not solve it to optimality before the one-hour time limit. For problem G150-3, SP-plain
solves it in 3318 seconds but SP-cuts cannot solve it to optimality before the one-hour
time limit. Both SP-plain and SP-cuts perform well in solving instances
of G50, G100, and GC50: all but one of these instances are solved to optimality in within
25 seconds.
For the remaining 20 instances that cannot be solved by either of the two approaches, 
SP-cuts finds a better objective value in 14 instances, 
SP-plain finds a better objective value in 4 instances,
and the two approaches find almost the same objective value in 2 instances. 
The average (over the unsolved 20 instances) optimality gap reduction 
(gap of SP-plain $-$ gap of SP-cuts) 
due to the added cuts is 1.95\% at the one-hour limit.  

We next compare the two approaches at 5 minutes, 10 minutes, and 30 minutes
of the problem-solving process, respectively.
At the 5-minute time limit, SP-cuts finds a better objective value in 14 instances
(9 are significantly better),
and SP-plain finds a better objective value in 7 instances
(1 of them is significantly better).
At the 10-minute time limit, SP-cuts finds a better objective value in 13 instances
(8 are significantly better),
and SP-plain finds a better (but not significantly better) objective value in 7 instances.
At the 30-minute time limit, SP-cuts finds a better objective value in 11 instances
(6 are significantly better),
and SP-plain finds a better objective value in 8 instances
(1 is significantly better).
In terms of optimality gap reduction due to added cuts, 
the average gap reduction is 2.88\%, 2.80\%, and 2.53\%
at the 5-minute, 10-minute, and 30-minute time limits, respectively,
with the largest gap reduction being 11.75\%.

Table~\ref{tab:mipgap-final} shows that for numerical instances that 
cannot be solved to optimality by either SP-plain or 
SP-cuts, the number of branch-and-bound nodes
processed by SP-cuts is 59.5\% of the  
nodes processed by SP-plain on average.
In particular, for instances GC200-1 and GC200-2,
SP-cuts is not able to process any branch-and-bound nodes.
This reveals a drawback of adding too many valid inequalities to the problem:
it can increase the time required to solve the linear relaxation at each node, possibly
reducing the total number of nodes explored.
\replace{}{More specifically, having too many valid inequalities 
increases the size of linear program relaxation at each branch-and-bound node 
and makes solving the node relaxation more time consuming. Most of these
valid inequalities are not binding in later branch-and-bound iterations,
but their presence can increase the problems' size in memory.
This drawback has been recognized in the literature 
\citep{botton2013_benders-decomp-for-hop-constrained-surv-netwk-des}.}
As a consequence, in some instances, 
SP-cuts can give a worse objective value than
SP-plain gives.

\subsection{Computational Performance of the RSHM} \label{sec:num-RSHM}

\begin{table}
\centering
\caption{\scriptsize
Results from applying the RSHM (with threshold $I = 3$) to 15 problem instances. 
Fuel Cost is 
the best total fuel cost identified by the RSHM and
Saving Rate is 
the savings
compared with the total fuel cost of shortest-path driving without platooning.
Iters, Termination, and CPU represent the number of iterations 
completed by the algorithm, the termination status upon time limit, and the actual
computational time in solving an instance, respectively. RelDev represents the 
value $\sigma/\mu$, where $\mu$ and $\sigma$ are the mean and variance of
the sequence of total fuel costs found by the algorithm in iterations, respectively.
 }\label{tab:RSHM}
{\scriptsize
\begin{tabular}{ccccccc}
\hline\hline
Instance	&	Fuel Cost	&	Saving Rate(\%)	&	RelDev(\%)	&	Iters	&	Termination	&	CPU(s)	\\
\hline
G50-0	&	1241.02	&	2.35	&	0.13	&	31	&	Y	&	71	\\
G50-1	&	1130.42	&	2.51	&	0.23	&	16	&	Y	&	25	\\
G50-2	&	1116.49	&	1.76	&	0.14	&	12	&	Y	&	20	\\
G50-3	&	1039.94	&	2.20	&	0.28	&	12	&	Y	&	18	\\
G50-4	&	1249.44	&	2.91	&	0.25	&	8	&	Y	&	20	\\
G100-0	&	2146.82	&	3.19	&	0.20	&	35	&	Y	&	376	\\
G100-1	&	2084.34	&	3.68	&	0.14	&	108	&	Y	&	1482	\\
G100-2	&	2285.91	&	4.13	&	0.15	&	82	&	Y	&	1043	\\
G100-3	&	2124.65	&	3.04	&	0.11	&	198	&	N	&	3621	\\
G100-4	&	2306.90	&	3.24	&	0.14	&	161	&	N	&	3600	\\
G150-0	&	3292.15	&	4.12	&	0.15	&	17	&	N	&	4062	\\
G150-1	&	3166.87	&	4.32	&	0.16	&	31	&	Y	&	1402	\\
G150-2	&	3294.10	&	4.43	&	0.15	&	20	&	N	&	3621	\\
G150-3	&	3328.20	&	4.10	&	0.11	&	8	&	N	&	3630	\\
G150-4	&	3401.21	&	4.35	&	0.27	&	9	&	N	&	3637	\\
\hline
\end{tabular}
}
\end{table}

We demonstrate the performance of the RSHM on 15 CVPP instances
corresponding to 50-, 100- and 150-vehicle systems. 
We set $I=3$ in our experiments,
which means the algorithm terminates after identifying
a specific route assignment $\cup_{v\in\cV}\cR_v$
$3$ times. In the computational setting, the valid inequalities  
identified in Section~\ref{sec:sol-route} are added to solve each \ref{opt:RDP}
and \ref{opt:SPF} instance
in every RSHM iteration.
We allow 10 minutes
for each \ref{opt:RDP} and \ref{opt:SPF} instance and limit the RSHM to
one hour in all experiments.
Based on the best fuel cost identified by the algorithm, 
the total fuel savings ranges from 2\% to 4.5\%
compared with the free-driving case (every vehicle drives its shortest path without platooning),
and the savings are more in systems with more vehicles.
As expected, the RSHM identifies different routes in different iterations,
and different route assignments lead to different total fuel costs after imposing the time constraints.
The results are summarized in Table~\ref{tab:RSHM} where 
the total fuel costs identified at different iterations are plotted
for \replace{six numerical instances}{an example problem instance}. 
Although there exists a deviation in total fuel costs
across iterations, the deviation (RelDev in Table~\ref{tab:RSHM}) is one magnitude
smaller compared with the aggregated fuel saving rate in all 15 instances. 

\begin{figure}[b!]
\centering
\includegraphics[trim={10 0 45 10}, clip, width=0.47\linewidth]{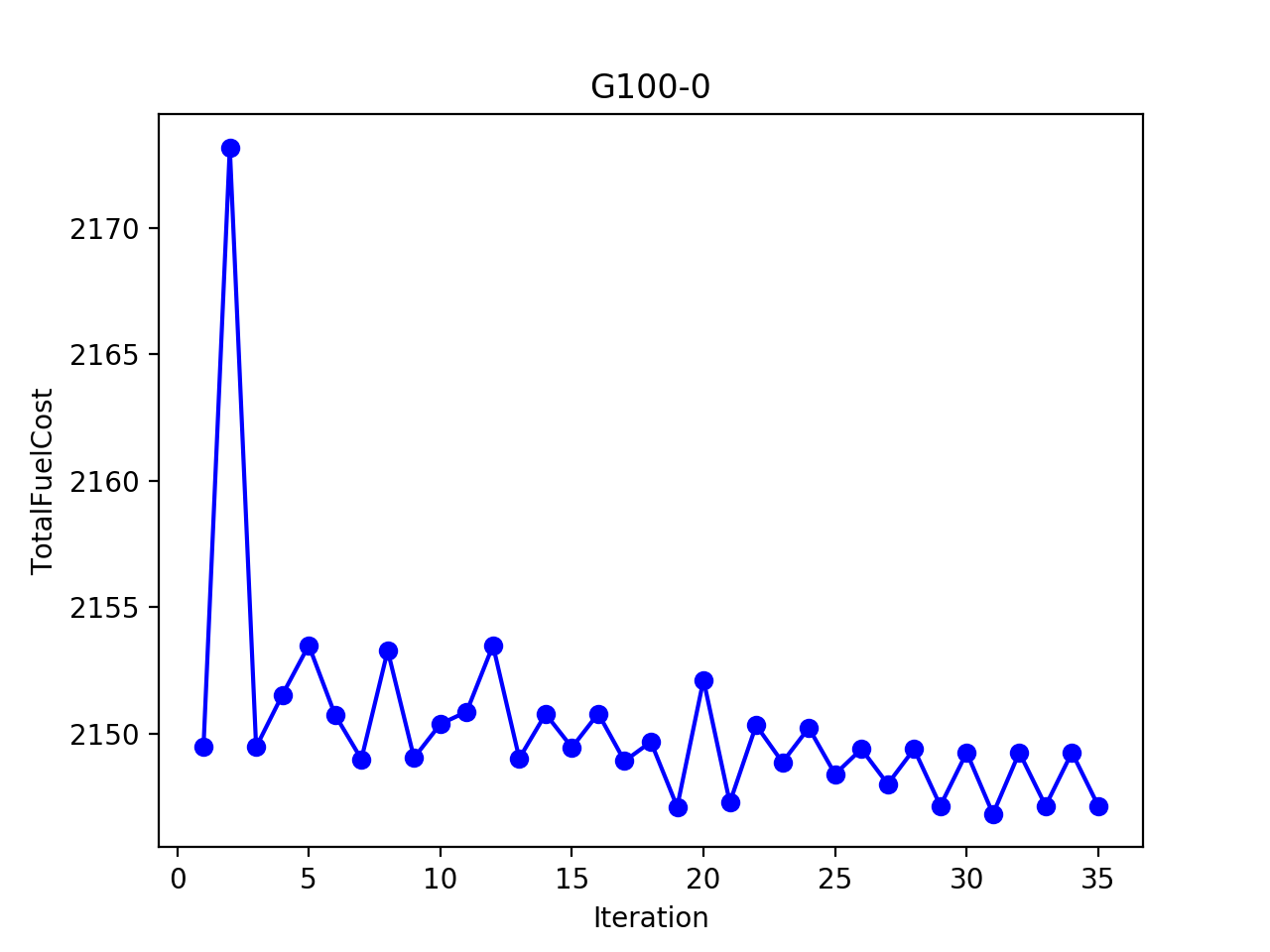} 
\hfill
\includegraphics[trim={10 0 45 10}, clip, width=0.47\linewidth]{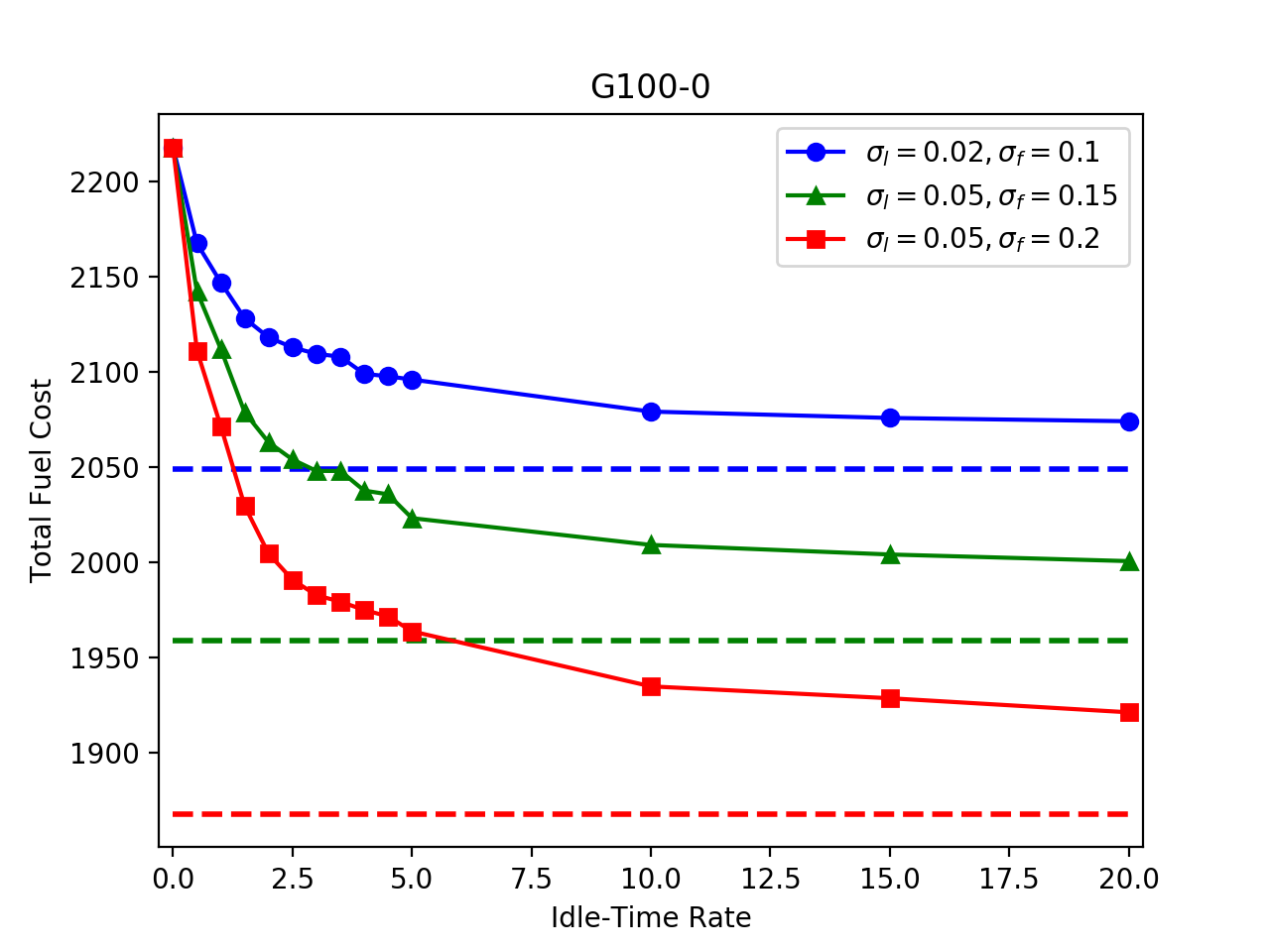} 
\caption{left) True total fuel cost found at every iteration in the RSHM for an example numerical instance.\\
right) Total fuel cost versus time-flexibility rate under three platooning savings scenarios.
}
\label{fig:RSHM-iters-plot_new}
\end{figure}


Figure~\ref{fig:RSHM-iters-plot_new} shows a zigzag pattern in the plot\replace{s}{} of the total fuel cost
versus iterations. This pattern reveals the essential idea of the algorithm: use an optimistic
estimation of fuel cost at edges that are not traveled by a vehicle and use 
realistic mean fuel cost for edges that are traveled by a vehicle in the previous main iteration.
Adjusting fuel cost in this manner encourages vehicles to try new routes that are potentially 
beneficial. But when the fuel cost based on new routes is realized (after solving the \ref{opt:SPF}),
the new routes may lead to a higher total fuel cost than before 
(from a valley to a peak in the plot). After some iterations, the algorithm can identify 
some competitive route assignments and realize the total fuel costs corresponding to
these assignments. The overall trend, then, is that the total fuel costs are decreasing
in a zigzag pattern as the number of iterations increases. 
\replace{(e.g., G50-1, G100-0, G100-1, G150-1)}{ See further examples in Section~\ref{app:add-num-res}.}

Table~\ref{tab:joint-MILP} also compares the performance of the RSHM with solving
a joint MILP formulation of the routing and scheduling problems as in 
\cite{luo2018-veh-platn-mult-speeds}. The comparison shows that the RSHM is much more efficient.
Specifically, we find that joint problem instances of 100- and 150-vehicle systems cannot be
loaded by the solver in an hour. For numerical instances of 50 vehicles, the solver takes about
one minute to load the problem and leaves a considerable optimality gap after solving the instances for four hours.
The best fuel cost found in four hours by directly solving the joint MILP is also less competitive 
than the RSHM.


We investigated the total fuel cost as a function of the time-flexibility rate of each vehicle.
The time-flexibility rate is defined as $r_v=\frac{T^D_v-T^O_v}{T^{sp}_v}-1$,
where $T^{sp}_v$ is the time cost of traversing the shortest path of vehicle $v$.
In all preceding sections, the time-flexibility rate $r_v$ is set to be $r=1$ for 
all vehicles. The right of Figure~\ref{fig:RSHM-iters-plot_new} and Figure~\ref{fig:fuel-vs-idleTime} show how the RSHM total fuel cost
changes with the time-flexibility rate $r$ under three different settings of the saving
parameters $\sigma_l$, and $\sigma_f$. 
The general trend seen from these figures
is that the fuel cost decreases as $r$ increases, because the vehicles allow
more flexibility to facilitate platoon formation.
As $r\to\infty$, the fuel cost approaches the lower bound that equals
the fuel cost in the absence of time constraints. The fuel cost decreases faster
in the range $r\in [0,5]$ compared with that in the range $r\in [10,20]$.  
These results indicate that a relatively small amount of waiting time 
can realize a majority of platooning opportunities.
Comparing the results for the instances generated by using the distributed model (G50-0, G100-0)
versus the instances generated by using the two-cluster model (GC50-0, GC100-0),
we observe that the GC-instances have a smaller gap between the fuel cost 
at $r=20$ and the lower-bound limit compared with the G-instances.
The reason is that in the GC-instances, the origins and destinations of vehicles 
are more concentrated and hence it is easier to coordinate their schedules in order 
to achieve the lower bound of the fuel cost
when there is enough time flexibility. Similarly,
comparing the results for the 50-vehicle instances
(G50-0, GC50-0) with the 100-vehicle instances (G100-0, GC100-0), 
we see that the schedules of 50-vehicle systems are easier to coordinate
in order to approach the fuel cost limit
when there is enough time flexibility.

\section{Concluding Remarks}
\replace{}{We note that the time-discretization 
approximation and time-extended network approach proposed by \cite{Abdolmaleki2019} 
can be incorporated in the RSHM framework to yield an alternative way of
formulating the scheduling problem. We believe such an approach will admit 
more easily solved problem instances (assuming a mild regularity condition on the routes).
This is our next step of research.} 
The valid inequalities developed in this work can
further speed up the computational time or reduce the optimality gap for routing
and scheduling problems. Although the RSHM is developed for the deterministic
CVPP, it can be directly applied
to a generalized framework where the impact of traffic on the traveling time is incorporated.
As an extension of this work, a robust optimization framework can be 
established to incorporate the impact of traffic, for example. 

\section*{Acknowledgments}
This material is based upon work supported by the
U.S.\ Department of Energy, Office of Science, under contract number DE-AC02-06CH11357.
\replace{}{We are grateful for the comments from three anonymous reviewers that
greatly improved an early version of this manuscript.}

\def\bibfont{\scriptsize}
\bibliographystyle{informs2014} 
\bibliography{../../../bibs/platooning2018}

\vfil
\scriptsize
\framebox{\parbox{\textwidth}{
The submitted manuscript has been created by UChicago Argonne, LLC, Operator of Argonne National Laboratory (“Argonne”).
Argonne, a U.S. Department of Energy Office of Science laboratory, is operated under Contract No. DE-AC02-06CH11357.
The U.S. Government retains for itself, and others acting on its behalf, a paid-up nonexclusive, irrevocable worldwide
license in said article to reproduce, prepare derivative works, distribute copies to the public, and perform publicly
and display publicly, by or on behalf of the Government.  The Department of Energy will provide public access to these
results of federally sponsored research in accordance with the DOE Public Access Plan.
\url{http://energy.gov/downloads/doe-public-access-plan}
}}
\normalsize

\appendix
\section{Supplement of Section~\ref{sec:edge-contraction}\label{sec:edge_contraction_proof}}
\begin{proof}{Proof of Proposition~\ref{prop:edge-contraction}}
\replace{}{
Let $\cV^{\prime}$ be the set of
vehicles sharing the path $L$.
Suppose in an optimal schedule
and platooning decision the set $S_1$ of platoons
on edge $(n_i,n_{i+1})$ is different from the set $S_2$ of platoons
on edge $(n_j,n_{j+1})$. Notice that the
schedule of all vehicles in $\cV^{\prime}$ ensure the platooning 
decisions $S_1$ and $S_2$ on every edge of $L$ are valid.
Furthermore, it contradicts to the optimality condition
if imposing the platoon decision
induced by $S_1$ on edge $(n_j,n_{j+1})$ incurs
smaller (greater) fuel cost on $(n_j,n_{j+1})$ than $S_2$. 
It follows that imposing $S_1$ and $S_2$ on $(n_i,n_{i+1})$
lead to the same fuel cost, and it also applies to $(n_j,n_{j+1})$.
Since $i$ and $j$ are arbitrary, the result is shown.}
\end{proof}

\section{Additional Numerical Results}
\label{app:add-num-res}
\begin{table}
\centering
\caption{Model size comparison before and after edge contraction for the \ref{opt:SPF}.
The columns Vars, Constrs and LoadT represent the number of variables, constraints
and loading time of the model, respectively.}\label{tab:edge-contraction}
{\scriptsize
\begin{tabular}{c|ccc|ccc}
\hline\hline
  & \multicolumn{3}{c|}{Before edge contraction}  &  \multicolumn{3}{c}{After edge contraction} \\
\hline
Instance	&	Vars	&	Constrs	&	LoadT(s)	&	Vars 	&	Constrs	&	LoadT(s)	\\
\hline
G100-0	&	1.35E+04	&	2.81E+04	&	4.8	&	3.80E+03	&	7.98E+03	&	0.5	\\
G100-1	&	1.18E+04	&	2.38E+04	&	3.3	&	3.32E+03	&	6.71E+03	&	0.3	\\
G100-2	&	1.54E+04	&	3.38E+04	&	6.9	&	3.55E+03	&	7.46E+03	&	0.4	\\
G100-3	&	1.27E+04	&	2.61E+04	&	4.4	&	3.32E+03	&	6.73E+03	&	0.4	\\
G100-4	&	1.42E+04	&	2.96E+04	&	5.0	&	4.08E+03	&	8.68E+03	&	0.5	\\
G200-0	&	4.28E+04	&	1.12E+05	&	238.8	&	1.42E+04	&	3.66E+04	&	7.0	\\
G200-1	&	4.26E+04	&	1.12E+05	&	159.8	&	1.34E+04	&	3.46E+04	&	6.2	\\
G200-2	&	4.12E+04	&	1.06E+05	&	209.1	&	1.40E+04	&	3.60E+04	&	6.7	\\
G200-3	&	3.79E+04	&	9.48E+04	&	64.3	&	1.19E+04	&	2.95E+04	&	4.9	\\
G200-4	&	4.51E+04	&	1.19E+05	&	127.1	&	1.58E+04	&	4.21E+04	&	8.9	\\
G400-0	&	1.31E+05	&	3.86E+05	&	3095.4	&	5.26E+04	&	1.54E+05	&	344.2	\\
G400-1	&	1.32E+05	&	3.84E+05	&	3781.0	&	5.40E+04	&	1.57E+05	&	416.4	\\
G400-2	&	1.57E+05	&	4.71E+05	&	5195.4	&	6.65E+04	&	1.99E+05	&	589.6	\\
G400-3	&	1.37E+05	&	4.03E+05	&	3083.9	&	5.16E+04	&	1.51E+05	&	358.4	\\
G400-4	&	1.49E+05	&	4.45E+05	&	4056.7	&	5.73E+04	&	1.70E+05	&	451.1	\\
G800-0	&	5.32E+05	&	1.71E+06	&	$>$10hrs	&	2.57E+05	&	8.25E+05	&	$>$10hrs	\\
G800-1	&	5.12E+05	&	1.64E+06	&	$>$10hrs	&	2.51E+05	&	8.05E+05	&	$>$10hrs	\\
G800-2	&	5.29E+05	&	1.70E+06	&	$>$10hrs	&	2.67E+05	&	8.57E+05	&	$>$10hrs	\\
G800-3	&	5.06E+05	&	1.62E+06	&	$>$10hrs	&	2.39E+05	&	7.63E+05	&	$>$10hrs	\\
G800-4	&	5.47E+05	&	1.77E+06	&	$>$10hrs	&	2.55E+05	&	8.21E+05	&	$>$10hrs	\\
\hline
GC100-0	&	4.35E+04	&	1.21E+05	&	296.6	&	1.45E+04	&	3.92E+04	&	8.4	\\
GC100-1	&	1.40E+05	&	3.93E+05	&	4462.1	&	2.14E+04	&	5.69E+04	&	19.5	\\
GC100-2	&	5.12E+04	&	1.44E+05	&	377.3	&	2.43E+04	&	6.83E+04	&	33.1	\\
GC100-3	&	6.64E+04	&	1.75E+05	&	835.2	&	1.42E+04	&	3.57E+04	&	7.2	\\
GC100-4	&	6.14E+04	&	1.73E+05	&	675.8	&	2.61E+04	&	7.27E+04	&	41.2	\\
GC200-0	&	1.54E+05	&	4.55E+05	&	4796.3	&	5.54E+04	&	1.60E+05	&	411.7	\\
GC200-1	&	5.25E+05	&	1.52E+06	&	$>$10hrs	&	9.74E+04	&	2.74E+05	&	1916.4	\\
GC200-2	&	1.96E+05	&	5.90E+05	&	8640.8	&	1.19E+05	&	3.56E+05	&	2878.9	\\
GC200-3	&	2.33E+05	&	6.56E+05	&	12083.9	&	5.95E+04	&	1.63E+05	&	557.0	\\
GC200-4	&	2.07E+05	&	6.22E+05	&	8869.0	&	1.15E+05	&	3.43E+05	&	2254.0	\\
GC400-0	&	5.91E+05	&	1.83E+06	&	$>$10hrs	&	2.36E+05	&	7.19E+05	&	11991.8	\\
GC400-1	&	2.08E+06	&	6.14E+06	&	$>$10hrs	&	3.92E+05	&	1.14E+06	&	36138.5	\\
GC400-2	&	7.47E+05	&	2.32E+06	&	$>$10hrs	&	5.51E+05	&	1.71E+06	&	15907.7	\\
GC400-3	&	9.33E+05	&	2.73E+06	&	$>$10hrs	&	2.63E+05	&	7.58E+05	&	$>$10hrs	\\
GC400-4	&	8.43E+05	&	2.61E+06	&	$>$10hrs	&	5.44E+05	&	1.67E+06	&	$>$10hrs	\\
GC800-0	&	2.29E+06	&	7.25E+06	&	$>$10hrs	&	9.69E+05	&	3.03E+06	&	$>$10hrs	\\
GC800-1	&	8.08E+06	&	2.41E+07	&	$>$10hrs	&	1.55E+06	&	4.57E+06	&	$>$10hrs	\\
GC800-2	&	2.84E+06	&	9.02E+06	&	$>$10hrs	&	2.26E+06	&	7.16E+06	&	$>$10hrs	\\
GC800-3	&	3.69E+06	&	1.10E+07	&	$>$10hrs	&	1.07E+06	&	3.16E+06	&	$>$10hrs	\\
GC800-4	&	3.13E+06	&	9.76E+06	&	$>$10hrs	&	2.13E+06	&	6.65E+06	&	$>$10hrs	\\
\hline
\end{tabular}
}
\end{table}

\begin{table}
\centering
\caption{Improvement on the root relaxation bound after 
adding disjunctive cuts and platooning cuts for the \ref{opt:SPF}.
The columns LPbd0, LPbd1 and LPbd2 correspond to the linear relaxation
bounds when (1) no cuts are added; (2) only disjunctive
cuts are added; (3) both disjunctive cuts and platooning cuts are added.
TimeDisjCut gives the computational time of generating all disjunctive cuts
by solving the linear programs \eqref{opt:disj-cut-gen}.
DisjCuts and Platcuts are the number of disjunctive cuts and platooning cuts
added to the problem, respectively. IMP1 and IMP2 are the improvement 
corresponding to adding the disjunctive cuts and adding platooning cuts, respectively.
They are defined as $\textrm{IMP1}=(\textrm{LPbd0}-\textrm{LPbd1})/\textrm{LPbd0}$,
and $\textrm{IMP2}=(\textrm{LPbd1}-\textrm{LPbd2})/\textrm{LPbd1}$.
}\label{tab:LP_bd}
{\scriptsize
\begin{tabular}{ccccccccc}
\hline\hline
Instance	&	LPbd0	&	LPbd1	&	LPbd2	&	TimeDisjCut	(s)&	DisjCuts	&	PlatCuts	&	IMP1(\%)	&	IMP2(\%)	\\
\hline
G50-0	&	34.08	&	33.34	&	32.55	&	0.16	&	32	&	348	&	2.15	&	2.33	\\
G50-1	&	32.88	&	32.34	&	31.57	&	0.08	&	23	&	381	&	1.63	&	2.37	\\
G50-2	&	24.74	&	24.40	&	23.10	&	0.11	&	27	&	327	&	1.36	&	5.25	\\
G50-3	&	27.64	&	27.25	&	27.23	&	0.26	&	30	&	315	&	1.40	&	0.10	\\
G50-4	&	41.24	&	40.42	&	38.38	&	0.17	&	30	&	265	&	1.99	&	4.93	\\
G100-0	&	82.32	&	81.36	&	78.81	&	0.66	&	74	&	1637	&	1.16	&	3.10	\\
G100-1	&	85.56	&	85.08	&	84.14	&	0.36	&	52	&	1265	&	0.55	&	1.10	\\
G100-2	&	109.96	&	108.96	&	105.98	&	0.31	&	67	&	1506	&	0.91	&	2.70	\\
G100-3	&	79.41	&	78.29	&	73.73	&	0.32	&	64	&	1333	&	1.41	&	5.75	\\
G100-4	&	90.51	&	89.17	&	84.44	&	0.39	&	67	&	1763	&	1.48	&	5.22	\\
G150-0	&	168.42	&	166.84	&	158.73	&	0.80	&	122	&	3642	&	0.94	&	4.82	\\
G150-1	&	161.74	&	160.36	&	153.06	&	0.73	&	109	&	3460	&	0.85	&	4.51	\\
G150-2	&	176.47	&	175.04	&	169.90	&	1.02	&	144	&	3895	&	0.81	&	2.91	\\
G150-3	&	171.90	&	170.36	&	163.82	&	0.92	&	126	&	3841	&	0.90	&	3.80	\\
G150-4	&	178.94	&	177.63	&	173.35	&	1.17	&	138	&	4290	&	0.74	&	2.39	\\
G200-0	&	287.17	&	285.31	&	273.58	&	2.37	&	248	&	8327	&	0.65	&	4.09	\\
G200-1	&	270.17	&	268.32	&	254.34	&	1.96	&	273	&	7900	&	0.68	&	5.18	\\
G200-2	&	266.94	&	265.84	&	257.84	&	2.31	&	159	&	8106	&	0.41	&	3.00	\\
G200-3	&	244.44	&	242.24	&	233.65	&	1.97	&	276	&	6685	&	0.90	&	3.51	\\
G200-4	&	271.52	&	270.55	&	256.23	&	2.35	&	181	&	9749	&	0.36	&	5.27	\\
\hline
GC50-0	&	64.63	&	63.78	&	60.16	&	0.49	&	70	&	1714	&	1.33	&	5.59	\\
GC50-1	&	127.10	&	126.46	&	121.17	&	0.49	&	87	&	2081	&	0.51	&	4.16	\\
GC50-2	&	61.51	&	61.03	&	56.94	&	0.68	&	77	&	2515	&	0.79	&	6.65	\\
GC50-3	&	118.47	&	118.41	&	114.82	&	0.17	&	21	&	1621	&	0.06	&	3.03	\\
GC50-4	&	79.20	&	78.94	&	75.67	&	0.56	&	77	&	3031	&	0.32	&	4.14	\\
GC100-0	&	172.35	&	171.91	&	162.52	&	1.33	&	155	&	10474	&	0.25	&	5.45	\\
GC100-1	&	309.81	&	309.59	&	304.35	&	1.81	&	157	&	17284	&	0.07	&	1.69	\\
GC100-2	&	177.76	&	177.48	&	169.94	&	2.33	&	170	&	19301	&	0.16	&	4.24	\\
GC100-3	&	293.90	&	293.75	&	283.32	&	1.30	&	128	&	10243	&	0.05	&	3.55	\\
GC100-4	&	207.51	&	207.17	&	201.58	&	3.76	&	391	&	20982	&	0.17	&	2.69	\\
GC150-0	&	283.61	&	282.70	&	267.79	&	3.75	&	286	&	25695	&	0.32	&	5.26	\\
GC150-1	&	493.02	&	492.89	&	485.04	&	6.14	&	265	&	41819	&	0.03	&	1.59	\\
GC150-2	&	305.55	&	305.08	&	296.86	&	9.24	&	399	&	52827	&	0.16	&	2.69	\\
GC150-3	&	458.36	&	458.10	&	445.96	&	4.37	&	322	&	25150	&	0.06	&	2.65	\\
GC150-4	&	321.23	&	320.59	&	311.97	&	7.95	&	506	&	53947	&	0.20	&	2.68	\\
GC200-0	&	410.29	&	409.55	&	387.90	&	5.77	&	371	&	45742	&	0.18	&	5.28	\\
GC200-1	&	677.99	&	677.81	&	671.74	&	11.07	&	467	&	86912	&	0.03	&	0.90	\\
GC200-2	&	418.79	&	418.34	&	409.41	&	18.36	&	578	&	105828	&	0.11	&	2.13	\\
GC200-3	&	615.50	&	614.99	&	605.03	&	7.67	&	406	&	49678	&	0.08	&	1.62	\\
GC200-4	&	442.81	&	442.29	&	432.53	&	13.01	&	563	&	101696	&	0.12	&	2.20	\\
\hline
\end{tabular}
}
\end{table}

\begin{table}[H]
\centering
\caption{Computational costs of solving the joint MILP developed in 
\cite{luo2018-veh-platn-mult-speeds}.}\label{tab:joint-MILP}
\scriptsize
\begin{tabular}{cccccc}
\hline\hline
Instance & Obj & Gap(\%) & LoadTime(s)  & CPU(s) & BBNodes \\
\hline 
G50-1 & 1276.18   &    14.8    &    50.4    &   14400    &   4619  \\
G50-2 & 1122.16   &    2.1    &    56.6    &   14400    &   12131  \\
\hline
\end{tabular}
\end{table}
\replace{}{Note that only the results of G50-1 and G50-2 are shown in
Table~\ref{tab:joint-MILP}. For the remaining instances, Gurobi cannot find an
incumbent solution within an hour (only a relaxation bound is given but no
optimality gap is available). The is due to the size of the integrated MILP
formulation and the presence of many big-M coefficients.}

\begin{figure}[h!]
\centering
\includegraphics[trim={10 0 45 10}, clip, width=0.40\linewidth]{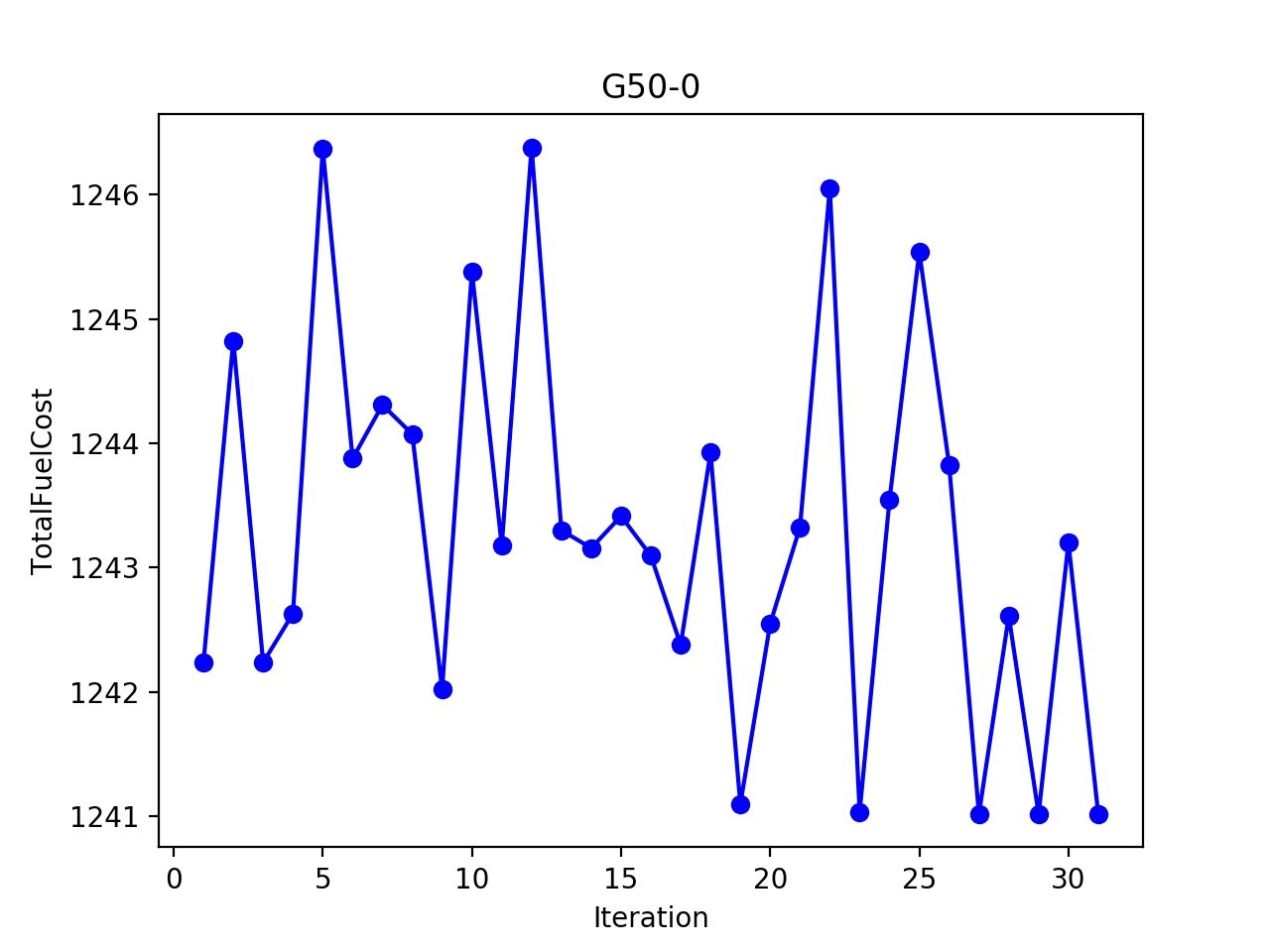} 
\includegraphics[trim={10 0 45 10}, clip, width=0.40\linewidth]{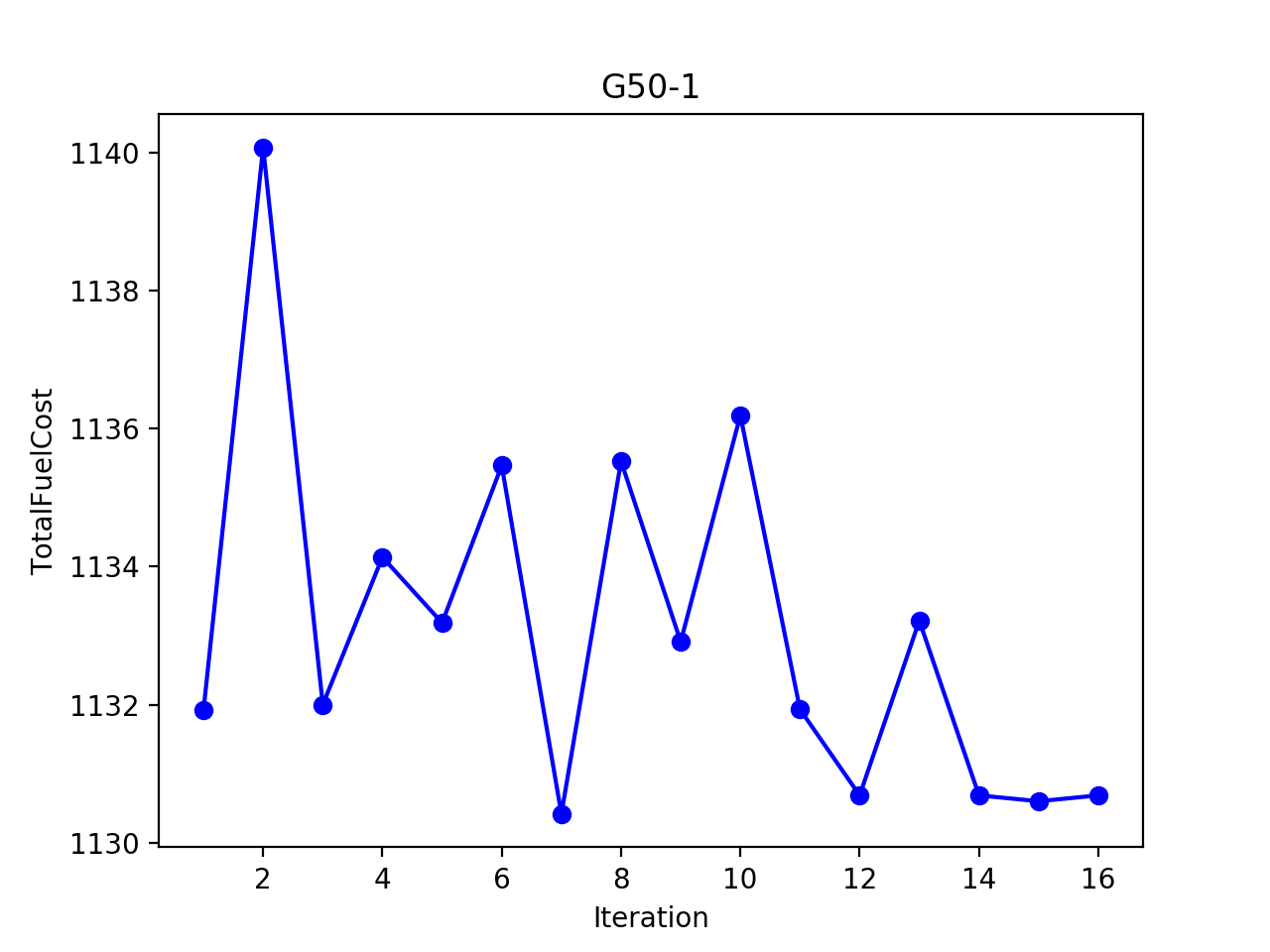}\\
\includegraphics[trim={10 0 45 10}, clip, width=0.40\linewidth]{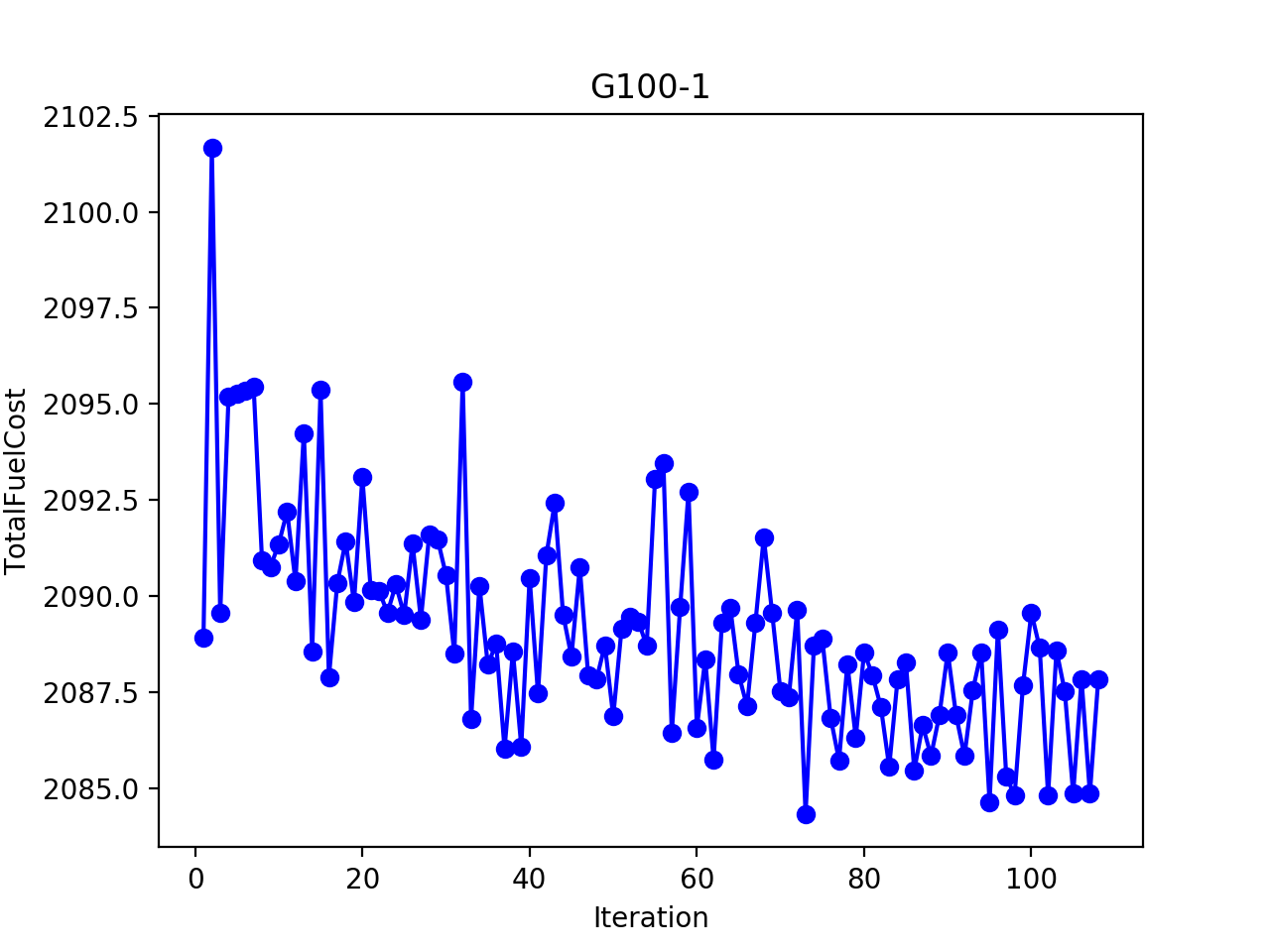} 
\includegraphics[trim={10 0 45 10}, clip, width=0.40\linewidth]{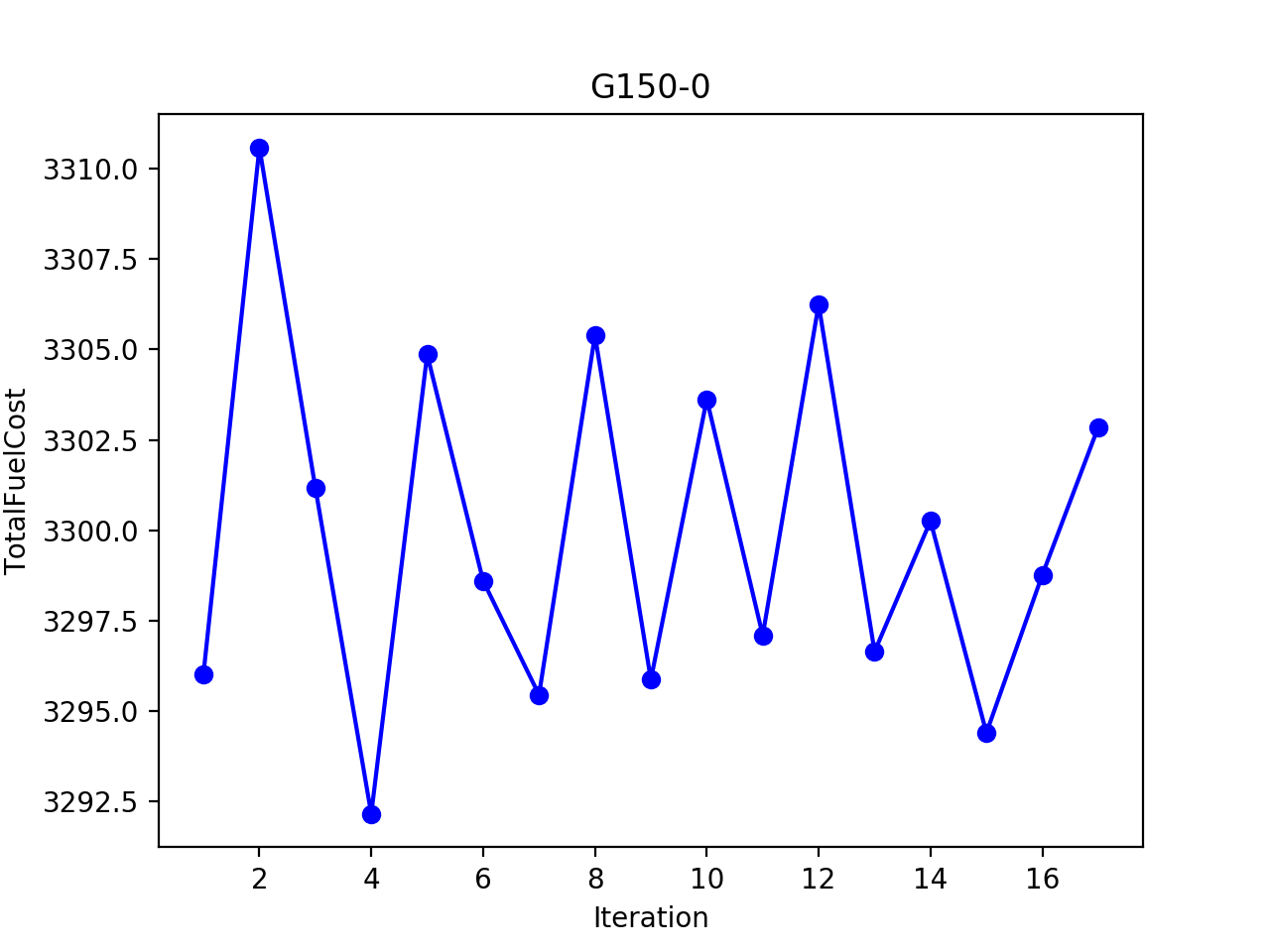} \\
\includegraphics[trim={10 0 45 10}, clip, width=0.40\linewidth]{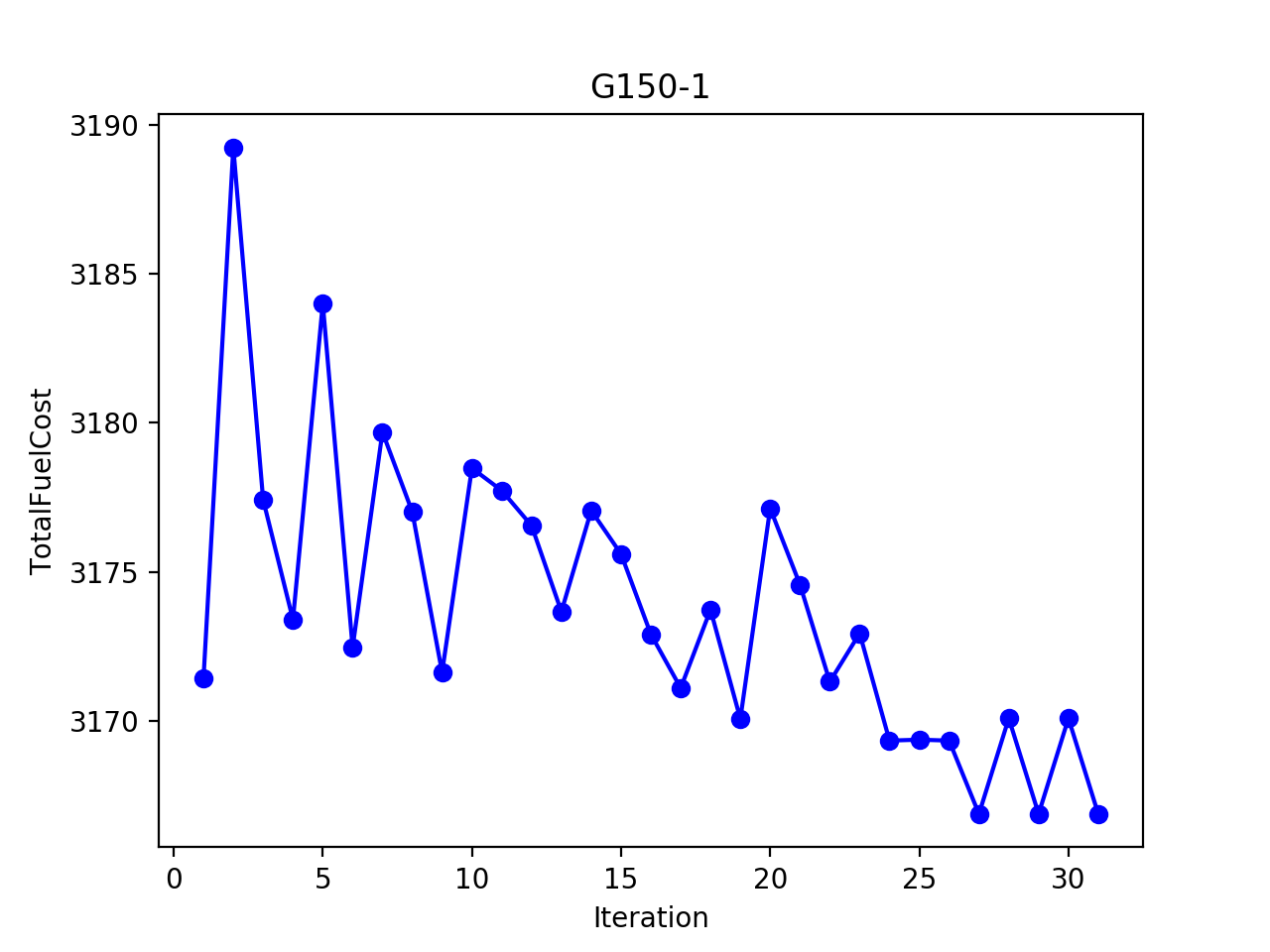} 
\caption{True total fuel cost found at every iteration in the RSHM for 5 numerical instances.}
\label{fig:RSHM-iters-plot}
\end{figure}

\begin{figure}
\centering
\captionsetup{justification=centering}
\subfigure{ \includegraphics[trim={10 0 45 10}, clip, width=0.40\linewidth]{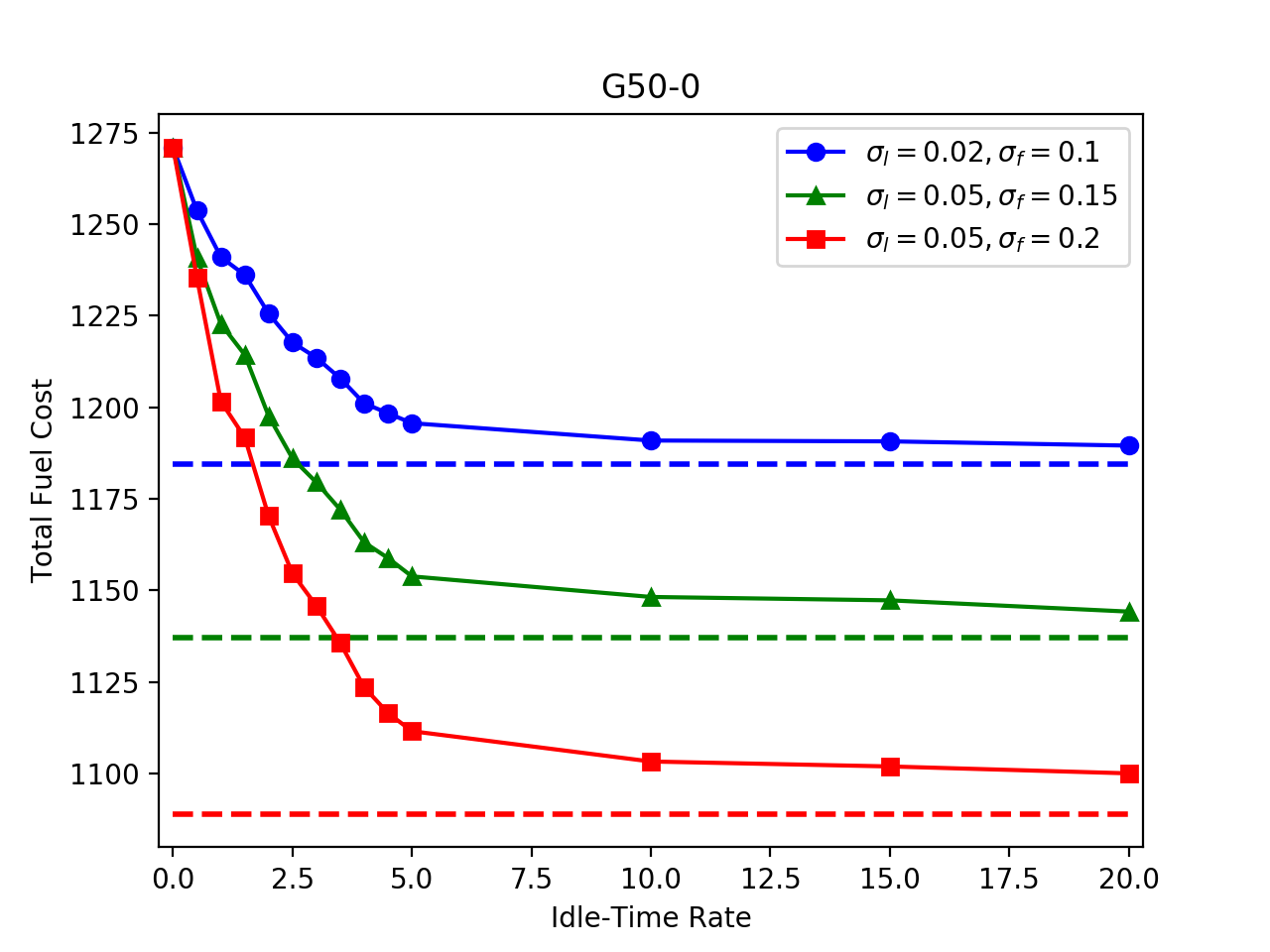} }
\subfigure{ \includegraphics[trim={10 0 45 10}, clip, width=0.40\linewidth]{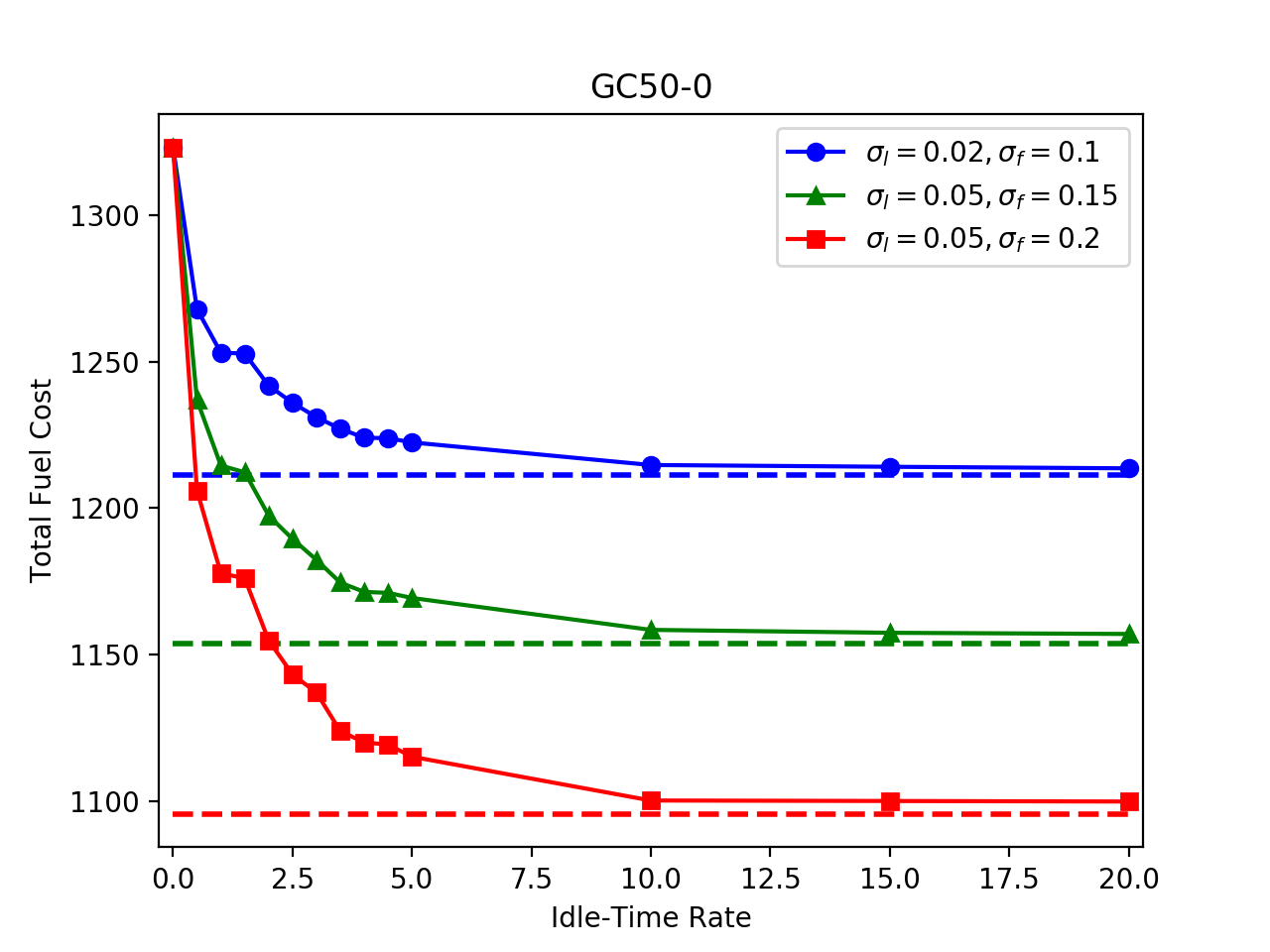} }\\
\subfigure{ \includegraphics[trim={10 0 45 10}, clip, width=0.40\linewidth]{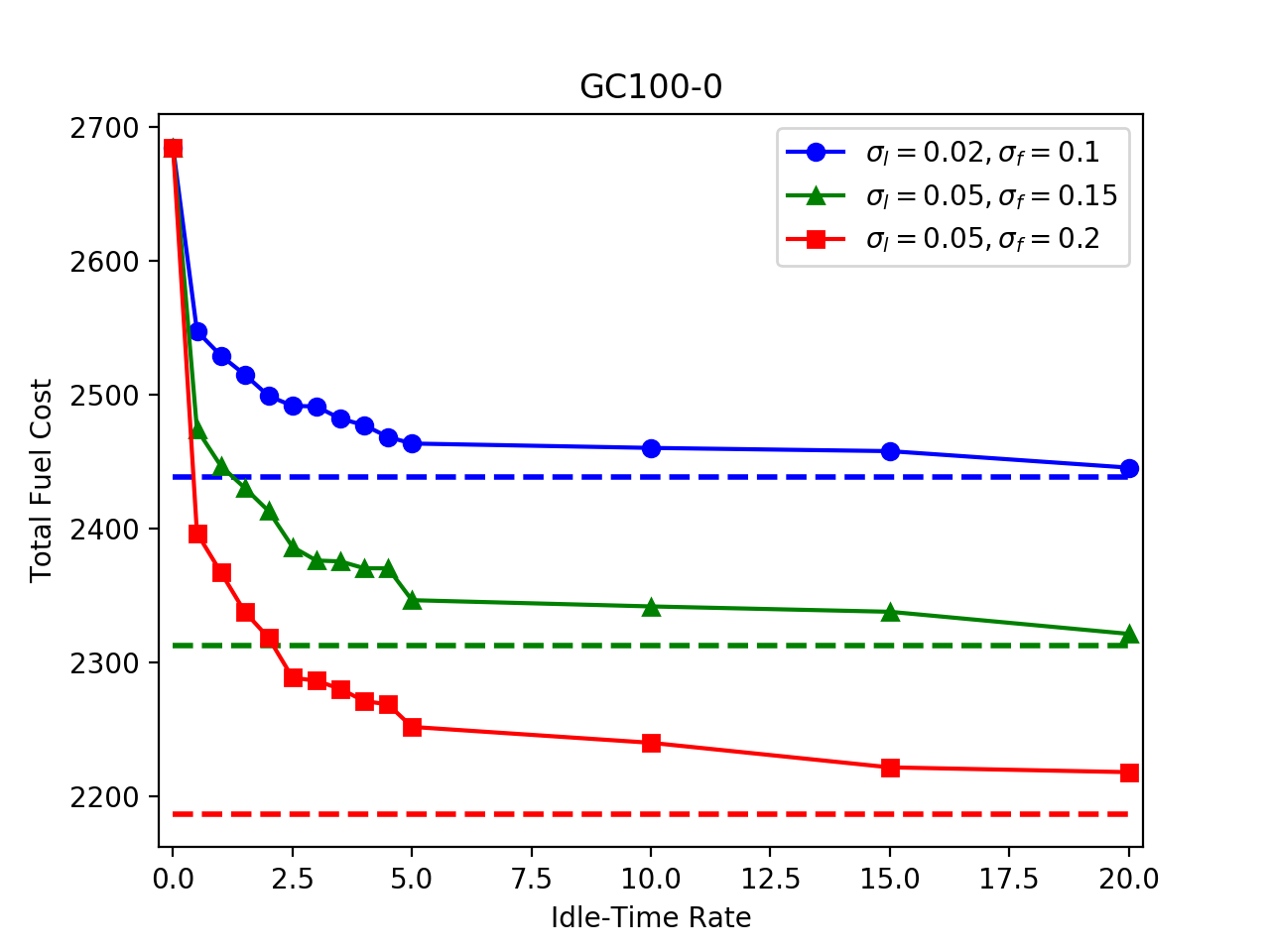} }
\caption{Total fuel cost versus time-flexibility rate for 3 CVPP instances and three platooning savings scenarios.
Dashed horizontal lines represent the fuel costs without time constraints at different settings of the saving parameters.}
\label{fig:fuel-vs-idleTime}
\end{figure}

\begin{landscape}
\begin{table}
\centering
\caption{Computational costs of
solving the \ref{opt:SPF} before and after adding cuts.
The time limit is one hour for all instances. 
Obj1 and Gap1 are the best objective
and relative MIP gap after solving the problem for 5 minutes, respectively.
If the problem is solved to optimality within 5 minutes, the Obj1 and Gap1
values are set to be the objective and MIP gap at termination.
Similarly, the Obj2/Gap2, Obj3/Gap3 and Objfinal/Gapfinal
correspond to solving the problem
for 10 minutes, 30 minutes, and at termination, respectively. 
All MIP gaps are percentages.}\label{tab:mipgap}
{\scriptsize
\begin{tabular}{c|ccccccccc|ccccccccc}
\hline\hline
& \multicolumn{9}{c|}{Plain}  & \multicolumn{9}{c}{AddCuts}  \\
\hline
 Instance	&	Obj1	&	Gap1	&	Obj2	&	Gap2	&	Obj3	&	Gap3	&	Objfinal	&	Gapfinal	& CPU(s)	&	Obj1	&	Gap1	&	Obj2	&	Gap2	&	Obj3	&	Gap3	&	Objfinal	&	Gapfinal	& CPU(s)	\\
\hline
G50-0 & 29.6 & 0.00 & -    & -    & -    & -    & -    & -    & 0.2 & 29.6 & 0.00 & -    & -    & -    & -    & -    & -    & 0.2 \\ 
G50-1 & 30.2 & 0.00 & -    & -    & -    & -    & -    & -    & 0.0 & 30.2 & 0.00 & -    & -    & -    & -    & -    & -    & 0.1 \\ 
G50-2 & 20.5 & 0.00 & -    & -    & -    & -    & -    & -    & 0.1 & 20.5 & 0.00 & -    & -    & -    & -    & -    & -    & 0.2 \\ 
G50-3 & 25.1 & 0.00 & -    & -    & -    & -    & -    & -    & 0.1 & 25.1 & 0.00 & -    & -    & -    & -    & -    & -    & 0.1 \\ 
G50-4 & 35.7 & 0.00 & -    & -    & -    & -    & -    & -    & 0.0 & 35.7 & 0.00 & -    & -    & -    & -    & -    & -    & 0.1 \\ 
G100-0 & 69.6 & 0.00 & -    & -    & -    & -    & -    & -    & 1.2 & 69.6 & 0.00 & -    & -    & -    & -    & -    & -    & 1.8 \\ 
G100-1 & 78.2 & 0.00 & -    & -    & -    & -    & -    & -    & 0.5 & 78.2 & 0.00 & -    & -    & -    & -    & -    & -    & 0.8 \\ 
G100-2 & 97.5 & 0.00 & -    & -    & -    & -    & -    & -     & 0.4 & 97.5 & 0.00 & -    & -    & -    & -    & -    & -     & 0.6 \\ 
G100-3 & 65.1 & 0.00 & -    & -    & -    & -    & -    & -    & 1.9 & 65.1 & 0.00 & -    & -    & -    & -    & -    & -    & 2.6 \\ 
G100-4 & 74.5 & 0.00 & -    & -    & -    & -    & -    & -    & 2.1 & 74.5 & 0.00 & -    & -    & -    & -    & -    & -     & 3.3 \\ 
G150-0 & 138.3 & 2.17 & 139.9 & 0.77 & 140.6 & 0.00 & -     & -     & 872.6 & 140.6 & 0.13 & 140.6 & 0.00 & -     & -    & -     & -     & 324.1 \\ 
G150-1 & 139.7 & 0.00 & -     & -    & -     & -    & -     & -     & 29.5 & 139.7 & 0.00 & -     & -    & -     & -    & -     & -     & 21.7 \\ 
G150-2 & 153.3 & 0.72 & 154.0 & 0.00 & -     & -    & -     & -     & 438.4 & 154.0 & 0.12 & 154.0 & 0.00 & -     & -    & -     & -     & 309.7 \\ 
G150-3 & 142.4 & 1.48 & 142.7 & 1.10 & 143.1 & 0.52 & 143.4 & 0.01 & 3318.1 & 142.2 & 1.82 & 142.3 & 1.56 & 142.3 & 1.32 & 142.6 & 0.97 & 3600 \\ 
G150-4 & 155.5 & 0.00 & -     & -    & -     & -    & -     & -     & 110.4 & 155.4 & 0.00 & -     & -    & -     & -    & -     & -     & 128.5 \\ 
G200-0 & 223.0 & 8.26 & 223.0 & 8.18 & 223.0 & 8.05 & 224.3 & 7.41 & 3600 & 222.4 & 8.38 & 224.4 & 7.48 & 225.7 & 6.82 & 225.9 & 6.61 & 3600 \\ 
G200-1 & 206.5 & 10.27 & 209.4 & 8.96 & 214.6 & 6.60 & 214.6 & 6.52 & 3600 & 211.3 & 8.02 & 211.3 & 7.92 & 211.3 & 7.81 & 215.9 & 5.72 & 3600 \\ 
G200-2 & 205.2 & 10.37 & 208.2 & 8.98 & 208.2 & 8.85 & 209.0 & 8.43 & 3600 & 205.5 & 10.18 & 205.5 & 10.09 & 210.7 & 7.73 & 212.1 & 7.03 & 3600 \\ 
G200-3 & 199.3 & 4.19 & 199.4 & 4.06 & 201.7 & 2.82 & 201.7 & 2.69 & 3600 & 199.8 & 3.92 & 201.5 & 3.02 & 202.7 & 2.25 & 202.7 & 2.15 & 3600 \\ 
G200-4 & 217.9 & 6.87 & 217.9 & 6.79 & 219.5 & 5.89 & 219.5 & 5.78 & 3600 & 214.5 & 8.19 & 216.8 & 7.11 & 218.2 & 6.39 & 221.4 & 4.93 & 3600 \\ 
\hline
GC50-0 & 54.7 & 0.00 & -    & -    & -    & -    & -    & -    & 0.6 & 54.7 & 0.00 & -    & -    & -    & -    & -    & -    & 0.6 \\ 
GC50-1 & 107.7 & 1.78 & 108.3 & 0.90 & 108.6 & 0.01 & -     & -    & 1142.4 & 108.5 & 1.02 & 108.5 & 0.41 & 108.6 & 0.01 & -     & -    & 701.7 \\ 
GC50-2 & 49.7 & 0.00 & -    & -    & -    & -    & -    & -    & 1.3 & 49.7 & 0.00 & -    & -    & -    & -    & -    & -    & 1.2 \\ 
GC50-3 & 104.0 & 0.00 & -     & -    & -     & -    & -     & -    & 19.8 & 103.7 & 0.00 & -     & -    & -     & -    & -     & -    & 24.1 \\ 
GC50-4 & 66.2 & 0.00 & -    & -    & -    & -    & -    & -    & 2.7 & 66.2 & 0.00 & -    & -    & -    & -    & -    & -    & 1.5 \\ 
GC100-0 & 138.5 & 1.06 & 138.5 & 0.84 & 138.5 & 0.53 & 138.5 & 0.35 & 3600 & 137.6 & 1.59 & 138.7 & 0.66 & 138.7 & 0.19 & 138.8 & 0.01 & 2084.0 \\ 
GC100-1 & 234.2 & 19.12 & 246.5 & 14.85 & 251.3 & 13.18 & 252.9 & 12.62 & 3600 & 246.6 & 14.19 & 246.6 & 14.02 & 254.0 & 11.44 & 259.4 & 9.54 & 3600 \\ 
GC100-2 & 135.1 & 11.83 & 138.0 & 9.92 & 139.4 & 8.88 & 139.4 & 8.82 & 3600 & 128.6 & 15.54 & 135.0 & 11.31 & 142.1 & 6.59 & 142.1 & 6.51 & 3600 \\ 
GC100-3 & 219.8 & 17.16 & 229.4 & 13.54 & 229.4 & 13.44 & 230.2 & 13.05 & 3600 & 220.6 & 16.36 & 220.6 & 16.35 & 232.4 & 11.82 & 234.1 & 11.18 & 3600 \\ 
GC100-4 & 161.5 & 11.98 & 163.2 & 11.04 & 168.9 & 7.79 & 168.9 & 7.57 & 3600 & 159.2 & 12.52 & 162.5 & 10.63 & 168.0 & 7.44 & 170.3 & 6.01 & 3600 \\ 
GC150-0 & 213.8 & 8.97 & 217.8 & 7.22 & 217.8 & 7.14 & 217.8 & 7.05 & 3600 & 209.5 & 10.32 & 216.4 & 7.34 & 216.4 & 7.28 & 218.7 & 6.24 & 3600 \\ 
GC150-1 & 320.7 & 32.20 & 320.7 & 32.11 & 320.7 & 32.11 & 373.5 & 20.94 & 3600 & 371.8 & 20.89 & 371.8 & 20.89 & 371.8 & 20.89 & 371.8 & 20.84 & 3600 \\ 
GC150-2 & 208.6 & 25.74 & 208.6 & 25.62 & 228.1 & 18.67 & 228.1 & 18.64 & 3600 & 223.2 & 19.14 & 223.2 & 19.14 & 223.2 & 18.95 & 233.1 & 15.32 & 3600 \\ 
GC150-3 & 307.1 & 26.85 & 307.1 & 26.81 & 321.2 & 23.45 & 321.2 & 23.43 & 3600 & 345.4 & 17.36 & 345.4 & 17.30 & 345.4 & 17.28 & 353.9 & 15.25 & 3600 \\ 
GC150-4 & 223.8 & 23.73 & 223.8 & 23.73 & 241.4 & 17.73 & 252.1 & 14.08 & 3600 & 231.8 & 20.08 & 231.8 & 20.02 & 231.8 & 19.84 & 231.8 & 19.84 & 3600 \\ 
GC200-0 & 272.6 & 21.21 & 272.6 & 21.20 & 285.0 & 17.58 & 290.4 & 16.02 & 3600 & 286.1 & 17.07 & 286.1 & 16.90 & 297.1 & 13.61 & 310.1 & 9.82 & 3600 \\ 
GC200-1 & 496.5 & 24.86 & 496.5 & 24.86 & 496.5 & 24.86 & 496.5 & 24.83 & 3600 & 496.5 & 24.91 & 496.5 & 24.91 & 496.5 & 24.91 & 496.5 & 24.85 & 3600 \\ 
GC200-2 & 259.4 & 33.43 & 259.4 & 33.43 & 259.4 & 33.31 & 259.4 & 33.31 & 3600 & 297.9 & 22.37 & 297.9 & 22.37 & 297.9 & 22.37 & 297.9 & 22.37 & 3600 \\ 
GC200-3 & 431.5 & 25.48 & 431.5 & 25.33 & 431.5 & 25.29 & 457.4 & 20.80 & 3600 & 456.9 & 20.33 & 456.9 & 20.33 & 456.9 & 20.18 & 456.9 & 20.10 & 3600 \\ 
GC200-4 & 276.1 & 33.83 & 276.1 & 33.68 & 276.1 & 33.68 & 321.0 & 22.80 & 3600 & 319.2 & 22.08 & 319.2 & 22.08 & 319.2 & 22.08 & 319.2 & 21.89 & 3600 \\ 
\hline
\end{tabular}
}
\end{table}
\end{landscape}

\begin{table}[h!]
\centering
\caption{Number of branch-and-bound nodes with and without additional cuts.
(Extends Table~\ref{tab:mipgap}.)
}
\label{tab:mipgap-final}
{\scriptsize
\begin{tabular}{c|cccc|cccc}
\hline\hline
& \multicolumn{4}{c|}{Plain} & \multicolumn{4}{c}{AddCuts} \\
\hline
Instance	&	Obj	&	Gap(\%)	&	CPUTime(s)	&	BBNodes	&	Obj	&	Gap(\%)	&	CPUTime(s)	&	BBNodes	\\
\hline
G50-0	&	29.6	&	0.00	&	0.2	&	7	&	29.6	&	0.00	&	0.2	&	29	\\
G50-1	&	30.2	&	0.00	&	0.0	&	1	&	30.2	&	0.00	&	0.1	&	1	\\
G50-2	&	20.5	&	0.00	&	0.1	&	1	&	20.5	&	0.00	&	0.2	&	0	\\
G50-3	&	25.1	&	0.00	&	0.1	&	1	&	25.1	&	0.00	&	0.1	&	1	\\
G50-4	&	35.7	&	0.00	&	0.0	&	1	&	35.7	&	0.00	&	0.1	&	1	\\
G100-0	&	69.6	&	0.00	&	1.2	&	23	&	69.6	&	0.00	&	1.8	&	51	\\
G100-1	&	78.2	&	0.00	&	0.5	&	0	&	78.2	&	0.00	&	0.8	&	0	\\
G100-2	&	97.5	&	0.01	&	0.4	&	0	&	97.5	&	0.01	&	0.6	&	0	\\
G100-3	&	65.1	&	0.00	&	1.9	&	61	&	65.1	&	0.00	&	2.6	&	77	\\
G100-4	&	74.5	&	0.00	&	2.1	&	21	&	74.5	&	0.01	&	3.3	&	102	\\
G150-0	&	140.6	&	0.01	&	872.6	&	33148	&	140.6	&	0.01	&	324.1	&	10022	\\
G150-1	&	139.7	&	0.01	&	29.5	&	728	&	139.7	&	0.01	&	21.7	&	325	\\
G150-2	&	154.0	&	0.01	&	438.4	&	16070	&	154.0	&	0.01	&	309.7	&	8894	\\
G150-3	&	143.4	&	0.01	&	3318.1	&	103178	&	142.6	&	0.97	&	3600	&	72531	\\
G150-4	&	155.5	&	0.01	&	110.4	&	2520	&	155.4	&	0.01	&	128.5	&	2758	\\
G200-0	&	224.3	&	7.41	&	3600	&	15779	&	225.9	&	6.61	&	3600	&	12972	\\
G200-1	&	214.6	&	6.52	&	3600	&	20424	&	215.9	&	5.72	&	3600	&	16409	\\
G200-2	&	209.0	&	8.43	&	3600	&	17328	&	212.1	&	7.03	&	3600	&	15394	\\
G200-3	&	201.7	&	2.69	&	3600	&	43954	&	202.7	&	2.15	&	3600	&	32838	\\
G200-4	&	219.5	&	5.78	&	3600	&	13998	&	221.4	&	4.93	&	3600	&	10997	\\
\hline
GC50-0	&	54.7	&	0.00	&	0.6	&	1	&	54.7	&	0.00	&	0.6	&	0	\\
GC50-1	&	108.6	&	0.01	&	1142.4	&	30967	&	108.6	&	0.00	&	701.7	&	15218	\\
GC50-2	&	49.7	&	0.00	&	1.3	&	1	&	49.7	&	0.00	&	1.2	&	1	\\
GC50-3	&	104.0	&	0.00	&	19.8	&	425	&	103.7	&	0.00	&	24.1	&	322	\\
GC50-4	&	66.2	&	0.00	&	2.7	&	9	&	66.2	&	0.00	&	1.5	&	1	\\
GC100-0	&	138.5	&	0.35	&	3600	&	28502	&	138.8	&	0.01	&	2084.0	&	19178	\\
GC100-1	&	252.9	&	12.62	&	3600	&	2504	&	259.4	&	9.54	&	3600	&	1558	\\
GC100-2	&	139.4	&	8.82	&	3600	&	3857	&	142.1	&	6.51	&	3600	&	2372	\\
GC100-3	&	230.2	&	13.05	&	3600	&	2775	&	234.1	&	11.18	&	3600	&	1684	\\
GC100-4	&	168.9	&	7.57	&	3600	&	5825	&	170.3	&	6.01	&	3600	&	3759	\\
GC150-0	&	217.8	&	7.05	&	3600	&	4175	&	218.7	&	6.24	&	3600	&	3939	\\
GC150-1	&	373.5	&	20.94	&	3600	&	308	&	371.8	&	20.84	&	3600	&	4	\\
GC150-2	&	228.1	&	18.64	&	3600	&	767	&	233.1	&	15.32	&	3600	&	189	\\
GC150-3	&	321.2	&	23.43	&	3600	&	801	&	353.9	&	15.25	&	3600	&	366	\\
GC150-4	&	252.1	&	14.08	&	3600	&	1105	&	231.8	&	19.84	&	3600	&	230	\\
GC200-0	&	290.4	&	16.02	&	3600	&	1300	&	310.1	&	9.82	&	3600	&	685	\\
GC200-1	&	496.5	&	24.83	&	3600	&	51	&	496.5	&	24.85	&	3600	&	0	\\
GC200-2	&	259.4	&	33.31	&	3600	&	130	&	297.9	&	22.37	&	3600	&	0	\\
GC200-3	&	457.4	&	20.80	&	3600	&	320	&	456.9	&	20.10	&	3600	&	46	\\
GC200-4	&	321.0	&	22.80	&	3600	&	433	&	319.2	&	21.89	&	3600	&	11	\\
\hline
\end{tabular}
}
\end{table}

\section{Supplement of Section~\ref{sec:form-sol-method}}
Recall that $ \Cp_{i,j}(l)$ (Definition~\ref{def:platn-fuel-cost}) denotes the total fuel consumption of a size-$l$ ($l\ge 1$) 
platoon on the edge $(i,j)$, with a size-$1$ platoon being a single vehicle.
The following properties hold.
\begin{observation}\label{obs:Cp_properties}
  This quantities $\Cp_{i,j}$ and $\Cph_{i,j}$ satisfy: 
\begin{enumerate}
  \item $\Cp_{i,j}(l)=\Cp_{i,j}(l_1)+(1-\sigma^f)(l-l_1)C_{i,j}$  for $2\le l_1\le l-1$, $l\ge 3$;
  \item $\Cp_{i,j}(l)\le \Cp_{i,j}(l_1) + \Cp_{i,j}(l-l_1)$ for $1\le l_1\le l-1$, $l\ge 2$;
  \item $\Cp_{i,j}(l_2)/l_2\le \Cp_{i,j}(l_1)/l_1$ for any $1\le l_1\le l_2$; 
  \item $\Cp_{i,j}(l)= \Cph_{i,j}(l_1)+(1-\sigma^f)(l-l_1)C_{i,j}$ for any $l_1\ge 1$, $l\ge2$, $l\ge l_1$.
\end{enumerate}
\end{observation}
\begin{proof}
1. Note that for $l\ge 2, l_1\ge 2$, we have $\Cp_{i,j}(l)=(1-\sigma^l)C_{i,j}+(1-\sigma^f)(l-1)C_{i,j}$,
and $\Cp_{i,j}(l_1)=(1-\sigma^l)C_{i,j}+(1-\sigma^f)(l_1-1)C_{i,j}$.
Therefore, $\Cp_{i,j}(l)-\Cp_{i,j}(l_1)=(l-l_1)(1-\sigma^f)C_{i,j}$. \newline
2. Since $\Cp_{i,j}(l_1)\ge(1-\sigma^l)C_{i,j}+(1-\sigma^f)(l_1-1)C_{i,j}$,
and $\Cp_{i,j}(l-l_1)\ge(1-\sigma^l)C_{i,j}+(1-\sigma^f)(l-l_1-1)C_{i,j}$,
we have 
\bdm
\ba
&\Cp_{i,j}(l_1)+\Cp_{i,j}(l-l_1)\ge2(1-\sigma^l)C_{i,j}+(1-\sigma^f)(l-2)C_{i,j} \\
&\ge (1-\sigma^l)C_{i,j}+(1-\sigma^f)C_{i,j}+(1-\sigma^f)(l-2)C_{i,j} \\
&\ge (1-\sigma^l)C_{i,j}+(1-\sigma^f)(l-1)C_{i,j}=\Cp_{i,j}(l).
\ea
\edm
3. The inequality obviously holds for $l_1=l_2=1$. For the case $l_1\ge 1$, $l_2\ge 2$,
we have 
\bdm
\ba
&\frac{\Cp_{i,j}(l_2)}{l_2}=\frac{(1-\sigma^l)C_{i,j}+(1-\sigma^f)(l_2-1)C_{i,j}}{l_2}
=\frac{(1-\sigma^l)C_{i,j}}{l_2}+(1-\sigma^f)\left(1-\frac{1}{l_2}\right)C_{i,j} \\
&= \frac{(\sigma^f-\sigma^l)C_{i,j}}{l_2}+(1-\sigma^f)C_{i,j}
\le \frac{(\sigma^f-\sigma^l)C_{i,j}}{l_1}+(1-\sigma^f)C_{i,j} \\
&=\frac{(1-\sigma^l)C_{i,j}+(1-\sigma^f)(l_1-1)C_{i,j}}{l_1}\le\frac{\Cp_{i,j}(l_1)}{l_1}.
\ea
\edm 
4. For $l_1=1$, we have $\Cph_{i,j}(l_1)=(1-\sigma^l)C_{i,j}$ and 
$\Cp_{i,j}(l)=(1-\sigma^l)C_{i,j}+(1-\sigma^f)(l-l_1)C_{i,j}$.
For $l_1\ge 2$, we have $\Cph_{i,j}(l_1)=\Cp_{i,j}(l_1)$, the equality
becomes the case 1.
\end{proof}

\label{prop:bound-on-path}
\begin{proposition}
If the cost of traversing any edge is proportional to the length of 
the edge, then the length of any optimal route of a vehicle $v$ in a CVPP instance
is no greater than $d(O_v,D_v)/(1-\sigma^f)$, where $d(O_v,D_v)$ is the 
length of shortest path from $O_v$ to $D_v$. 
\end{proposition}
\begin{proof}
We prove by contradiction. Assume there exists an optimal route assignment
$\cup_{u\in\cV}\cR^*_u$ (the routes corresponding to an optimal solution of the 
CVPP instance), such that $\sum_{(i,j)\in\cR^*_v}d_{i,j}>d(O_v,D_v)/(1-\sigma^f)$ for some vehicle $v$,
where $d_{i,j}$ is the length of the edge $(i,j)$. For any $(i,j)\in\cR^*_v$, let $\cP^*_{v,i,j}$
be the platoon that contains $v$ on edge $(i,j)$ in the optimal solution. Let $z^*$
be the optimal total fuel cost. 
We consider the following feasible solution of the CVPP instance: The vehicle $v$ 
traverses the shortest path from $O_v$ to $D_v$ without forming a platoon with any
other vehicles, and the routes and schedules of all the other vehicles remain the same
as in the optimal solution. Let $z^{\prime}$ be total fuel cost corresponding to this 
feasible solution. Define the following subset of edges in $\cR^*_v$:
\beq
\cE^{\prime}=\Set*{(i,j)\in\cR^*_v}{|\cP^*_{v,i,j}|\ge 2}.
\eeq
For any $(i,j)\in\cE^{\prime}$, let $\cP^{\prime}_{i,j}$ be the platoon consisting of 
vehicles in $\cP^*_{v,i,j}\setminus\{v\}$. Let $\alpha=C_{i,j}/d_{i,j}$ be the constant
rate between the fuel cost and distance for any edge in the network.
The difference of total fuel cost between the optimal solution and 
the constructed feasible solution can be bounded as:
\beq
\ba
&z^*-z^{\prime}=\sum_{(i,j)\in\cR^*_v}\Cp_{i,j}(|\cP^*_{v,i,j}|)
-\left(\sum_{(i,j)\in\cE^{\prime}}\Cp_{i,j}(|\cP^{\prime}_{i,j}|)+ \alpha d(O_v,D_v) \right) \\
&=\sum_{(i,j)\in\cR^*_v\setminus\cE^{\prime}}C_{i,j} 
+\sum_{(i,j)\in\cE^{\prime}} \left[\Cp_{i,j}(|\cP^*_{v,i,j}|)-\Cp_{i,j}(|\cP^{\prime}_{i,j}|) \right]
-\alpha d(O_v,D_v) \\
&=\sum_{(i,j)\in\cR^*_v\setminus\cE^{\prime}}C_{i,j} +\sum_{(i,j)\in\cE^{\prime}}(1-\sigma^f)C_{i,j}
-\alpha d(O_v,D_v) \\
&\ge\sum_{(i,j)\in\cR^*_v}(1-\sigma^f)C_{i,j} - \alpha d(O_v,D_v) 
=\alpha(1-\sigma^f)\sum_{(i,j)\in\cR^*_v}d_{i,j}- \alpha d(O_v,D_v) \\
&>\alpha (1-\sigma^f)\frac{d(O_v,D_v)}{1-\sigma^f}- \alpha d(O_v,D_v)=0.
\ea
\eeq
This contradicts $z^*$ being the optimal total fuel cost.
\end{proof}

The proof of Theorem~\ref{thm:RSHM-alg-converge} uses the result from Lemma~\ref{lem:sum-of-d<f_e},
and the proof of Lemma~\ref{lem:sum-of-d<f_e} depends on Proposition~\ref{prop:I1=I2}.
\begin{proposition}\label{prop:I1=I2} 
Let $\mathcal{S}=\{\cup_v\cR^{(k)}_v\}^n_{k=1}$ be the sequence of 
route assignments generated from Algorithm~\ref{alg:RSHM}.
Let $\cU$ be a subset of vehicles and $(i,j)$ be an edge.
Let $\mI^*:=\Set*{k}{1\le k\le n,\; \exists w\in\mL^{(k)}_{i,j}\textrm{ such that }  
\mP^{(k)}_{w,i,j}\cap\cU\neq\emptyset }$
(recall that the set $\mL^{(k)}_{i,j}$ is defined in Definition~\ref{def:notations}).
If for any $1\le k\le n$ and any $w\in\mL^{(k)}_{i,j}$,
$\mP^{(k)}_{w,i,j}\cap\cU\neq\emptyset$ implies
$\cU\subseteq\mP^{(k)}_{w,i,j}$,
then it holds that $\Set*{k+1}{k\in\mI(n,v,i,j)}\subseteq\mI^*$,
and $\mI(n,v,i,j)=\mI(n,v^{\prime},i,j)$ for any $v,v^{\prime}\in\cU$. 
\end{proposition}
\begin{proof} By definition of $\mI$, for any $k\in\mI(n,v,i,j)$, vehicle
$v$ is assigned to the edge $(i,j)$ at iteration $k+1$. 
By definition, $\mP^{(k+1)}_{v,i,j}$ is the platoon that  
contains $v$ on edge $(i,j)$ at iteration $k+1$.
Because $v\in\mP^{(k+1)}_{v,i,j}\cap\cU$, it implies
that $k+1\in\mI^*$ by the definition of $\mI^*$.
Now we show that $\mI(n,v,i,j)=\mI(n,v^{\prime},i,j)$ 
for any $v,v^{\prime}\in\cU$. 
For any iteration index $k\in\mI(n,v,i,j)$,
by the definition of $\mI(n,v,i,j)$, we must have:
\bdm
  (1)\; 1\le k\le n-2, \quad (2)\; v\in\cV^{(k+1)}_{i,j}, \quad
  (3)\; \Set*{\mP^{(k)}_{u,i,j}}{u\in\cV^{(k)}_{i,j}}=\Set*{\mP^{(n)}_{u,i,j}}{u\in\cV^{(n)}_{i,j}}.
\edm
Note that (2) implies that $\mP^{(k+1)}_{v,i,j}\cap\cU\neq\emptyset$,
and hence $\cU\subseteq\mP^{(k+1)}_{v,i,j}$ by the hypothesis in the proposition. 
Therefore, we also have 
$v^{\prime}\in\cU\subseteq\mP^{(k+1)}_{v,i,j}$, and hence
$v^{\prime}\in\cV^{(k+1)}_{i,j}$. Then we have $k\in\mI(n,v^{\prime},i,j)$
by definition. This shows that $\mI(n,v,i,j)\subseteq\mI(n,v^{\prime},i,j)$.
Similar arguments show that $\mI(n,v^{\prime},i,j)\subseteq\mI(n,v,i,j)$, and 
therefore, $\mI(n,v,i,j)=\mI(n,v^{\prime},i,j)$.
\end{proof}

\begin{lemma}\label{lem:sum-of-d<f_e}
Let $\cup_v\cR^*_v$ be an optimal route assignment.
Let $\mathcal{S}=\{\cup_v\cR^{(k)}_v\}^n_{k=1}$ be the sequence of 
route assignments generated from Algorithm~\ref{alg:RSHM}.
If $\mathcal{S}$ satisfy the non-crossing condition (Definition~\ref{def:noncrossing-cond})
with respect to $\cup_v\cR^*_v$, 
then for any $(i,j)\in\left(\cup_v\cR^*_v\right)\cap\left(\cup^n_{k=1}\cup_v\cR^{(k)}_v\right)$ 
and any platoon $\mP^*_{u,i,j}$ on $(i,j)$, the following inequality holds
\begin{equation}\label{eqn:sum-of-d<f_e}
  \sum_{v\in\mP^*_{u,i,j}} C^{(n+1)}_{v,i,j} \le \Cp_{i,j}(|\mP^*_{u,i,j}|).
\end{equation}
\end{lemma}
\begin{proof}
For any edge $(i,j)\in\left(\cup_v\cR^*_v\right)\cap\left(\cup^n_{k=1}\cup_v\cR^{(k)}_v\right)$ 
and any platoon $\mP^*_{u,i,j}$, we divide the proof for two cases.

Case (a): Condition $(i)$ of Definition~\ref{def:noncrossing-cond} is satisfied at $\mP^*_{u,i,j}$. 
This again splits into two cases: 
\begin{itemize}
	\item Case (a.1): the local nested condition (i) is satisfied by $\cup_v\cR^{(n)}_v$ at $\mP^*_{u,i,j}$;
	\item Case (a.2): the local nested condition (ii) is satisfied by $\cup_v\cR^{(n)}_v$ at $\mP^*_{u,i,j}$.
\end{itemize}
We first prove the result for Case (a.1).
Let $\mP^{(n)}_{u^{\prime},i,j}$ be the platoon from $\cV^{(n)}_{i,j}$
such that $\mP^{(n)}_{u^{\prime},i,j}\cap\mP^*_{u,i,j}$ is non-empty,
$|\mP^{(n)}_{u^{\prime},i,j}|\ge 1+\bs{1}\{|\mP^*_{u,i,j}|>1\}$,
and $\mP^{(n)}_{w,i,j}\cap\mP^*_{u,i,j}=\emptyset$ for any platoon
$\mP^{(n)}_{w,i,j}$ in $\cV^{(n)}_{i,j}$ with $w\in\cV^{(n)}_{i,j}\setminus\mP^{(n)}_{u^{\prime},i,j}$.
Also by Definition~\ref{def:noncrossing-cond}(i), the index set $\mI(n,v,i,j)$ 
is empty for each $v\in\mP^*_{u,i,j}\setminus\mP^{(n)}_{u^{\prime},i,j}$.
If $|\mP^*_{u,i,j}|=1$, $\mP^{(n)}_{u^{\prime},i,j}$ is a single vehicle platoon 
and \eqref{eqn:sum-of-d<f_e} holds trivially.
Now suppose $|\mP^*_{u,i,j}|>1$. From \eqref{eqn:C^n_ve} Case 1, 
we know that that $C^{(n+1)}_{v,i,j}=\Cp_{i,j}(|\mP^{(n)}_{u^{\prime},i,j}|)/|\mP^{(n)}_{u^{\prime},i,j}|$ 
for each $v\in\mP^{(n)}_{u^{\prime},i,j}$. 
From \eqref{eqn:C^n_ve} Case 2, 
$C^{(n+1)}_{v,i,j}=(1-\sigma^f)C_{i,j}$ for each $v\in\mP^*_{u,i,j}\setminus\mP^{(n)}_{u^{\prime},i,j}$.
Then it follows that 
\begin{equation}
\begin{aligned}
&\sum_{v\in\mP^*_{u,i,j}} C^{(n+1)}_{v,i,j} = \sum_{v\in\mP^*_{u,i,j}\cap\mP^{(n)}_{u^{\prime},i,j}} C^{(n+1)}_{v,i,j} + \sum_{v\in \mP^*_{u,i,j}\setminus\mP^{(n)}_{u^{\prime},i,j}} C^{(n+1)}_{v,i,j} \\
&=\sum_{v\in\mP^*_{u,i,j}\cap\mP^{(n)}_{u^{\prime},i,j}} \Cp_{i,j}(|\mP^{(n)}_{u^{\prime},i,j}|)/|\mP^{(n)}_{u^{\prime},i,j}| + \sum_{v\in \mP^*_{u,i,j}\setminus\mP^{(n)}_{u^{\prime},i,j}} (1-\sigma^f)C_{i,j}   \\
&=\Cp_{i,j}(|\mP^{(n)}_{u^{\prime},i,j}|)\frac{|\mP^*_{u,i,j}\cap\mP^{(n)}_{u^{\prime},i,j}|}{|\mP^{(n)}_{u^{\prime},i,j}|} + |\mP^*_{u,i,j}\setminus\mP^{(n)}_{u^{\prime},i,j}|(1-\sigma^f)C_{i,j}  \\
&=\frac{|\mP^*_{u,i,j}\cap\mP^{(n)}_{u^{\prime},i,j}|}{|\mP^{(n)}_{u^{\prime},i,j}|}
\Big[(1-\sigma_l)+(1-\sigma_f)(|\mP^{(n)}_{u^{\prime},i,j}|-1)\Big]C_{i,j} + |\mP^*_{u,i,j}\setminus\mP^{(n)}_{u^{\prime},i,j}|(1-\sigma^f)C_{i,j} \\
&=\frac{|\mP^*_{u,i,j}\cap\mP^{(n)}_{u^{\prime},i,j}|}{|\mP^{(n)}_{u^{\prime},i,j}|}(1-\sigma_l)C_{i,j}
+\Big(\frac{|\mP^*_{u,i,j}\cap\mP^{(n)}_{u^{\prime},i,j}|}{|\mP^{(n)}_{u^{\prime},i,j}|}(|\mP^{(n)}_{u^{\prime},i,j}|-1)
 +|\mP^*_{u,i,j}\setminus\mP^{(n)}_{u^{\prime},i,j}| \Big)(1-\sigma_f)C_{i,j} \\
&=\frac{|\mP^*_{u,i,j}\cap\mP^{(n)}_{u^{\prime},i,j}|}{|\mP^{(n)}_{u^{\prime},i,j}|}(\sigma_f-\sigma_l)C_{i,j}
+|\mP^*_{u,i,j}|(1-\sigma_f)C_{i,j} \\
&\le (1-\sigma_l)C_{i,j}+(|\mP^*_{u,i,j}|-1)(1-\sigma_f)C_{i,j} \\
&= \Cp_{i,j}(|\mP^*_{u,i,j}|)
\end{aligned}
\end{equation}
Therefore, \eqref{eqn:sum-of-d<f_e} holds in Case (a.1).
Now we need to prove that the result holds in Case (a.2).
Since by Definition~\ref{def:nest-cond}(ii) we have 
$\mP^*_{u,i,j}\cap\cV^{(n)}_{i,j}=\emptyset$ and by 
Definition~\ref{def:noncrossing-cond}(i), the index set $\mI(n,v,i,j)$ 
is empty for each $v\in\mP^*_{u,i,j}$.
Then from \eqref{eqn:C^n_ve} Case 2, 
we have $C^{(n+1)}_{v,i,j}=(1-\sigma^f)C_{i,j}$ for every $v\in\mP^*_{u,i,j}$.
It then follows that 
\begin{displaymath}
\begin{aligned}
&\sum_{v\in\mP^*_{u,i,j}} C^{(n+1)}_{v,i,j}=|\mP^*_{u,i,j}|(1-\sigma^f)C_{i,j}\le\Cp_{i,j}(|\mP^*_{u,i,j}|).
\end{aligned}
\end{displaymath}
This proves that \eqref{eqn:sum-of-d<f_e} holds in Case (a.2).

Case (b): Condition $(ii)$ of Definition~\ref{def:noncrossing-cond}
is satisfied at $\mP^*_{u,i,j}$. \newline
Let $\mI^*:=\Set*{k}{1\le k\le n,\; \exists w\in\mL^{(k)}_{i,j}\textrm{ such that }  
\mP^{(k)}_{w,i,j}\cap\mP^*_{u,i,j}\neq\emptyset }$, for the set $\mL^{(k)}_{i,j}$
defined in Definition~\ref{def:notations}. 
For a given $k\in\mI^*$, let $\mP^{(k)}_{w,i,j}$ be the platoon from $\cV^{(k)}_{i,j}$
such that $\mP^{(k)}_{w,i,j}\cap\mP^*_{u,i,j}\neq\emptyset$ 
(the uniqueness of $\mP^{(k)}_{w,i,j}$ is guaranteed by Definition~\ref{def:noncrossing-cond}($ii$)). 
We also have $\mP^*_{u,i,j}\subseteq\mP^{(k)}_{w,i,j}$ by 
condition $(ii)$ of Definition~\ref{def:noncrossing-cond}. 
We consider the following two sub-cases: \newline
Case (b.1): $n\in\mI^*$.
In this case there exists a platoon $\mP^{(n)}_{w,i,j}$ such that 
$\mP^*_{u,i,j}\subseteq\mP^{(n)}_{w,i,j}$ by condition $(ii)$ of Definition~\ref{def:noncrossing-cond}. 
Then it follows that
\begin{equation}
  \sum_{v\in\mP^*_{u,i,j}}C^{(n+1)}_{v,i,j} = \sum_{v\in\mP^*_{u,i,j}}\frac{\Cp_{i,j}(|\mP^{(n)}_{w,i,j}|)}{|\mP^{(n)}_{w,i,j}|}
  \le \sum_{v\in\mP^*_{u,i,j}}\frac{\Cp_{i,j}(|\mP^*_{u,i,j}|)}{|\mP^*_{u,i,j}|} = \Cp_{i,j}(|\mP^*_{u,i,j}|).
\end{equation}
Case (b.2): $n\notin\mI^*$. 
Let $\cV_0=\Set*{v\in\mP^*_{u,i,j}}{\mI(n,v,i,j)=\emptyset}$ 
and $\cV_1=\mP^*_{u,i,j}\setminus\cV_0$.
For $v\in\cV_1$, let $n_v=I(n,v,i,j)+2$.
Then the left side of \eqref{eqn:sum-of-d<f_e} can be upper bounded as follows:
\begin{equation}\label{eqn:up-bd-sum-d}
\begin{aligned}
&\sum_{v\in\mP^*_{u,i,j}}C^{(n+1)}_{v,i,j} = \sum_{v\in\cV_0} C^{(n+1)}_{v,i,j} + \sum_{v\in\cV_1} C^{(n+1)}_{v,i,j} 
=\sum_{v\in\cV_0}(1-\sigma^f)C_{i,j} + \sum_{v\in\cV_1} C^{(n_v)}_{v,i,j} \\
&=\sum_{v\in\cV_0}(1-\sigma^f)C_{i,j} + \sum_{v\in\cV_1} \frac{\Cp_{i,j}(|\mP^{(n_v)}_{v,i,j}|)}{|\mP^{(n_v)}_{v,i,j}|} 
\le \sum_{v\in\cV_0}(1-\sigma^f)C_{i,j} + \sum_{v\in\cV_1} \frac{\Cp_{i,j}(|\mP^*_{u,i,j}|)}{|\mP^*_{u,i,j}|} \\
&\le \sum_{v\in\cV_0}(1-\sigma^f)C_{i,j} + \sum_{v\in\cV_1} \frac{\Cp_{i,j}(|\cV_1|)}{|\cV_1|} 
= |\cV_0|(1-\sigma^f)C_{i,j} + \Cp_{i,j}(|\cV_1|) = \Cp_{i,j}(|\mP^*_{u,i,j}|),
\end{aligned}
\end{equation}
where we apply \eqref{eqn:C^n_ve} in the second and third equalities above.
In the first inequality, we use the following facts: 
$\mP^*_{u,i,j}\subseteq\mP^{(n_v)}_{v,i,j}$ based on the assumption of Case (b),
$\mP^{(n_v)}_{v,i,j}=\mP^{(n_{v^{\prime}})}_{v^{\prime},i,j}$ for any $v,v^{\prime}\in\mP^*_{u,i,j}$
(results of Proposition~\ref{prop:I1=I2}),
and the third property of $\Cp_{i,j}$ given in Observation~\ref{obs:Cp_properties}.  
In conclusion, \eqref{eqn:sum-of-d<f_e} holds for Cases (a) and (b).  
\end{proof}

\begin{proof}{Proof of Theorem~\ref{thm:RSHM-alg-converge}.}\label{proof_of_RSHM_theorem}
{\bf{(a)}} The number of possible \ref{opt:RDP} routes is finite. 
If the algorithm does not terminate in finitely many iterations, 
then there exists an \ref{opt:RDP} route that is generated at infinitely many
iterations in the algorithm, which will satisfy the termination
criteria.

{\bf{(b)}} By assumption the algorithm terminates at iteration $n$
and $\cR^{(n)}_v=\cR^{(n-1)}_v$ for all $v\in\cV$. 
We must have $z^{(n-1)}=\Frdp{n}(\cup_v\cR^{(n-1)}_v)=\Frdp{n}(\cup_v\cR^{(n)}_v)=z^{(n)}$.
To understand these relations, we note that the first equality is due to the way of updating 
presumed fuel cost given in the first case of \eqref{eqn:C^n_ve}.
That is at any iteration $k\ge 2$ if we assign vehicles to their previous routes 
$\cup_v\cR^{(k-1)}_v$ in the routing problem \ref{opt:RDP}-$(k)$,
then the objective value will be equal to the 
total fuel cost $z^{(k-1)}$ obtained by solving the scheduling problem 
\ref{opt:SPF}-$(k-1)$. The second equality is due to the hypothesis 
that $\cR^{(n)}_v=\cR^{(n-1)}_v$ for all $v$.
The third equality is due to the scheduling problem \ref{opt:SPF}-$(n)$
being identical to \ref{opt:SPF}-$(n-1)$ because
both problems have the same input routes by the hypothesis.
We also have $\Frdp{n}(\cup_v\cR^{(n)}_v)\le\Frdp{n}(\cup_v\cR^*_v)$, 
because $\cup_v\cR^{(n)}_v$ is an optimal solution to the problem \ref{opt:RDP}-$(n)$. 
By the way the fuel costs are updated in \eqref{eqn:C^n_ve}, it follows that
\begin{equation}\label{eqn:F^nR^*}
\begin{aligned}
&\Frdp{n}(\cup_v\cR^*_v) = 
\sum_{(i,j)\in(\cup_v\cR^*_v)\setminus(\cup_v\cR^{(n-1)}_v)}\Cp_{i,j}(|\cV^*_{i,j}|)+
\sum_{(i,j)\in(\cup_v\cR^*_v)\cap(\cup_v\cR^{(n-1)}_v)}\sum_{v\in\cV^*_{i,j}} C^{(n)}_{v,i,j} \\
=&\sum_{(i,j)\in(\cup_v\cR^*_v)\setminus(\cup_v\cR^{(n-1)}_v)}\Cp_{i,j}(|\cV^*_{i,j}|)+
\sum_{(i,j)\in(\cup_v\cR^*_v)\cap(\cup_v\cR^{(n-1)}_v)}\sum_{u\in\mL^*_{i,j}}
\sum_{v\in\mP^*_{u,i,j}}C^{(n)}_{v,i,j} \\
=&\sum_{(i,j)\in(\cup_v\cR^*_v)\setminus(\cup_v\cR^{(n-1)}_v)}\Cp_{i,j}(|\cV^*_{i,j}|)\\
&+ \sum_{(i,j)\in(\cup_v\cR^*_v)\cap(\cup_v\cR^{(n-1)}_v)}\sum_{u\in\mL^*_{i,j}}
\left(\sum_{v\in\cU^{(n-1)}_{u,i,j}}C^{(n)}_{v,i,j}+\sum_{v\in\mP^*_{u,i,j}\setminus\cU^{(n-1)}_{u,i,j}}C^{(n)}_{v,i,j} \right) \\
=&\sum_{(i,j)\in(\cup_v\cR^*_v)\setminus(\cup_v\cR^{(n-1)}_v)}\Cp_{i,j}(|\cV^*_{i,j}|) \\
&+\sum_{(i,j)\in(\cup_v\cR^*_v)\cap(\cup_v\cR^{(n-1)}_v)}\sum_{u\in\mL^*_{i,j}}
\qquad \left(\sum_{v\in\cU^{(n-1)}_{u,i,j}}C^{(n)}_{v,i,j}+|\mP^*_{u,i,j}\setminus\cU^{(n-1)}_{u,i,j}|(1-\sigma^f)C_{i,j} \right) 
\end{aligned}
\end{equation}  
By definition, the set $\cU^{(n-1)}_{u,i,j}$ can be partitioned into the two subsets: 
$\mP^*_{u,i,j}\cap\cV^{(n-1)}_{i,j}$ and 
$\mP^*_{u,i,j}\cap\left(\Set*{v}{\mI(n-1,v,i,j)\neq\emptyset}\setminus\cV^{(n-1)}_{i,j}\right)$.
For any vehicle $v$ in the first subset, 
we have $C^{(n)}_{v,i,j}=\Cp_{i,j}(|\mP^{(n-1)}_{v,i,j}|)/|\mP^{(n-1)}_{v,i,j}|$. 
While for any vehicle $v$ in the second subset, we have 
$C^{(n)}_{v,i,j}=C^{(I(n-1,v,i,j)+1)}_{v,i,j}=\Cp_{i,j}(|\mP^{(I(n-1,v,i,j)+1)}_{v,i,j}|)/|\mP^{(I(n-1,v,i,j)+2)}_{v,i,j}|$.
Using the definition of $n_v$ given in the statement of the theorem, we have
\begin{equation}\label{eqn:sum-d}
  \sum_{v\in\cU^{(n-1)}_{u,i,j}}C^{(n)}_{v,i,j}=
  \sum_{v\in\cU^{(n-1)}_{u,i,j}}\frac{\Cp_{i,j}(|\mP^{(n_v)}_{v,i,j}|)}{|\mP^{(n_v)}_{v,i,j}|}.
\end{equation}
Substituting \eqref{eqn:sum-d} into \eqref{eqn:F^nR^*} gives the following inequality:
\beq\label{eqn:Fuel^n}
\ba
&z^{(n)}\le\Frdp{n}(\cup_v\cR^*_v)\le\sum_{(i,j)\in(\cup_v\cR^*_v)\setminus(\cup_v\cR^{(n-1)}_v)}\Cp_{i,j}(|\cV^*_{i,j}|) \\
&+\sum_{(i,j)\in(\cup_v\cR^*_v)\cap(\cup_v\cR^{(n-1)}_v)}\sum_{u\in\mL^*_{i,j}}
\left(\sum_{v\in\cU^{(n-1)}_{u,i,j}}\frac{\Cp_{i,j}(|\mP^{(n_v)}_{v,i,j}|)}{|\mP^{(n_v)}_{v,i,j}|}+|\mP^*_{u,i,j}\setminus\cU^{(n-1)}_{u,i,j}|(1-\sigma^f)C_{i,j} \right).
\ea
\eeq
On the other side, the optimal fuel consumption $z^*$ can be written as:
\begin{equation}\label{eqn:APP-G(R*)}
\begin{aligned}
  &z^*=\sum_{(i,j)\in(\cup_v\cR^*_v)\setminus(\cup_v\cR^{(n-1)}_v)}
  \sum_{u\in\mL^*_{i,j}}\Cp_{i,j}(|\mP^*_{u,i,j}|)+ 
  \sum_{(i,j)\in(\cup_v\cR^*_v)\cap(\cup_v\cR^{(n-1)}_v)} \sum_{u\in\mL^*_{i,j}}\Cp_{i,j}(|\mP^*_{u,i,j}|)  \\
  &\ge \sum_{(i,j)\in(\cup_v\cR^*_v)\setminus(\cup_v\cR^{(n-1)}_v)}
  \sum_{u\in\mL^*_{i,j}}\Cp_{i,j}(|\mP^*_{u,i,j}|)+ \\
  &\qquad \sum_{(i,j)\in(\cup_v\cR^*_v)\cap(\cup_v\cR^{(n-1)}_v)} \sum_{u\in\mL^*_{i,j}}
  \left(\Cph_{i,j}(|\cU^{(n-1)}_{u,i,j}|)+|\mP^*_{u,i,j}\setminus\cU^{(n-1)}_{u,i,j}|(1-\sigma^f)C_{i,j}\right) \\
  &\ge\sum_{(i,j)\in(\cup_v\cR^*_v)\setminus(\cup_v\cR^{(n-1)}_v)}
  \sum_{u\in\mL^*_{i,j}}\Cp_{i,j}(|\mP^*_{u,i,j}|)+ \\
  &\sum_{(i,j)\in(\cup_v\cR^*_v)\cap(\cup_v\cR^{(n-1)}_v)} \sum_{u\in\mL^*_{i,j}}
  \left(\sum_{v\in\cU^{(n-1)}_{u,i,j}}\frac{\Cph_{i,j}(|\cU^{(n-1)}_{u,i,j}|)}{|\cU^{(n-1)}_{u,i,j}|}
  +|\mP^*_{u,i,j}\setminus\cU^{(n-1)}_{u,i,j}|(1-\sigma^f)C_{i,j}\right),
\end{aligned}
\end{equation}
where the first inequality uses the property: 
$\Cp_{i,j}(l)\ge\Cph_{i,j}(l_1)+(l-l_1)(1-\sigma^f)C_{i,j}$ for $0\le l_1\le l$.
After combining \eqref{eqn:APP-G(R*)} and \eqref{eqn:Fuel^n}, we obtain
\begin{equation}
\begin{aligned}
&z^{(n)}-z^*\le
\sum_{(i,j)\in(\cup_v\cR^*_v)\setminus(\cup_v\cR^{(n-1)}_v)}\Cp_{i,j}(|\cV^*_{i,j}|)
-\sum_{(i,j)\in(\cup_v\cR^*_v)\setminus(\cup_v\cR^{(n-1)}_v)}\sum_{u\in\mL^*_{i,j}}\Cp_{i,j}(|\mP^*_{u,i,j}|) \\
&\qquad +\sum_{(i,j)\in(\cup_v\cR^*_v)\cap(\cup_v\cR^{(n-1)}_v)}
\sum_{u\in\mL^*_{i,j}}\sum_{v\in\cU^{(n-1)}_{u,i,j}}
\left(\frac{\Cp_{i,j}(|\mP^{(n_v)}_{v,i,j}|)}{|\mP^{(n_v)}_{v,i,j}|} - \frac{\Cph_{i,j}(|\cU^{(n-1)}_{u,i,j}|)}{|\cU^{(n-1)}_{u,i,j}|}\right) \\
&\le\sum_{(i,j)\in(\cup_v\cR^*_v)\setminus(\cup_v\cR^{(n-1)}_v)}-\sigma^f(|\mL^*_{i,j}|-1) \\
&\qquad + \sum_{(i,j)\in(\cup_v\cR^*_v)\cap(\cup_v\cR^{(n-1)}_v)}
\sum_{u\in\mL^*_{i,j}}\sum_{v\in\cU^{(n-1)}_{u,i,j}}
\left(\frac{\Cp_{i,j}(|\mP^{(n-1)}_{v,i,j}|)}{|\mP^{(n-1)}_{v,i,j}|} - \frac{\Cph_{i,j}(|\cU^{(n-1)}_{u,i,j}|)}{|\cU^{(n-1)}_{u,i,j}|}\right),
\end{aligned}
\end{equation}
where the last inequality uses the property 
$\Cp(|\cV^*_{i,j}|)-\sum_{u\in\mL^*_{i,j}}\Cp_{i,j}(|\mP^*_{u,i,j}|)\le-\sigma^f(|\mL^*_{i,j}|-1)$
for any partition of vehicle set $\cV^*_{i,j}$ into $|\mL^*_{i,j}|$ platoons
on an edge $(i,j)$.

{\bf{(c)}} 
Depending on whether or not an edge $(i,j)\in\cup_v\cR^*_v$ belongs to previous route assignments 
$\cup^{n-1}_{k=1}\left(\cup_v\cR^{(k)}_v\right)$, 
the objective value $F^{(n)}_{\tn{\ref{opt:RDP}}}(\cup_v\cR^*_v)$ can be decomposed into two terms:
\begin{equation}\label{eqn:F^n(R*)}
\begin{aligned}
  F^{(n)}_{\tn{\ref{opt:RDP}}}(\cup_v\cR^*_v) 
&=\sum_{(i,j)\in (\cup_v\cR^*_v)\setminus(\cup^{n-1}_{k=1}\cup_v\cR^{(k)}_v)}\Cp_{i,j}\left(|\cV^*_{i,j}|\right)
+ \sum_{(i,j)\in (\cup_v\cR^*_v)\cap(\cup^{n-1}_{k=1}\cup_v\cR^{(k)}_v)} \sum_{v\in\cV^*_{i,j}}  C^{(n)}_{v,i,j},
\end{aligned}
\end{equation}
where the first term in \eqref{eqn:F^n(R*)} comes from combining the first three terms 
in \eqref{eqn:RDP-obj-update} when evaluated at 
$(i,j)\in (\cup_v\cR^*_v)\setminus(\cup^{n-1}_{k=1}\cup_v\cR^{(k)}_v)$.
Using the property of $\Cp_{i,j}$, the first term in \eqref{eqn:F^n(R*)} can be estimated as:
\begin{equation}\label{eqn:first-term}
  \sum_{(i,j)\in (\cup_v\cR^*_v)\setminus(\cup^{n-1}_{k=1}\cup_v\cR^{(k)}_v)} \Cp_{i,j}\left(|\cV^*_{i,j}|\right)
  \le \sum_{(i,j)\in (\cup_v\cR^*_v)\setminus(\cup^{n-1}_{k=1}\cup_v\cR^{(k)}_v)} \sum_{u\in\mL^*_{i,j}} \Cp_{i,j}(|\mP^*_{u,i,j}|).
\end{equation}
Applying Lemma~\ref{lem:sum-of-d<f_e} to $\{\cup\cR^{(k)}\}^{n-1}_{k=1}$, 
the second term in \eqref{eqn:F^n(R*)} can be estimated as:
\begin{equation}\label{eqn:second-term}
\begin{aligned}
 & \sum_{(i,j)\in (\cup_v\cR^*_v)\cap(\cup^{n-1}_{k=1}\cup_v\cR^{(k)}_v)} \sum_{v\in\cV^*_{i,j}}  C^{(n)}_{v,i,j}
  = \sum_{(i,j)\in (\cup_v\cR^*_v)\cap(\cup^{n-1}_{k=1}\cup_v\cR^{(k)}_v)} \sum_{u\in\mL^*_{i,j}}\sum_{v\in\mP^*_{u,i,j}} C^{(n)}_{v,i,j} \\
 &\qquad\qquad \le \sum_{(i,j)\in (\cup_v\cR^*_v)\cap(\cup^{n-1}_{k=1}\cup_v\cR^{(k)}_v)} \sum_{u\in\mL^*_{i,j}} \Cp_{i,j}(|\mP^*_{u,i,j}|)
\end{aligned}
\end{equation}
Substituting \eqref{eqn:first-term} and \eqref{eqn:second-term} into \eqref{eqn:F^n(R*)} gives
\begin{equation}
\begin{aligned}
  F^{(n)}_{\tn{\ref{opt:RDP}}}(\cup_v\cR^*_v)&\le \sum_{(i,j)\in (\cup_v\cR^*_v)\setminus(\cup^{n-1}_{k=1}\cup_v\cR^{(k)}_v)}\Cp_{i,j}\left(|\cV^*_{i,j}|\right)
  + \sum_{(i,j)\in (\cup_v\cR^*_v)\cap(\cup^{n-1}_{k=1}\cup_v\cR^{(k)}_v)} \sum_{u\in\mL^*_{i,j}} \Cp_{i,j}(|\mP^*_{u,i,j}|) \\
  &\le \sum_{(i,j)\in\cup_v\cR^*_v}\sum_{u\in\mL^*_{i,j}} \Cp_{i,j}(|\mP^*_{u,i,j}|) = z^*, 
\end{aligned}
\end{equation}
which concludes the proof of Part (c).

{\bf{(d)}}  We have
$\cR^{(n)}_v=\cR^{(n-1)}_v$ for all $v\in\cV$ by assumption, and 
$F^{(n)}_{\tn{\ref{opt:RDP}}}(\cup_v\cR^{(n)})=z^{(n)}$ shown in the first
paragraph of the proof of Part (b). 
We will show that $z^{(n)}=z^*$, 
which implies that $\cup_v\cR^n_v$ is an optimal route assignment. 
Suppose for contradiction that $z^{(n)}>z^*$. 
It follows that $F^{(n)}_{\tn{\ref{opt:RDP}}}(\cup_v\cR^n_v)=z^{(n)}>z^*$,
while by the assumed non-crossing condition and Part (c), we have
$F^{(n)}_{\tn{\ref{opt:RDP}}}(\cup_v\cR^*_v)\le z^*$.
Therefore, $F^{(n)}_{\tn{\ref{opt:RDP}}}(\cup_v\cR^*_v)<F^{(n)}_{\tn{\ref{opt:RDP}}}(\cup_v\cR^n_v)$,
which contradicts $\cup_v\cR^{(n)}_v$ being an optimal solution of \ref{opt:RDP}-$(n)$. 

{\bf{(e)}} Let $n$ be the iteration the algorithm terminates.
By assumption we have $\cR^{(n)}_u=\cR^{(n-1)}_{u}$
for all $u\in\cV$, and hence $z^{(n)}=z^{(n-1)}$.
We first consider the total fuel cost obtained at iteration $(n-1)$.
This value is
\beq
z^{(n-1)}=\sum_{(i,j)\in\cup_u\cR^{(n-1)}_u}\sum_{v\in\cL^{(n-1)}_{i,j}}\Cp_{i,j}(|\cP^{(n-1)}_{v,i,j}|).
\eeq
Consider an optimal solution of the CVPP in which $\{\cR^*_u\}_{u\in\cV}$ are the optimal routes.
The optimal fuel cost is 
\beq\label{eqn:z*}
z^*=\sum_{(i,j)\in\cup_u\cR^*_u}\sum_{v\in\cL^*_{i,j}}\Cp_{i,j}(|\cP^*_{v,i,j}|).
\eeq
We will consider the following three subsets of edges separately:
\beq
\ba
&E_1=\left(\cup_u\cR^{(n-1)}_u\right)\setminus\left(\cup_u\cR^{*}_u\right), \\
&E_2=\left(\cup_u\cR^{(n-1)}_u\right)\cap\left(\cup_u\cR^{*}_u\right), \\
&E_3=\left(\cup_u\cR^{*}_u\right)\setminus\left(\cup_u\cR^{(n-1)}_u\right).
\ea
\eeq
Let $\Fopt{n}$ be the optimal objective of \ref{opt:RDP}-$(n)$.
\beq\label{eqn:Fopt-n}
\ba
&\Fopt{n}=\Frdp{n}(\cup_u\cR^{(n)}_u)=\Frdp{n}(\cup_u\cR^{(n-1)}_u)
=\sum_{(i,j)\in\cup_u\cR^{(n-1)}_u}\sum_{v\in\cV^{(n-1)}_{i,j}}C^{(n)}_{v,i,j} \\
&=\sum_{(i,j)\in\cup_u\cR^{(n-1)}_u}\sum_{v\in\cV^{(n-1)}_{i,j}}\frac{\Cp_{i,j}(|\cP^{(n-1)}_{v,i,j}|)}{|\cP^{(n-1)}_{v,i,j}|}
=\sum_{(i,j)\in\cup_u\cR^{(n-1)}_u}\sum_{v\in\cL^{(n-1)}_{i,j}}\Cp_{i,j}(|\cP^{(n-1)}_{v,i,j}|)=z^{(n-1)}
\ea
\eeq
The evaluation of $\Frdp{n}$ at routes $\cup_u\cR^*_u$ is:
\beq\label{eqn:Frdp-n}
\ba
&\Frdp{n}(\cup_u\cR^*_u)=\sum_{(i,j)\in E_2}\left(\sum_{v\in\cV^{*}_{i,j}\cap\cV^{(n-1)}_{i,j}}C^{(n)}_{v,i,j}
+\sum_{v\in\cV^{*}_{i,j}\setminus\cV^{(n-1)}_{i,j}}C^{(n)}_{v,i,j} \right)
+\sum_{(i,j)\in E_3}\sum_{v\in\cV^*_{i,j}} C^{(n)}_{v,i,j} \\
&=\sum_{(i,j)\in E_2}\left(\sum_{v\in\cV^{*}_{i,j}\cap\cV^{(n-1)}_{i,j}}\frac{\Cp_{i,j}(|\cP^{(n)}_{v,i,j}|)}{|\cP^{(n)}_{v,i,j}|}
+\sum_{v\in\cV^{*}_{i,j}\setminus\cV^{(n-1)}_{i,j}}(1-\sigma_f)C_{i,j} \right)
+\sum_{(i,j)\in E_3}\sum_{v\in\cV^*_{i,j}} (1-\sigma_f)C_{i,j},
\ea
\eeq
where we use \eqref{eqn:C^n_ve}, $\mP^{(n)}_{v,i,j}=\mP^{(n-1)}_{v,i,j}$,
and $\mI(n-1,v,i,j)=\emptyset$ (by the hypothesis) in the second equality.
Based on \eqref{eqn:z*}, we have the following lower bound of $z^*$:
\beq\label{eqn:z*_lb}
\ba
&z^*=\sum_{(i,j)\in E_2}\sum_{v\in\cV^*_{i,j}}\frac{\Cp_{i,j}(|\cP^*_{v,i,j}|)}{|\cP^*_{v,i,j}|}
+\sum_{(i,j)\in E_3}\sum_{v\in\cV^*_{i,j}}\frac{\Cp_{i,j}(|\cP^*_{v,i,j}|)}{|\cP^*_{v,i,j}|} \\
&\ge \sum_{(i,j)\in E_2}\sum_{v\in\cV^*_{i,j}}\frac{\Cp_{i,j}(|\cP^*_{v,i,j}|)}{|\cP^*_{v,i,j}|}
+\sum_{(i,j)\in E_3}\sum_{v\in\cV^*_{i,j}}(1-\sigma_f)C_{i,j} \\
&=\sum_{(i,j)\in E_2}\sum_{v\in\cV^*_{i,j}\cap\cV^{(n-1)}_{i,j}}\frac{\Cp_{i,j}(|\cP^*_{v,i,j}|)}{|\cP^*_{v,i,j}|}
+\sum_{(i,j)\in E_2}\sum_{v\in\cV^*_{i,j}\setminus\cV^{(n-1)}_{i,j}}\frac{\Cp_{i,j}(|\cP^*_{v,i,j}|)}{|\cP^*_{v,i,j}|}
+\sum_{(i,j)\in E_3}\sum_{v\in\cV^*_{i,j}}(1-\sigma_f)C_{i,j} \\
&\ge \sum_{(i,j)\in E_2}\sum_{v\in\cV^*_{i,j}\cap\cV^{(n-1)}_{i,j}}\frac{\Cp_{i,j}(|\cP^*_{v,i,j}|)}{|\cP^*_{v,i,j}|}
+\sum_{(i,j)\in E_2}\sum_{v\in\cV^*_{i,j}\setminus\cV^{(n-1)}_{i,j}}(1-\sigma_f)C_{i,j}
+\sum_{(i,j)\in E_3}\sum_{v\in\cV^*_{i,j}}(1-\sigma_f)C_{i,j}.
\ea
\eeq
From \eqref{eqn:z*_lb} and \eqref{eqn:Frdp-n}, 
\beq
\ba
&z^{(n)}-z^*=z^{(n-1)}-z^*=\Fopt{n}-z^*\le\Frdp{n}(\cup_u\cR^*_u)-z^* \\
&\le\sum_{(i,j)\in E_2}\sum_{v\in\cV^{*}_{i,j}\cap\cV^{(n-1)}_{i,j}}\frac{\Cp_{i,j}(|\cP^{(n)}_{v,i,j}|)}{|\cP^{(n)}_{v,i,j}|}
-\sum_{(i,j)\in E_2}\sum_{v\in\cV^*_{i,j}\cap\cV^{(n-1)}_{i,j}}\frac{\Cp_{i,j}(|\cP^*_{v,i,j}|)}{|\cP^*_{v,i,j}|} \\
&=\sum_{(i,j)\in E_2}\sum_{v\in\cV^{*}_{i,j}\cap\cV^{(n)}_{i,j}}\left(\frac{\Cp_{i,j}(|\cP^{(n)}_{v,i,j}|)}{|\cP^{(n)}_{v,i,j}|} -\frac{\Cp_{i,j}(|\cP^*_{v,i,j}|)}{|\cP^*_{v,i,j}|}\right) \\
&\le \sum_{(i,j)\in E_2}\sum_{v\in\cV^{*}_{i,j}\cap\cV^{(n)}_{i,j}} \left(\frac{\Cp_{i,j}(|\cP^{(n)}_{v,i,j}|)}{|\cP^{(n)}_{v,i,j}|}- 
\frac{(\lambda-1)(1-\sigma_f)+(1-\sigma_l)}{\lambda} C_{i,j}\right) \\
&\le \sum_{(i,j)\in\cup_u\cR^{(n)}_u}\sum_{v\in\cV^{(n)}_{i,j}} \left(\frac{\Cp_{i,j}(|\cP^{(n)}_{v,i,j}|)}{|\cP^{(n)}_{v,i,j}|}-
 \frac{(\lambda-1)(1-\sigma_f)+(1-\sigma_l)}{\lambda} C_{i,j}\right) \\
&= \sum_{(i,j)\in\cup_u\cR^{(n)}_u}\sum_{v\in\cL^{(n)}_{i,j}}
\left(\Cp_{i,j}(|\cP^{(n)}_{v,i,j}|)-|\cP^{(n)}_{v,i,j}|\cdot\frac{(\lambda-1)(1-\sigma_f)+(1-\sigma_l)}{\lambda}C_{i,j} \right),
\ea
\eeq
where we use the fact that $|\cP^*_{v,i,j}|\le\lambda$.
Note that $\Cp_{i,j}(|\cP^{(n)}_{v,i,j}|)$ equals $(|\cP^{(n)}_{v,i,j}|-1)(1-\sigma_f)+(1-\sigma_l)C_{i,j}$ for 
$|\cP^{(n)}_{v,i,j}|\ge2$, and $C_{i,j}$ for $|\cP^{(n)}_{v,i,j}|=1$.
Substituting both cases into the above inequality gives the bound:
\beq
z^{(n)}-z^*\le \sum_{(i,j)\in\cup_u\cR^{(n)}_u}\sum_{v\in\cL^{(n)}_{i,j}}\left[\left(1-\frac{|\cP^{(n)}_{v,i,j}|}{\lambda} \right)(\sigma_f-\sigma_l)+\bs{1}\{|\cP^{(n)}_{v,i,j}|=1\}\sigma_l\right]C_{i,j},
\eeq
which concludes the proof.
\end{proof}

\section{Supplement of Section~\ref{sec:sol-route}}

\begin{lemma}\label{lem:a<x<b}
Let $n\ge 2$, and let real numbers 
$\{a_i\}^n_1$, $\{b_i\}^n_1$, $\{A_i\}^n_1$, $A$, and $B$ satisfy
the following conditions: $a_i\le b_i$ for $i\in\replacemath{[n]}{\{1,\ldots,n\}}$, 
$A\le\sum^n_{i=1}b_i$, $\sum^n_{i=1}a_i\le B$, and $A\le B$.
Then there exist $n$ real numbers $\{x_i\}^n_{1}$, 
such that: $a_i\le x\le b_i$ for $i\in\replacemath{[n]}{\{1,\ldots,n\}}$,  
and $A\le\sum^n_{i=1}x_i\le B$.
\end{lemma}
\begin{proof}
Let $x_i=\alpha a_i + (1-\alpha)b_i$ for $i\in\replacemath{[n]}{\{1,\ldots,n\}}$, 
where $\alpha\in[0,1]$ is a coefficient which is to be determined.
We set $A\le\sum^n_{i=1} [\alpha a_i + (1-\alpha)b_i] \le B$, 
which requires $\alpha$ to satisfy
\begin{equation}\label{eqn:alpha-1}
(\sum^n_{i=1}b_i-B)/(\sum^n_{i=1}b_i - \sum^n_{i=1}a_i)
 \le \alpha\le (\sum^n_{i=1}b_i-A)/(\sum^n_{i=1}b_i - \sum^n_{i=1}a_i).
\end{equation}
Note that conditions in the lemma ensure that 
$(\sum^n_{i=1}b_i-B)/(\sum^n_{i=1}b_i - \sum^n_{i=1}a_i)\le 1$, 
$(\sum^n_{i=1}b_i-A)/(\sum^n_{i=1}b_i - \sum^n_{i=1}a_i)\ge 0$,
and 
$(\sum^n_{i=1}b_i-B)/(\sum^n_{i=1}b_i - \sum^n_{i=1}a_i)\le 
(\sum^n_{i=1}b_i-A)/(\sum^n_{i=1}b_i - \sum^n_{i=1}a_i)$,
which are sufficient condition to guarantee that 
there exists an $\alpha\in[0,1]$ satisfying \eqref{eqn:alpha-1}. 
Then the $\{x\}^n_{i=1}$ constructed
using this $\alpha$ satisfy the desired inequalities.  
\end{proof}

\label{proof_of_RDP_thrm}
\begin{proof}{Proof of Theorem~\ref{thm:P_RDP_ij_full}.}
(a) Note that the dimension of $P^{\textrm{RDP}}_{i,j}$ is $|V|+3$.
To prove the two inequalities are facet defining, we
first show that they are valid, and then for each inequality
we construct $|\cV|+3$ affinely independent feasible integral points 
of $P^{\textrm{RDP}}_{i,j}$ that are on the plane defined by
the inequality. 
For clarity, we omit the edge indices $\{i,j\}$ in the subscript of each variable.
The validness of $y\ge  y'$ is clear, since at least two vehicles
on the edge implies that there is at least one vehicle on the edge.
Next, we show the inequality $\sum_{v\in\cV}x_{v}\ge y+  y'$ is valid. 
The inequality \eqref{eqn:RDP_3}
implies that $\sum_{v\in\cV}x_{v}\ge y$ is valid, 
and we sum this inequality with $\sum_{v\in\cV}x_v\ge 2 y'$
yielding $\sum_{v\in\cV} x_v\ge \frac{1}{2}y+ y'$. 
Since all the variables in this inequality are binary, 
we have $\sum_{v\in\cV}x_v\ge y+ y'$.

To show that $y\ge y'$ is facet defining for $P^{\textrm{RDP}}_{i,j}$,
we construct the following $|\cV|+3$ affinely independent feasible integral vectors 
(entries are organized as $[x_1,\ldots,x_{|\cV|},y, y',w]$) that are on the plane defined by $y= y'$:
\begin{equation}
\begin{aligned}
&\big[\bs{e}^{|\cV|}_k+\bs{e}^{|\cV|}_{k+1},\; 1,\; 1,\; 1\big]  \;\;\forall k=1,\ldots,|\cV|-1, 
\quad \big[\bs{e}^{|\cV|}_1+\bs{e}^{|\cV|}_{|V|},\; 1,\; 1,\; 1\big], \\
&\big[\bs{1}^{|\cV|},\; 1,\; 1,\; |V|-1\big],  \quad \big[\bs{1}^{|\cV|},\; 1,\; 1,\; 0\big], 
\quad \big[\bs{0}^{|\cV|},\; 0,\; 0,\; 0\big],
\end{aligned}
\end{equation}
where $\bs{1}^{|\cV|}$ (resp. $\bs{0}^{|V|}$) denotes the $|\cV|$ 
dimensional vector with every entry being 1 (resp. 0),
and $\bs{e}^{|V|}_k$ denotes the $|\cV|$ dimensional vector with
the $k^{\textrm{th}}$ entry being 1 and other entries being 0.
To show that $\sum_{v\in\cV}x_{v}\ge y +  y'$ is facet defining for $P^{\textrm{RDP}}_{i,j}$,
we can construct the following $|\cV|+3$ affinely independent
 feasible integral vectors that are on the plane defined by 
$\sum_{v\in\cV}x_{v} = y +  y'$:
\begin{equation}
\begin{aligned}
&\big[\bs{e}^{|\cV|}_k+\bs{e}^{|\cV|}_{k+1},\;1,\; 1,\; 1\big]  \;\;\forall k=1,\ldots,|\cV|-1, 
\quad \big[\bs{e}^{|\cV|}_1+\bs{e}^{|\cV|}_{|\cV|}, \;1, \;1, \;1\big], \\
&\big[\bs{e}^{|\cV|}_1+\bs{e}^{|\cV|}_2,\; 1,\; 1,\; 0\big],  \quad
 \big[\bs{e}^{|\cV|}_1,\; 1,\; 0,\; 0\big], \quad \big[\bs{0}^{|\cV|},\; 0,\; 0, \; 0\big].
\end{aligned}
\end{equation}

(b) Denote the polytope defined by the right side of 
\eqref{eqn:P_RDP_ij_full} as $\widetilde{P}^{\textrm{RDP}}_{i,j}$.  
We show that $\widetilde{P}^{\textrm{RDP}}_{i,j}=P^{\textrm{RDP}}_{i,j}$.
We use the lift-and-project method to prove the inequalities 
in \eqref{eqn:P_RDP_ij_full} are sufficient to describe $P^{\textrm{RDP}}_{i,j}$.
First, we branch the binary variable $ y'$. 
Specifically, let  $P_1:=\textrm{conv}\big\{P^{\textrm{RDP}}_{i,j}
\cap\{[x,y, y',w]:\; y'=1 \}\big\}$ and 
$P_2:=\textrm{conv}\big\{P^{\textrm{RDP}}_{i,j}
\cap\{[x,y, y',w]:\; y'=0 \}\big\}$. Therefore,
\begin{equation}
P_1=\Set*{[x,y, y',w]}
                {\begin{aligned}
                  &\sum_{v\in\cV}x_v\ge 2, \;\; w-\sum_{v\in\cV}x_v+1\le 0,\\
                  &0\le x_v\le 1\;\;\forall v\in\cV,\;\; w\ge 0,\;\; y=1,\;\;  y'=1
                 \end{aligned}}.
\end{equation}
\begin{equation}
P_2=\Set*{[x,y, y',w]  }
                    { \begin{aligned}
                      &x_v\le y\;\;\forall v\in\cV, \;\; w-\sum_{v\in\cV}x_v +y\le 0,\;\; y'=0 \\
                  &0\le x_v\le 1\;\;\forall v\in\cV, \;\; 0\le y\le 1,\;\; w\ge 0 
                  \end{aligned}},
\end{equation}
The claim holds because the extreme points of $P_1$ and $P_2$ are integral. 
Clearly, we have $P^{\textrm{RDP}}_{i,j}=\textrm{conv}\{P_1\cup P_2\}$. 
To describe $\textrm{conv}\{P_1\cup P_2\}$ 
using linear inequalities, consider the lifted polytope:
\begin{equation}
\mathscr{E}^{\textrm{lift}}:
=\Set*{\begin{aligned} & x, y,  y', w, \delta  \\ &x^1, y^1, y^{\prime 1}, w^1,\\ & x^2, y^2, y^{\prime 2}, w^2,   \end{aligned}}
                 {\begin{aligned}
                   &\sum_{v\in\cV} x^1_v\ge 2 \delta, \;\; z^1-\sum_{v\in\cV} x^1_v + y^1 \le 0,\\
                   & \;\; 0\le x^1_v\le  \delta\;\;\forall v\in\cV, \;\;y^1= \delta, \;\;y^{\prime 1}=\delta,\;\; w^1\ge 0,  \\   
                   &x^2_v\le y^2\;\;\forall v\in\cV, \;\; z^2-\sum_{v\in\cV} x^2_v +y^2\le 0, \\
                   &0\le x^2_v\le 1- \delta \;\; \forall v\in\cV, \;\; 0\le y^2\le 1- \delta,\;\; y^{\prime 2}=0,\;\; w^2\ge 0, \\
                   &x=x^1+x^2, \;\; y=y^1+y^2,\;\; w=w^1+w^2,\;\; y^{\prime}=y^{\prime1}+y^{\prime2} \\
                   &  0\le\delta\le 1
                   \end{aligned}
                   },
\end{equation}
where the auxiliary variables $[x^1,y^1,y^{\prime1},w^1]$ (resp. $[x^2,y^2,y^{\prime2},w^2]$)
are used to represent the polytope $P_1$ (resp. $P_2$),
and $\delta$ is the convexifying variable.
Projecting out the variables $y^{\prime1}$, $y^{\prime2}$, $\delta$
by substituting $y^{\prime2}=0$, $y^{\prime1}=y^{\prime}$ and $\delta = y^{\prime}$
to simplify $\mathscr{E}^{\textrm{lift}}$, 
we get the following projection of $\mathscr{E}^{\textrm{lift}}$ on the space induced by 
$[x,y,y^{\prime},w,x^1,y^1,w^1,x^2,y^2,w^2]$:
\begin{equation}
\textrm{Proj}_{\substack{x,y,y',w, \\ x^1,y^1,w^1,\\ x^2,y^2,w^2}}\mathscr{E}^{\textrm{lift}}:
=\Set*{\begin{aligned}x, y,  y', w,\\ x^1, y^1, w^1,\\ x^2, y^2, w^2    \end{aligned}}
                 {\begin{aligned}
                   &\sum_{v\in\cV} x^1_v\ge 2 y', \;\; z^1-\sum_{v\in\cV} x^1_v + y^1 \le 0,\\
                   & \;\; 0\le x^1_v\le  y', \;\;y^1= y' \;\;\forall v\in\cV,\;\; w^1\ge 0,  \\   
                   &x^2_v\le y^2\;\;\forall v\in\cV, \;\; z^2-\sum_{v\in\cV} x^2_v +y^2\le 0, \\
                   &0\le x^2_v\le 1- y' \;\; \forall v\in\cV, \;\; 0\le y^2\le 1- y',\;\; w^2\ge 0, \\
                   &x=x^1+x^2, \;\; y=y^1+y^2,\;\; w=w^1+w^2, \;\; 0\le  y'\le 1
                   \end{aligned}
                   }.
\end{equation}
We have $\textrm{conv}\{P_1\cup P_2\}=
\textrm{Proj}_{x,y, y',w}\mathscr{E}^{\textrm{lift}}$ 
based on the property of lifted polytope (Theorem~4.39 of \cite{conforti-IP-2014}). 
We further eliminate $y^1,x^2, y^2, w^2$ by substituting $y^1= y'$, 
$x^2=x-x^1$,
$y^2=y-y^1=y- y'$ and $w^2=w-w^1$, respectively. 
This leaves
\begin{equation}\label{eqn:Proj_xylzx1z1_E}
\textrm{Proj}_{\substack{x,x^1, \\ y,y^{\prime}, \\ w,w^1}}\mathscr{E}^{\textrm{lift}}=
                \Set*{\begin{aligned}&x, x^1, \\  & y, y^{\prime}, \\ & w, w^1   \end{aligned}}
							   {\begin{aligned}
                   &\sum_{v\in\cV} x^1_v\ge 2 y',\;\; w^1-\sum_{v\in\cV} x^1_v +  y' \le 0, \\
                   & 0\le x^1_v\le  y' \;\;\forall v\in\cV,\\
                   &w^1\ge 0, \;\; x_v - x^1_v\le y -  y' \;\;\forall v\in\cV, \\   
                   &w-\sum_{v\in\cV} x_v +y - y' \le w^1-\sum_{v\in\cV} x^1_v, \\
                   &0\le x_v-x^1_v\le 1- y' \;\; \forall v\in\cV, \\
                   &  y' \le y\le 1,\;\; w^1\le w,\;\; 0\le y' \le 1							
                   \end{aligned}
                   }.
\end{equation}
If we collect all inequalities involving $w^1$ from \eqref{eqn:Proj_xylzx1z1_E}, we obtain:
\begin{equation}\label{eqn:z1-constr}
0\le w^1,\;\; w-\sum_{v\in\cV}x_v+y- y' + \sum_{v\in\cV}x^1_v \le w^1,\;\;
 w^1\le \sum_{v\in\cV}x^1_v- y',\;\; w^1\le w.
\end{equation}
We can apply the Fourier-Motzkin elimination
to $w^1$ from \eqref{eqn:Proj_xylzx1z1_E}.
That is we replace constraints \eqref{eqn:z1-constr} with the following constraints 
in \eqref{eqn:Proj_xylzx1z1_E}:
\begin{displaymath}
\begin{aligned}
0\le w^1,\;w^1\le \sum_{v\in\cV}x^1_v- y' &\Longrightarrow 0\le \sum_{v\in\cV}x^1_v- y', \\
0\le w^1,\; w^1\le w &\Longrightarrow 0\le w, \\
w-\sum_{v\in\cV}x_v+y- y' + \sum_{v\in\cV}x^1_v \le w^1,\; w^1\le \sum_{v\in\cV}x^1_v- y' 
&\Longrightarrow w-\sum_{v\in\cV}x_v+y- y' + \sum_{v\in\cV}x^1_v\le \sum_{v\in\cV}x^1_v- y', \\
w-\sum_{v\in\cV}x_v+y- y' + \sum_{v\in\cV}x^1_v \le w^1,\;w^1\le w
&\Longrightarrow w-\sum_{v\in\cV}x_v+y- y' + \sum_{v\in\cV}x^1_v\le w.
\end{aligned}
\end{displaymath}
Therefore, 
\begin{equation}
\textrm{Proj}_{x,y, y',w,x^1}\mathscr{E}^{\textrm{lift}}=
                \Set*{\begin{array}{l}x, y,  y', \\  w, x^1  \end{array}}
                   {\def\arraystretch{1.5}
                     \begin{array}{l}
                      \sum_{v\in\cV} x^1_v\ge 2 y', \;\; 0\le x^1_v\le  y' \;\;\forall v\in\cV,  \\
                        w - \sum_{v\in\cV}x_v+y\le 0, \\
                       x_v-x^1_v\le y- y'\;\; \forall v\in\cV, \\
                       0\le x_v-x^1_v\le 1- y' \;\;\forall v\in\cV,\\
                       \sum_{v\in\cV} x_v - \sum_{v\in\cV} x^1_v \ge y- y', \\
                        y'\le y\le 1, \;\; 0\le  y'\le 1, \;\; w\ge 0   					 
                     \end{array} 
                   } 									
\end{equation}
Let $\widetilde{\mathscr{E}}^{\textrm{lift}}=
\textrm{Proj}_{x,y, y',w,x^1}\mathscr{E}^{\textrm{lift}}$, 
then it remains to show that $\wt{P}^{\textrm{RDP}}_{i,j}
=\textrm{Proj}_{x,y, y',w}\widetilde{\mathscr{E}}^{\textrm{lift}}$.
First, for any point $[\wt{x},\widetilde{y},
\wt{ y'},\wt{w},\wt{x}^1]\in\wt{\mathscr{E}}^{\textrm{lift}}$, 
we need to verify that $[\wt{x},\widetilde{y},\wt{ y'},\wt{w}]$ 
satisfies all inequalities defining $\wt{P}^{\textrm{RDP}}_{i,j}$. 
The verification is given as:
\begin{equation}
\begin{aligned}
\textrm{inequalities from }\widetilde{\mathscr{E}}^{\textrm{lift}}\qquad\qquad &
\qquad\quad \textrm{inequalities from }\wt{P}^{\textrm{RDP}}_{i,j} \\
\sum_{v\in\cV}\tilde{x}^1_v\ge 2\tilde{ y'},\;
\sum_{v\in\cV}\tilde{x}_v-\sum_{v\in\cV}\tilde{x}^1_v\ge\tilde{y}-\tilde{ y'},\;
\tilde{ y'}\le\tilde{y}\quad&\Longrightarrow \quad\sum_{v\in\cV}\tilde{x}_v\ge 2\tilde{ y'} \\
\tilde{x}_v-\tilde{x}^1_v\le \tilde{y}-\tilde{ y'},\; \tilde{x}^1_v\le\tilde{ y'}
\quad&\Longrightarrow \quad\tilde{x}_v\le \tilde{y} \\
\tilde{w} - \sum_{v\in\cV}\tilde{x}_v+\tilde{y}\le 0 \quad&
\Longrightarrow\quad\tilde{w} - \sum_{v\in\cV}\tilde{x}_v+\tilde{y}\le 0 \\
\sum_{v\in\cV}\tilde{x}^1_v\ge 2\tilde{ y'},\;
\sum_{v\in\cV}\tilde{x}_v-\sum_{v\in\cV}\tilde{x}^1_v\ge\tilde{y}-\tilde{ y'}
\quad&\Longrightarrow \quad\sum_{v\in\cV}\tilde{x}_v\ge \tilde{y}+\tilde{ y'}.
\end{aligned}
\end{equation}
This shows that $\textrm{Proj}_{x,y, y',w}\mathscr{E}^{\textrm{lift}}
\subseteq \wt{P}^{\textrm{RDP}}_{i,j}$.
Conversely, for any point 
$[\wt{x},\widetilde{y},\wt{ y'},\wt{w}]\in\wt{P}^{\textrm{RDP}}_{i,j}$, 
we need to show that 
there exists an $\wt{x}^1$ such that 
$[\wt{x},\wt{y},\wt{ y'},\wt{w},\wt{x}^1]\in\widetilde{\mathscr{E}}^{\textrm{lift}}$.
Equivalently, $\wt{x}^1$ should satisfy:
\begin{align}
&\textrm{max}\{0, \wt{x}_v-\wt{y}+\wt{ y'}\} \le \wt{x}^1_v \le 
\textrm{min}\{ \wt{x}_v, \wt{ y'} \} \;\; \forall v\in\cV  \mbox{, and}\nonumber \\
&2\wt{ y'} \le \sum_{v\in\cV} \wt{x}^1_v \le \sum_{v\in\cV}\wt{x}_v - \wt{y} + \wt{ y'}.    \nonumber
\end{align}
We now apply Lemma~\ref{lem:a<x<b} on the above linear system 
by setting parameters in the lemma as: 
$a_v=\textrm{max}\{0, \wt{x}_v-\wt{y}+\wt{ y'}\}$, 
$b_v=\textrm{min}\{ \wt{x}_v,  y' \}$ for $v\in\cV$, 
$A=2\wt{ y'}$, and $B=\sum_{v\in\cV}\wt{x}_v - \wt{y} + \wt{ y'}$.
It is easy to verify that these parameters
satisfy the conditions in Lemma~\ref{lem:a<x<b}.
Therefore, such $\wt{x}^1$ exists, and hence 
$\wt{P}^{\textrm{RDP}}_{i,j}\subseteq 
\textrm{Proj}_{x,y, y',w}\mathscr{E}^{\textrm{lift}}$.
To conclude, we have proved that 
$\wt{P}^{\textrm{RDP}}_{i,j}=\textrm{Proj}_{x,y, y',w}\mathscr{E}^{\textrm{lift}}$, 
which shows 
$\wt{P}^{\textrm{RDP}}_{i,j}=P^{\textrm{RDP}}_{i,j}$.
\end{proof}

\subsection{Supplement of Section~\ref{sec:disj-cut}}
\label{app:act-constr}
\subsubsection{An Illustrative Example of Algorithm~\ref{alg:act-constr-coll}.}
\label{app:disj-cut-example}
Consider an example of 4 vehicles $v_1,v_2,v_3,v_4$ in
Figure~\ref{fig:ill-exp-alg} with routes
\begin{displaymath}
\begin{aligned}
R_{v_1}:\; A\longrightarrow B\longrightarrow C\longrightarrow D\longrightarrow E, \quad
&R_{v_2}:\; F\longrightarrow C\longrightarrow D\longrightarrow G\longrightarrow H, \\
R_{v_3}:\; I\longrightarrow B\longrightarrow C\longrightarrow J,  \quad
&R_{v_4}:\; L\longrightarrow D \longrightarrow G \longrightarrow K.
\end{aligned}
\end{displaymath}

\begin{figure}
\centering
    \includegraphics[trim=0cm 16.5cm 6cm 4.5cm]{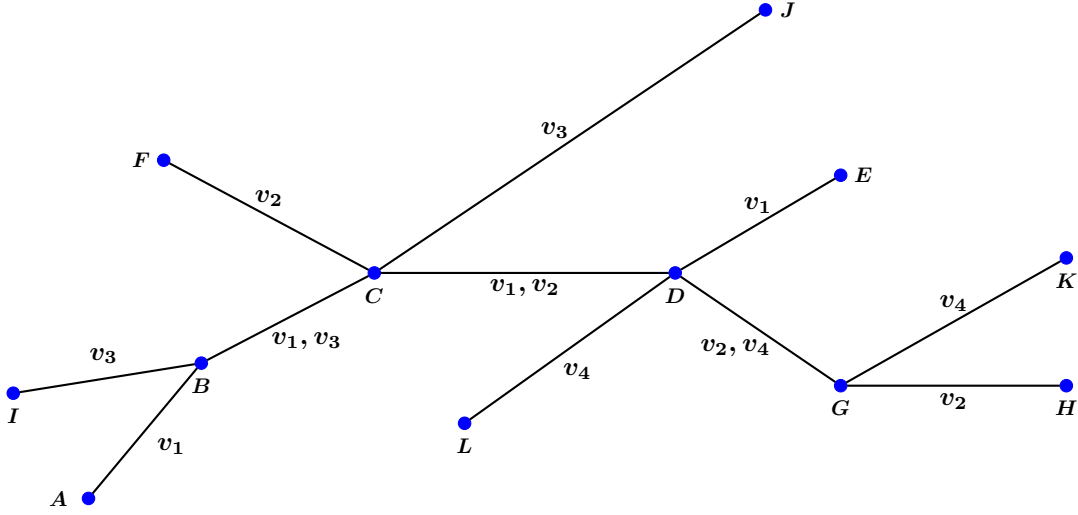}
    \caption{Illustrative example for Algorithm~\ref{alg:act-constr-coll} that generates disjunctive cuts
    for the scheduling problem.}
    \label{fig:ill-exp-alg}
\end{figure} 
Suppose the time of traversing the edge $(C,J)$ is $T_{C,J}=2$, and time of traversing
all other edges is 1. The departure- and arrival-time parameters as well as the $M$ constants are: 
\begin{displaymath}
\begin{aligned}
&T^O_{v_1}=0, \; T^D_{v_1}=8, \quad 
T^O_{v_2}=1,  \; T^D_{v_2}=9, \quad
T^O_{v_3}=3, \; T^D_{v_3}=8,  \quad
T^O_{v_4}=2,  \; T^D_{v_4} =7, \\
&M_{v_3v_1,BC}=\ol{t}_{v_3,B}-\ul{t}_{v_1,B}=5-1=4, \quad
M_{v_2v_1,CD}=\ol{t}_{v_2,C}-\ul{t}_{v_1,C}=6-2=4, \\
&M_{v_4v_2,DG}=\ol{t}_{v_2,D}-\ul{t}_{v_4,D}=7-3=4.
\end{aligned}
\end{displaymath}
Consider a fractional solution $[\hat{t},\hat{y}]$ given as:
\begin{displaymath}
\begin{aligned}
&\hat{t}_{v_1,A}=3,\; \hat{t}_{v_1,B}=4,\; \hat{t}_{v_1,C}=5,\;\hat{t}_{v_1,D}=6,\;\hat{t}_{v_1,E}=7, \\
&\hat{t}_{v_2,F}=3,\;  \hat{t}_{v_2,C}= 4,\; \hat{t}_{v_2,D}=5,\;  \hat{t}_{v_2,G}=6,\; \hat{t}_{v_2,H}=7, \\
&\hat{t}_{v_3,I}=3,\; \hat{t}_{v_3,B}=4,\; \hat{t}_{v_3,C}=5,\;\hat{t}_{v_3,J}=7, \\
&\hat{t}_{v_4,L}=4,\;\hat{t}_{v_4,D}=5,\;\hat{t}_{v_4,G}=6,\;\hat{t}_{v_4,K}=7, \\
&\hat{f}_{v_3v_1,BC}=1,\;\hat{f}_{v_4v_2,DG}=1,\;\hat{f}_{v_2v_1,CD}=\frac{3}{4}.
\end{aligned}
\end{displaymath}
In this fractional solution,
vehicles $v_1,v_3$ form a platoon on their shared edge $(B,C)$ ($\hat{f}_{v_3v_1,BC}=1$).
Vehicle $v_3$ departs at its earliest possible time ($\hat{t}_{v_3,I}=T^{v_3}_O$), which
forces the departure time of $v_1$ to be 3.
Vehicles $v_2,v_4$ form a platoon on their shared
edge $(D,G)$ ($\hat{f}_{v_4v_2,DG}=1$). Vehicle $v_4$ departs at its latest possible time 
($\hat{t}_{v_4,L}=\ol{t}_{v_4,L}$), which forces the departure time of $v_2$ to be 3. 
The arrival times of $v_1$ and $v_2$ at node $C$
are determined by $v_3$ and $v_4$, respectively due to the platooning.
The fractional $\hat{f}_{v_2v_1,CD}$ value is determined by activating the
constraint $t_{v_2,C}-t_{v_1,C}\le M_{v_2v_1,CD}(1-f_{v_2,v_1,CD})$.

Consider running Algorithm~\ref{alg:act-constr-coll} with $[\hat{t},\hat{y}]$ as the input.
At Line~\ref{line:y} of Part 1, the fractional value $\hat{f}_{v_2v_1,CD}$ is identified,
and it can be verified that $|\hat{t}_{v_2,C}-\hat{t}_{v_1,C}|=M_{v_2v_1,CD}(1-\hat{f}_{v_2v_1,CD})$. 
At Line~\ref{line:init-V1-V2} of Part 1, $\cV_1$ and $\cV_2$ are initialized as
$\cV_1\gets\{v_2\}$, $\cV_2\gets\{v_1\}$. When Line~\ref{line:deep-search} of Part 1 is reached for $\cV_1$,
the vehicle $v_4$ is identified and added into $V_1$ at Line~\ref{line:add-u} of Part 2.
Similarly, when running Line~\ref{line:deep-search} of Part 1 for $\cV_2$,
the vehicle $v_3$ is identified and added into $\cV_2$ at Line~\ref{line:add-u} of Part 2.
Finally, Algorithm~\ref{alg:act-constr-coll} returns $\cV_1=\{v_2,v_4\}$ and $\cV_2=\{v_1,v_3\}$.
Based on \eqref{eqn:ac1}--\eqref{eqn:ac6}, the following constraints are active at $[\hat{t},\hat{f}]$:
\begin{displaymath}
\begin{aligned}
&t_{v_1,B}=t_{v_1,A}+1,\;t_{v_1,C}=t_{v_1,B}+1,\;t_{v_1,D}=t_{v_1,C}+1,\;t_{v_1,E}=t_{v_1,D}+1,\\
&t_{v_2,C}=t_{v_2,F}+1,\;t_{v_2,D}=t_{v_2,C}+1,\;t_{v_2,G}=t_{v_2,D}+1,\;t_{v_2,H}=t_{v_2,G}+1,\\
&t_{v_3,B}=t_{v_3,I}+1,\;t_{v_3,C}=t_{v_3,B}+1,\;t_{v_3,J}=t_{v_3,C}+2,\\
&t_{v_4,D}=t_{v_4,L}+1,\;t_{v_4,G}=t_{v_4,D}+1,\;t_{v_4,K}=t_{v_4,G}+1, \\
&t_{v_3,I}\ge \ul{t}_{v_3,I},\; t_{v_4,L}\le\ol{t}_{v_4,L}, \\
&t_{v_3,B}-t_{v_1,B}\le M_{v_3v_1,BC}(1-f_{v_3,v_1,BC}), \\
&t_{v_4,D}-t_{v_2,D}\le M_{v_4v_2,DG}(1-f_{v_4,v_2,DG}), \\
&t_{v_2,C}-t_{v_1,C}\le M_{v_2v_1,CD}(1-f_{v_2,v_1,CD}).
\end{aligned}
\end{displaymath}

\subsubsection{Proofs for Section~\ref{sec:disj-cut}.}
\label{proofs_for_disj_cuts}
\begin{proof}{Proof of Proposition~\ref{prop:constrList}.}
There are three locations where Algorithm~\ref{alg:act-constr-coll} can return `None':
Line~\ref{line:empty1}, Line~\ref{line:empty2} and Line~\ref{line:empty3}.  
It suffices to prove that the conditions for dropping into the three situations
in the algorithm do not hold. 
First, the condition for reaching Line~\ref{line:empty1} of Part 1 is that there 
\emph{does not exist} a fractional $\hat{f}_{u^*,v^*,i^*,j^*}$, such that 
\begin{equation}
|\hat{t}_{u^*,i^*}-\hat{t}_{v^*,i^*}|=M_{u^*,v^*,i^*,j^*}(1-\hat{f}_{u^*,v^*,i^*,j^*}).
\end{equation}
Since $[\hat{t},\hat{f}]$ is a fractional solution by assumption, to reach Line~\ref{line:empty1},
we must have that for any fractional $\hat{f}_{u^*,v^*,i^*,j^*}$, the following inequality hold
\begin{equation}
\hat{t}_{u^*,i^*}-\hat{t}_{v^*,i^*}<M_{u^*,v^*,i,j}(1-\hat{f}_{u^*,v^*,i^*,j^*}).
\end{equation}
Let $\delta>0$ be sufficient small such that $\hat{t}_{u^*,i^*}-\hat{t}_{v^*,i^*}<M_{u^*,v^*,i,j}(1-\hat{f}_{u^*,v^*,i^*,j^*}-\delta)$
and $\hat{f}_{u^*,v^*,i^*,j^*}>\delta$.
Then construct two points $[\hat{t}^1,\hat{f}^1]$ and $[\hat{t}^2,\hat{f}^2]$,
such that $\hat{f}^1_{u^*,v^*,i^*,j^*}=\hat{f}_{u^*,v^*,i^*,j^*}-\delta$, $\hat{f}^2_{u^*,v^*,i^*,j^*}=\hat{f}_{u^*,v^*,i^*,j^*}+\delta$
and all the other entries of the two points are the same as $[\hat{t},\hat{f}]$.
It is easy to check that $[\hat{t}^1,\hat{f}^1]$ and $[\hat{t}^2,\hat{f}^2]$ are in $\Psprel$. 
The convex combination $[\hat{t},\hat{f}]=\frac{1}{2}[\hat{t}^1,\hat{f}^1]+\frac{1}{2}[\hat{t}^2,\hat{f}^2]$ holds.
It contradicts that $[\hat{t},\hat{f}]$ is an extremal point of $\Psprel$.

To show that the algorithm will not reach Line~\ref{line:empty2} 
and Line~\ref{line:empty3} of the main part,
we assume that $\hat{t}_{u^*,i^*}>\hat{t}_{v^*,i^*}$. The proof for the case 
of $\hat{t}_{u^*,i^*}<\hat{t}_{v^*,i^*}$ is similar.
Suppose the algorithm terminates at Line~\ref{line:empty2} of Part 1.
Then the algorithm must reach Line~\ref{line:signal0} 
when running the procedure \Call{DeepSearch}{$flag$,$V_1$}.
Consider the set $V_1$ when the procedure \Call{DeepSearch}{$flag$,$V_1$} terminates.
Suppose $\cV_1=\{u_k\}^m_{k=1}$ where vehicles $u_1,u_2,\ldots,u_m$ 
are added to $\cV_1$ following the sequence.
Line~\ref{line:init-V1-V2} (Part 1) implies that $u_1=u^*$.
Due to the condition at Line~\ref{line:cond} in the procedure,
the following properties hold:
\begin{equation}\label{eqn:property}
\begin{aligned}
&\forall k=1,\ldots,m-1,\;\exists n_{k+1}\in\;\textrm{NodesOf}(\cR_u\cap\cR_v) 
\textrm{ satisfying } \hat{t}_{u_k,n_{k+1}}-\hat{t}_{u_{k+1},n_{k+1}}=0, \\
&\forall u\in \cV_1,\;\forall v\in\cV\setminus \cV_1,\;\forall (i,j)\in\cR_u\cap\cR_v,
\;|\hat{t}_{u,i}-\hat{t}_{v,i}|<M_{\ol{u,v},i,j}(1-\hat{f}_{u,v,i,j}).
\end{aligned}
\end{equation}
Based on the above properties, we can construct two feasible points
$[\hat{t}^1,\hat{f}^1]$ and $[\hat{t}^2,\hat{f}^2]$ in $\Psprel$,
such that a convex combination of them is equal to $[\hat{t},\hat{f}]$.   
These two points are constructed as:
\begin{equation}\label{eqn:two-pts}
\begin{aligned}
\hat{t}^1_{w,i}=\hat{t}_{w,i}+\delta \quad\forall w\in\cV_1,\;\forall (i,j)\in\cR_w, 
&\qquad \hat{f}^1_{u^*,v^*,i^*,j^*}=\hat{f}_{u^*,v^*,i^*,j^*}-\delta/M_{u^*,v^*,i^*,j^*},  \\
\hat{t}^2_{w,i}=\hat{t}_{w,i}-\delta \quad\forall w\in\cV_1, \;\forall (i,j)\in\cR_w,
&\qquad  \hat{f}^2_{u^*,v^*,i^*,j^*}=\hat{f}_{u^*,v^*,i^*,j^*}+\delta/M_{u^*,v^*,i^*,j^*}, 
\end{aligned}
\end{equation}
where $\delta$ is a sufficiently small positive constant, 
and $\hat{f}_{u^*,v^*,i^*,j^*}$ is the fractional element of $[\hat{t},\hat{f}]$ identified at Line~\ref{line:y}.
The other entries of $[\hat{t}^1,\hat{f}^1]$ and $[\hat{t}^2,\hat{f}^2]$ are the same as in $[\hat{t},\hat{f}]$.
Note that we prove in the previous paragraph that the equality
$\hat{t}_{u^*,i^*}-\hat{t}_{v^*,i^*}=M_{u^*,v^*,i^*,j^*}(1-\hat{f}_{u^*,v^*,i^*,j^*})$ holds.
One can verify that the two points constructed in \eqref{eqn:two-pts} satisfy 
$\hat{t}^1_{u^*,i^*}-\hat{t}^1_{v^*,i^*}=M_{u^*,v^*,i^*,j^*}(1-\hat{f}^1_{u^*,v^*,i^*,j^*})$ 
and $\hat{t}^2_{u^*,i^*}-\hat{t}^2_{v^*,i^*}=M_{u^*,v^*,i^*,j^*}(1-\hat{f}^2_{u^*,v^*,i^*,j^*})$, respectively.
Since for the two points, the departure times of all vehicles in $\cV_1$
have been shifted for the same amount, the properties in \eqref{eqn:property} are satisfied 
at the two points. The sufficient small shifting will not violate the earliest possible departure time
or the arrival deadline since $\underline{t}_{u,O_u}<\hat{t}_{u,O_u}<\ol{t}_{u,O_u}$ 
for any $u\in\cV_1$ which is due to the condition at Line~\ref{line:cond} of the procedure is not satisfied
(otherwise the algorithm will not terminate at Line~\ref{line:empty2} in the main part).
Therefore, they are two feasible points in $\Psprel$.
Furthermore, we have $[\hat{t},\hat{y}]=\frac{1}{2}[\hat{t}^1,\hat{f}^1]+\frac{1}{2}[\hat{t}^2,\hat{f}^2]$,
which contradicts $[\hat{t},\hat{f}]$ being an extremal point of $\Psprel$.

Suppose the algorithm terminates at Line~\ref{line:empty3} in the main part,
we construct the following two feasible points:
\begin{equation}
\begin{aligned}
\hat{t}^1_{w,i,j}=\hat{t}_{w,i,j}+\delta \quad\forall w\in\cV_2,\;\forall (i,j)\in\cR_w,  
&\qquad \hat{f}^1_{u^*,v^*,i^*,j^*}=\hat{f}_{u^*,v^*,i^*,j^*}+\delta/M_{u^*,v^*,i^*,j^*},\\
\hat{t}^2_{w,i,j}=\hat{t}_{w,i,j}-\delta \quad\forall w\in\cV_2, \;\forall (i,j)\in\cR_w,
&\qquad  \hat{f}^2_{u^*,v^*,i^*,j^*}=\hat{f}_{u^*,v^*,i^*,j^*}-\delta/M_{u^*,v^*,i^*,j^*}, 
\end{aligned}
\end{equation}  
and all other entries of $[\hat{t}^1,\hat{f}^1]$ and $[\hat{t}^2,\hat{f}^2]$
remain the same as $[\hat{t},\hat{f}]$.
Then we can use a similar argument to get the contradiction.
\end{proof}

\begin{proof}{Proof of Theorem~\ref{thm:disj-cut}.}
(a) We only prove for the case that $\hat{t}_{u^*,i^*}-\hat{t}_{v^*,i^*}> 0$, the 
proof for the case $\hat{t}_{u^*,i^*}-\hat{t}_{v^*,i^*}<0$ is almost identical.
We first show that $A\hat{\omega}=b$ and 
$\hat{t}_{u^*,i^*}-\hat{t}_{v^*,i^*}=M_{u^*,v^*,i^*,j^*}(1-\hat{f}_{u^*,v^*,i^*,j^*})$. 
Note that the second equality holds due to Line~\ref{line:y} of 
Algorithm~\ref{alg:act-constr-coll}.
To see $A\hat{\omega}=b$, we notice that when a vehicle $u$ is added to $vehSet$
at Line~\ref{line:add-u} in the procedure, 
by the condition at Line~\ref{line:cond-add-u}, 
there exists a vehicle $v$ from the current $vehSet$
and $(r,s)\in\cR_v\cap\cR_u$ such that $\hat{f}_{u,v,r,s}=1$ and 
$\hat{t}_{u,r}-\hat{t}_{v,r}=M_{u,v,r,s}(1-\hat{f}_{u,v,r,s})$,
where the equality $\hat{t}_{u,r}-\hat{t}_{v,r}=M_{u,v,r,s}(1-\hat{f}_{u,v,r,s})$
holds because $[\hat{t},\hat{f}]$ is in $\Psprel$, and hence it satisfies that
$\hat{t}_{u,r}-\hat{t}_{v,r}\le M_{u,v,r,s}(1-\hat{f}_{u,v,r,s})$ and
$\hat{t}_{u,r}-\hat{t}_{v,r}\ge -M_{u,v,r,s}(1-\hat{f}_{u,v,r,s})$.
Because $\hat{f}_{u,v,r,s}=1$, it holds that 
$\hat{t}_{u,r}-\hat{t}_{v,r}=0=M_{u,v,r,s}(1-\hat{f}_{u,v,r,s})$.
When the procedure of Algorithm~\ref{alg:act-constr-coll} terminates,
it must have $continue=0$ (Line~\ref{line:cont=0}), which only 
happens when $\hat{t}_{u,O_{u}}=\underline{t}_{u,O_u}$ 
or $\hat{t}_{u,O_{u}}=\ol{t}_{u,O_u}$ holds (Line~\ref{line:cond}).
Moreover, due to the definition of $\mathcal{U}_1$ and $\mathcal{U}_2$,
the constraints $t_{u,O_u}\ge\ul{t}_{u,O_u}$ are active
for all $u\in\mathcal{U}_1$, and constraints $t_{u,O_u}\le\ol{t}_{u,O_u}$
are active for all $u\in\mathcal{U}_2$. Therefore, we have 
$A\hat{\omega}=b$ and 
$\hat{t}_{u,r}-\hat{t}_{v,r}=M_{u,v,r,s}(1-\hat{f}_{u,v,r,s})$. 

We now show that $[\hat{\omega},\hat{f}_{u^*,v^*,i^*,j^*}]$ is the unique solution to the linear equation system
 \begin{equation}\label{eqn:lin-sys}
        \left\{
        \begin{array}{l}
        A\omega=b \\
        \hat{t}_{u^*,i^*}-\hat{t}_{v^*,i^*}=M_{u^*,v^*,i^*,j^*}(1-\hat{f}_{u^*,v^*,i^*,j^*}).
        \end{array}
        \right.
\end{equation}
We prove this by showing that the number of variables in $\omega$ is equal to 
the number of equations in the above linear system, and the equalities in the 
above linear system are linearly independent. 
First, it is easy to see that the equality counterparts of constraints 
\eqref{eqn:ac1}-\eqref{eqn:ac6} are linearly independent because 
any two equalities involve different sets of variables.
Suppose when the algorithm terminates, we have $\cV_1=\{u_k\}^m_{k=1}$,
where vehicles $u_1,u_2,\ldots,u_m$ are added to $\cV_1$ following the sequence.
Line~\ref{line:init-V1-V2} (main part) implies that $u_1=u^*$.
By the condition of adding vehicles (Line~\ref{line:cond-add-u} of the procedure), the following 
equalities must hold:
\begin{equation}\label{eqn:act-constr1}
\begin{aligned}
&t_{u_k,j}=t_{u_k,i}+T_{i,j}  &\forall(i,j)\in\cR_{u_k}, \forall k\in\replacemath{[m]}{\{1,\ldots,m\}}, \\
&\hat{t}_{u_k,r_{k+1}}-\hat{t}_{u_{k+1},r_{k+1}}
=M_{u_ku_{k+1},r_{k+1}s_{k+1}}(1-\hat{f}_{\ol{u_ku_{k+1}},r_{k+1}s_{k+1}}) &\forall k\in[m-1], \\
&\hat{t}_{u_k,r_{k+1}}-\hat{t}_{u_{k+1},r_{k+1}}
=-M_{u_ku_{k+1},r_{k+1}s_{k+1}}(1-\hat{f}_{\ol{u_ku_{k+1}},r_{k+1}s_{k+1}}) &\forall k\in[m-1],\end{aligned}
\end{equation}
where $(r_{k+1},s_{k+1})$ is the edge satisfying the condition at Line~\ref{line:cond-add-u} in Part 2
when $u_{k+1}$ is added into $\cV_1$. 
By the termination condition in Line~\ref{line:cond} (Part 2), we have 
\begin{equation}\label{eqn:act-constr2}
t_{u_m,O_{u_m}}=\ul{t}_{u_m,O_{u_m}} \textrm{ or }\; t_{u_m,O_{u_m}}=\ol{t}_{u_m,O_{u_m}}.
\end{equation}
Similarly, suppose $\cV_2=\{v_k\}^l_{k=1}$, where vehicles $v_1,v_2,\ldots,v_l$
are added into $\cV_2$ following the sequence. We have $v_1=v^*$, and 
\begin{equation}\label{eqn:act-constr3}
\begin{aligned}
&t_{v_k,j}=t_{v_k,i}+T_{i,j}  &\forall(i,j)\in\cR_{v_k}, \forall k\in\replacemath{[l]}{\{1,\ldots,l\}}, \\
&\hat{t}_{v_k,r^{\prime}_{k+1}}-\hat{t}_{v_{k+1},r^{\prime}_{k+1}}
=M_{v_k,v_{k+1},r^{\prime}_{k+1},s^{\prime}_{k+1}}(1-\hat{f}_{\ol{v_k,v_{k+1}},r^{\prime}_{k+1},s^{\prime}_{k+1}}) &\forall k\in[l-1], \\
&\hat{t}_{v_k,r^{\prime}_{k+1}}-\hat{t}_{v_{k+1},r^{\prime}_{k+1}}
=-M_{v_k,v_{k+1},r^{\prime}_{k+1},s^{\prime}_{k+1}}(1-\hat{f}_{\ol{v_k,v_{k+1}},r^{\prime}_{k+1},s^{\prime}_{k+1}}) &\forall k\in[l-1],
\end{aligned}
\end{equation}
where $(r^{\prime}_{k+1},s^{\prime}_{k+1})$ is the edge satisfying 
the condition at Line~\ref{line:cond-add-u} in Part 2
when $v_{k+1}$ is added into $\cV_2$. Similarly, it also holds that 
\begin{equation}\label{eqn:act-constr4}
t_{v_l,O_{v_l}}=\ul{t}_{v_l,O_{v_l}} \textrm{ or }\; t_{v_l,O_{v_l}}=\ol{t}_{v_l,O_{v_l}}.
\end{equation} 
Line~\ref{line:y} at Part 1 guarantees that the following equality holds:
\begin{equation}\label{eqn:act-constr5}
\hat{t}_{u_1,i^*}-\hat{t}_{v_1,i^*}=M_{u_1v_1,i^*,j^*}(1-\hat{f}_{\ol{u_1v_1},i^*,j^*}).
\end{equation} 
Combining \eqref{eqn:act-constr1}--\eqref{eqn:act-constr5}, we see that 
there are 
\begin{displaymath}
\begin{aligned}
&\left(\sum^m_{k=1}|R_{u_k}|+ 2(m-1)+1\right)+\left(\sum^l_{k=1}|R_{v_k}|+ 2(l-1)+1\right)+1 \\
&=\left(\sum^m_{k=1}(|N_{u_k}|-1)+ 2(m-1)+1\right)+\left(\sum^l_{k=1}(|N_{v_k}|-1)+ 2(l-1)+1\right)+1 \\
&=\sum^m_{k=1}|N_{u_k}|+\sum^l_{k=1}|N_{v_k}|+m+l-1
\end{aligned}
\end{displaymath} 
equalities (active constraints). The variables involved
in these equalities are:
\begin{displaymath}
\begin{aligned}
&t_{u_k,i} & \forall i\in N_{u_k}\;\forall k\in\replacemath{[m]}{\{1,\ldots,m\}},\\
&t_{v_k,i} & \forall i\in N_{v_k}\;\forall k\in\replacemath{[l]}{\{1,\ldots,l\}}, \\
&f_{\ol{u_k,u_{k+1}},r_{k+1},s_{k+1}} & \forall k\in\replacemath{[m-1]}{\{1,\ldots,m-1\}},   \\
&f_{\ol{v_k,v_{k+1}},r^{\prime}_{k+1},s^{\prime}_{k+1}} & \forall k\in\replacemath{[l-1]}{\{1,\ldots,l-1\}}, \\
&f_{u_1,v_1,i^*,j^*}.
\end{aligned}
\end{displaymath}
So the number of variables is also $\sum^m_{k=1}|N_{u_k}|+\sum^l_{k=1}|N_{v_k}|+m+l-1$.
Therefore, $\hat{\omega}$ is the unique solution to the linear system of equations \eqref{eqn:lin-sys}.

(b) Consider the following three polytopes:
\begin{equation}
\begin{aligned}
&P=\textrm{conv}\Set*{[\omega,f_{u^*,v^*,i^*,j^*}]}{
\def\arraystretch{2}
\begin{array}{l}
\tilde{A}\omega+\tilde{p}f_{u^*,v^*,i^*,j^*}\ge \tilde{b},\; f_{u^*,v^*,i^*,j^*}\in\{0,1\}, \\
f_{u,v,i,j}\in\{0,1\}\quad\forall(u,v,i,j)\in\mathcal{F}
\end{array}
}, \\
&P_0=\Set*{[\omega,f_{u^*,v^*,i^*,j^*}]\in P}{f_{u^*,v^*,i^*,j^*}=0},\quad \\
&P_1=\Set*{[\omega,f_{u^*,v^*,i^*,j^*}]\in P}{f_{u^*,v^*,i^*,j^*}=1}.
\end{aligned}
\end{equation}
Let $P^{\prime}_0$ and $P^{\prime}_1$ be the linear relaxation polytopes of
$P_0$ and $P_1$, respectively, i.e.,
\begin{equation}
\begin{aligned}
&P^{\prime}_0=\Set*{[\omega,f_{u^*,v^*,i^*,j^*}]}{
\tilde{A}\omega+\tilde{p}f_{u^*,v^*,i^*,j^*}\ge\tilde{b},\quad f_{u^*,v^*,i^*,j^*}=0
} \\
&P^{\prime}_1=\Set*{[\omega,f_{u^*,v^*,i^*,j^*}]}{
\tilde{A}\omega+\tilde{p}f_{u^*,v^*,i^*,j^*}\ge\tilde{b},\quad f_{u^*,v^*,i^*,j^*}=1
}.
\end{aligned}
\end{equation}
Clearly, we have 
$P=\textrm{conv}(\cP_0\cup\cP_1)\subseteq\textrm{conv}(\cP^{\prime}_0\cup\cP^{\prime}_1)$.
Using Theorem~4.39 of \cite{conforti-IP-2014}, we have
\begin{equation}\label{eqn:proj_w}
\begin{aligned}
&\textrm{conv}(\cP^{\prime}_0\cup\cP^{\prime}_1)
=\textrm{Proj}_{[\omega,f_{u^*,v^*,i^*,j^*}]}
\Set*{\begin{array}{l}
\omega,\omega^0,\omega^1, \\
\f,\\
\f^0,\\
\f^1
\end{array}}
{
\def\arraystretch{1.5}
\begin{array}{l}
\omega=\omega^0+\omega^1,\\
\f=\f^0+\f^1, \\
\widetilde{A}\omega^0+\tilde{p}f^0_{u^*,v^*,i^*,j^*}\ge(1- s)\tilde{b},\\
\widetilde{A}\omega^1+\tilde{p}f^1_{u^*,v^*,i^*,j^*}\ge s\tilde{b},\\
\f^0 = 0,\;\f^1=s, \\
0\le s\le 1
\end{array}
} \\
&=\textrm{Proj}_{[\omega,f_{u^*,v^*,i^*,j^*}]}
\Set*{\begin{array}{l}
\omega,\omega^0,\omega^1, \\
\f
\end{array}}
{
\def\arraystretch{1.5}
\begin{array}{l}
\omega=\omega^0+\omega^1,\\
\widetilde{A}\omega^0\ge(1- \f)\tilde{b},\\
\widetilde{A}\omega^1+\tilde{p}\f\ge \f\tilde{b},\\
0\le \f\le 1
\end{array}
}
\end{aligned}
\end{equation}
where $\omega^0$ and $\omega^1$ are two copies of variables of $\omega$.
To derive the family of valid inequalities of $\textrm{conv}(\cP^{\prime}_0\cup\cP^{\prime}_1)$,
we assign dual variables to every constraint in \eqref{eqn:proj_w} as follows:
\begin{displaymath}
{\def\arraystretch{1.5}
\begin{array}{ll}
\textrm{constraint}  & \textrm{dual variable} \\
\omega=\omega^0+\omega^1 & \alpha\in\mathbb{R}^{\textrm{dim}(\omega)}, \\
\widetilde{A}\omega^0\ge (1- f_{u^*,v^*,i^*,j^*})\tilde{b} & \beta^0\ge\bs{0},  \\
\widetilde{A}\omega^1+\tilde{p}\f\ge \f\tilde{b} & \beta^1\ge\bs{0},  \\
f_{u^*,v^*,i^*,j^*}\ge 0 & \gamma^0\ge 0, \\
1-f^1_{u^*,v^*,i^*,j^*}\ge 0 & \gamma^1\ge 0, 
\end{array}
}
\end{displaymath}
Multiply each constraint with the associate dual variable, we obtain the following inequality:
\begin{equation}\label{eqn:L>=0}
L(f_{u^*,v^*,i^*,j^*},\omega,\omega^0,\omega^1;\alpha,\beta^0,\beta^1,
\gamma^0,\gamma^1)\ge0, 
\end{equation}
where
\begin{displaymath}
\begin{aligned}
&L(f_{u^*,v^*,i^*,j^*},\omega,\omega^0,\omega^1;\alpha,\beta^0,\beta^1,
\gamma^0,\gamma^1) \\
&=\alpha^{\top}(\omega-\omega^0-\omega^1)
+\beta^{0\top}[\widetilde{A}\omega^0-(1- f_{u^*,v^*,i^*,j^*})\tilde{b}]
+\beta^{1\top}(\widetilde{A}\omega^1+\tilde{p}\f- \tilde{b}\f) \\
&\quad+\gamma^0f_{u^*,v^*,i^*,j^*}
+\gamma^1(1-f_{u^*,v^*,i^*,j^*}).
\end{aligned}
\end{displaymath}
After reorganizing terms in 
$L(f_{u^*,v^*,i^*,j^*},\omega,\omega^0,\omega^1;\alpha,\beta^0,\beta^1,
\gamma^0,\gamma^1)$,
we get
\begin{displaymath}
\begin{aligned}
&L(f_{u^*,v^*,i^*,j^*},\omega,\omega^0,\omega^1;\alpha,\beta^0,\beta^1,
\gamma^0,\gamma^1) \\
&=\alpha^{\top}\omega+(\beta^{0\top}\widetilde{A}-\alpha^{\top})\omega^0
+(\beta^{1\top}\widetilde{A}-\alpha^{\top})\omega^1 \\
&\quad+(\beta^{0\top}\tilde{b}-\beta^{1\top}\tilde{b}+\beta^{1\top}\tilde{p}
+\gamma^0-\gamma^1)f_{u^*,v^*,i^*,j^*}
-\beta^{0\top}\tilde{b}+\gamma^1,
\end{aligned}
\end{displaymath} 
We can eliminate variables $\omega^0$, $\omega^1$ by setting
their coefficients to be zero in the inequality \eqref{eqn:L>=0}.
Then we obtain the family of all valid inequalities of 
$\textrm{conv}(\cP^{\prime}_0\cup\cP^{\prime}_1)$  
as follows:
\begin{equation}\label{eqn:valid-ineq}
\alpha^{\top}\omega+(\beta^{0\top}\tilde{b}-\beta^{1\top}\tilde{b}
+\beta^{1\top}\tilde{p}+\gamma^0-\gamma^1)f_{u^*,v^*,i^*,j^*}-\beta^{0\top}\tilde{b}+\gamma^1\ge 0,
\end{equation}
if and only if the coefficients $\alpha,\beta^0,\beta^1,\gamma^0,\gamma^1$
satisfy the following conditions:
\begin{equation}\label{eqn:dual-var-cond}
\begin{aligned}
&\widetilde{A}^{\top}\beta^{0}-\alpha=0, \\
&\widetilde{A}^{\top}\beta^{1}-\alpha=0, \\
&\beta^0\ge\bs{0},\;\beta^1\ge\bs{0},\;\gamma^0\ge0,\;\gamma^1\ge0.
\end{aligned}
\end{equation}
Note that the inequality \eqref{eqn:valid-ineq} and the system \eqref{eqn:dual-var-cond},
can be re-scaled by a positive constant. Therefore, we can normalize the system by imposing
$0\le\gamma^1\le 1$ in \eqref{eqn:dual-var-cond}.
Note that the right hand side of \eqref{eqn:valid-ineq} is the objective of 
the linear program \eqref{opt:disj-cut-gen} by setting $\omega\gets\hat{\omega}$, 
and the system \eqref{eqn:dual-var-cond} is the set of constraints in \eqref{opt:disj-cut-gen}.
The normalization $0\le\gamma^1\le 1$ ensures that the linear program \eqref{opt:disj-cut-gen}
is bounded. Because the optimal solution $[\hat{\alpha},\hat{\beta}^0,\hat{\beta}^1,
\hat{\gamma}^0,\hat{\gamma}^1]$ of \eqref{opt:disj-cut-gen}
satisfies \eqref{eqn:dual-var-cond}, it is clear that the inequality \eqref{eqn:disj-valid-ineq}
is valid for $\textrm{conv}(\cP^{\prime}_0\cup\cP^{\prime}_1)$.
Because $P\subseteq\textrm{conv}(\cP^{\prime}_0\cup\cP^{\prime}_1)$,
and $P$ is defined by a subset of variables and constraints from 
the representation of $\Psp_1$, this inequality \eqref{eqn:disj-valid-ineq}
is also valid for $\Psp_1$.

We prove by contradiction to show that the inequality \eqref{eqn:disj-valid-ineq}
strongly separates $[\hat{t},\hat{f}]$ from $\Psp_1$.
Suppose \eqref{eqn:disj-valid-ineq} does not strongly separate 
$[\hat{t},\hat{f}]$ from $\Psp_1$. 
Then we must have $[\hat{t},\hat{f}]\in\Psp_1$, and the following inequality holds
\begin{displaymath}
\hat{\alpha}^{\top}\hat{\omega}
+(\hat{\beta}^{0\top}\tilde{b}-\hat{\beta}^{1\top}\tilde{b}
+\hat{\beta}^{1\top}\tilde{p}+\hat{\gamma}^0-\hat{\gamma}^1)\hat{f}_{u^*,v^*,i^*,j^*}
-\hat{\beta}^{0\top}\tilde{b}+\hat{\gamma}^1\ge 0.
\end{displaymath}
It implies that 
\begin{displaymath}
\ba
&\alpha^{\top}\hat{\omega}
+(\beta^{0\top}\tilde{b}-\beta^{1\top}\tilde{b}
+\beta^{1\top}\tilde{p}+\gamma^0-\gamma^1)\hat{f}_{u^*,v^*,i^*,j^*}
-\beta^{0\top}\tilde{b}+\gamma^1	\\
&\ge \hat{\alpha}^{\top}\hat{\omega}
+(\hat{\beta}^{0\top}\tilde{b}-\hat{\beta}^{1\top}\tilde{b}
+\hat{\beta}^{1\top}\tilde{p}+\hat{\gamma}^0-\hat{\gamma}^1)\hat{f}_{u^*,v^*,i^*,j^*}
-\hat{\beta}^{0\top}\tilde{b}+\hat{\gamma}^1\ge 0,
\ea
\end{displaymath}
for any coefficient vector $[\alpha,\beta^0,\beta^1,\gamma^0,\gamma^1]$
satisfying \eqref{eqn:dual-var-cond}, 
because $[\hat{\alpha},\hat{\beta}^0,\hat{\beta}^1,\hat{\gamma}^0,\hat{\gamma}^1]$
is an optimal solution of \eqref{opt:disj-cut-gen}.
This indicates that the point $[\hat{\omega},\hat{f}_{u^*,v^*,i^*,j^*}]$ 
respects all valid inequalities 
for $\textrm{conv}(\cP^{\prime}_0\cup\cP^{\prime}_1)$.
Therefore, we must have 
$[\hat{\omega},\hat{f}_{u^*,v^*,i^*,j^*}]\in\textrm{conv}(\cP^{\prime}_0\cup\cP^{\prime}_1)$.
Let $P_{\textrm{rel}}$ denote the linear relaxation of $P$, i.e., relaxing 
the integral constraints on $f_{u^*,v^*,i^*,j^*}$ and $f_{u,v,i,j}$ for all $(u,v,i,j)\in\mathcal{F}$.
Note that $\omega$ is a subset of variables from $[t,f]$,
and $P_{\textrm{rel}}$ is defined by the subset of constraints
from $\Psprel$ induced by $[\omega,f_{u^*,v^*,i^*,j^*}]$.
Since $[\hat{t},\hat{f}]$ is an extreme point of $\Psprel$, 
$[\hat{\omega},\hat{f}_{u^*,v^*,i^*,j^*}]$ must be an extreme point of $P_{\textrm{rel}}$.
Moreover, since $\textrm{conv}(\cP^{\prime}_0\cup\cP^{\prime}_1)\subseteq\cP_{\textrm{rel}}$,
and $[\hat{\omega},\hat{f}_{u^*,v^*,i^*,j^*}]\in\textrm{conv}(\cP^{\prime}_0\cup\cP^{\prime}_1)$,
it follows that $[\hat{\omega},\hat{f}_{u^*,v^*,i^*,j^*}]$ is an extreme point of $\textrm{conv}(\cP^{\prime}_0\cup\cP^{\prime}_1)$.
However, at any extreme point of $\textrm{conv}(\cP^{\prime}_0\cup\cP^{\prime}_1)$,
the $f_{u^*,v^*,i^*,j^*}$ entry must be 0 or 1, 
which contradicts $\hat{f}_{u^*,v^*,i^*,j^*}$ being fractional.
\end{proof}

\subsection{Supplement of Section~\ref{sec:valid-ineq-SPF2}}
\begin{proposition}\label{prop:Ax=b}
Let $A$ be a $m\times n$ matrix of full rank, with $m<n$, 
and $b\in\mathbb{R}^m$ is a vector. 
A vector $x\in\mathbb{R}^n$ satisfies the constrained fractional condition if 
\begin{displaymath}
Ax=b, \emph{ and } 0<x_i<1 \quad \forall i\in\replacemath{[n]}{\{1,\ldots,n\}}.
\end{displaymath} 
Suppose there exists a vector $\hat{x}$ satisfying the constrained fractional condition. 
Then there exist two vectors $\hat{y}$ and $\hat{z}$ such that both $\hat{y}$ and $\hat{z}$
satisfy the constrained fractional condition, and $\hat{x}=(\hat{y}+\hat{z})/2$.  
\end{proposition}
\begin{proof}
Since $\rank(A)=m$, and $m<n$, we can select $m$ independent columns
of $A$ which form a sub-matrix $A^1$. Write $A$ as $A=[A^1,A^2]$,
and $x=[x^1,x^2]^\top$, where $x^1$ and $x^2$ are sub-vectors of $x$
corresponding to columns in $A^1$ and $A^2$, respectively. Then the 
linear system can be written as:
\begin{displaymath}
A^1x^1=b-A^2x^2.
\end{displaymath}
Select an arbitrary non-zero $n-m$ dimensional vector $c$, and choose a sufficient
small positive constant $\delta$ and let
\begin{displaymath}
p=\begin{bmatrix} 
-(A^1)^{-1}A^2c \\
c  
\end{bmatrix}
\end{displaymath}
and then it can be verified that 
\bdm
\begin{aligned}
&A(\hat{x}\pm\delta p)=A\hat{x}\pm\delta[A^1,A^2]\begin{bmatrix} 
-(A^1)^{-1}A^2c \\
c  
\end{bmatrix}
=A\hat{x}=b.
\end{aligned}
\edm
Furthermore, since every entry of $\hat{x}$ is bounded away from 0 and 1,
a sufficient small perturbation added to $\hat{x}$ can still keep them away 
from the boundary.
This shows that $\hat{x}+\delta p$ and $\hat{x}-\delta p$ satisfy the constrained fractional
condition. Clearly, we have $\hat{x}=\frac{1}{2}[(\hat{x}+\delta p)+(\hat{x}-\delta p)]$.
\end{proof}

\begin{algorithm}
  \footnotesize
  \caption{Counting $|\mtc{I}_0|$.}
  \label{alg:seq-take-E1}
  \begin{algorithmic}[1]
  \State {\bf Input}: $\mtc{H}_1,\mtc{H}_2$.
  \State {\bf Output}: The value of $|\mtc{I}_0|$.
  \State Set $N\gets \sum_{D\in\mtc{H}_1}|\mtc{I}(D)|$, $\mtc{H}^{\prime}_1\gets\mtc{H}_1$, 
      $\mtc{H}^{\prime}_2\gets\mtc{H}_2$.
  \While{$\mtc{H}^{\prime}_2\neq\emptyset$.} 
    \While{$\exists\cD_1\in\mtc{H}^{\prime}_1$, and $\exists\cD_2\in\mtc{H}^{\prime}_2$ 
         such that $\mtc{I}(D_1)\cap\mtc{I}(D_2)\neq\emptyset$.} \label{lin:while-cond}
      \State Set $\mtc{H}^{\prime}_1\gets\mtc{H}^{\prime}_1\cup\{D_2\}$,
          $\mtc{H}^{\prime}_2\gets\mtc{H}^{\prime}_2\setminus\{D_2\}$,
          and $N\gets\cN+|\mtc{I}(D_2)|-1$.
    \EndWhile
    \State Select an arbitrary linear equation $D\in\mtc{H}^{\prime}_2$. \label{lin:add-arb}
    \State Set $\mtc{H}^{\prime}_1\gets\mtc{H}^{\prime}_1\cup\{D\}$,
          $\mtc{H}^{\prime}_2\gets\mtc{H}^{\prime}_2\setminus\{D\}$,
          and $N\gets\cN+|\mtc{I}(D)|$.
  \EndWhile
  \State \Return $N$.
  \end{algorithmic}
\end{algorithm}

\label{proof_of_star_part}
\begin{proof}{Proof of Theorem~\ref{thm:facet-star-partition}.}
(a) For clarity, we omit the edge index $(i,j)$ in the $f$ variables for the proof. \newline
Let $P^{\prime}$ be the polytope on the right side of \eqref{eqn:claim1}. 
We need to show that $\widetilde{P}_{3,i,j}=P^{\prime}$. Let 
\begin{equation}
  \mtc{S}=\Big\{ \displaystyle  \sum_{v\in\cV_{i,j}<v_{\textrm{max}}}f_{v_{\textrm{max}},v}\le 1,\quad f_{u,v}+\sum_{w\in\cV_{i,j}<v}f_{v,w}\le 1 \quad \forall u,v\in\cV_{i,j}, u>v,\;v>v_{\textrm{min}} \Big\}
\end{equation}
be the set of structured constraints in $P^{\prime}$.
We begin by showing that every constrain in $\mtc{S}$ defines a facet of $\widetilde{P}_{3,i,j}$.
Note that the dimension of $\widetilde{P}_{3,i,j}$ is $\textrm{dim}\widetilde{P}_{3,i,j}=\binom{|\cV_{i,j}|}{2}$. 
To ease the presentation, we let $n=|\cV_{i,j}|$ and label vehicles from $\cV_{i,j}$ as 
$1,2,\ldots,n$ in the proof.
To show that the constraint $f_{u,v}+\sum_{w\in\cV_{i,j}<v}f_{v,w}\le 1$ is valid for $\widetilde{P}_{3,i,j}$,
we consider the following possible cases: (a) $f_{u,v}=0$ and (b) $f_{u,v}=1$.
For Case (a), the constraint $f_{u,v}+\sum_{w\in\cV_{i,j}<v}f_{v,w}\le 1$
reduces to $\sum_{w\in\cV_{i,j}<v}f_{v,w}\le 1$, which is valid because
there is at most one vehicle that $v$ can follow. For Case (b), the constraint 
$f_{u,v}+\sum_{w\in\cV_{i,j}<v}f_{v,w}\le 1$ reduces to 
$\sum_{w\in\cV_{i,j}<v}f_{v,w}\le 0$, which is valid because $v$ is 
a lead vehicle and it cannot follow any other vehicles. Therefore, the inequality
$f_{u,v}+\sum_{w\in\cV_{i,j}<v}f_{v,w}\le 1$ is valid for $\widetilde{P}_{3,i,j}$.
We now construct $\binom{n}{2}$ affinely independent points from $\widetilde{P}_{3,i,j}$
that satisfy the equality $f_{u,v}+\sum_{w\in\cV_{i,j}<v}f_{v,w}=1$.
These points are
\begin{displaymath}
\begin{aligned}
&e_{u,v},  & \\
&e_{u,v}+e_{v_2v_1}, & \forall v_1,v_2\in\cV_{i,j},\; v_2>v_1,\; v_1,v_2\notin\{u,v\}, \\
&e_{u,v}+e_{v_2v}, & \forall v_2\in\cV_{i,j},\; v_2>v,\; v_2\neq u, \\
&e_{v,w},  & \forall w\in\cV_{i,j},\;w<v,  \\
&e_{v1}+e_{\ol{u,w}},  & \forall w\in\cV_{i,j},\;w\notin\{u,v\}, 
\end{aligned}
\end{displaymath}
where we recall that $\ol{u,w}$ represents the tuple $(u,w)$ if $u>w$, 
and the tuple $(w,u)$ if $w>u$.
It can be verified that the number of points given above is equal to 
\begin{displaymath}
1+\binom{n-2}{2}+(n-v-1)+(v-1)+(n-2)
=\binom{n-2}{2}+(n-1)+(n-2)=\frac{1}{2}n(n-1)=\binom{n}{2}.
\end{displaymath} 
The proof of that the inequality 
$\sum_{v\in\cV_{i,j}<v_{\textrm{max}}}f_{v_{\textrm{max}},v}\le 1$
defines a facet is similar and easier, which we omit it here.
Therefore, every inequality in $\mtc{S}$ defines a facet of $\widetilde{P}_{3,i,j}$,
and hence $\widetilde{P}_{3,i,j}\subseteq P^{\prime}$.

Next we show that every integral point in $P^{\prime}$ specifies a star partition.
We claim that for every integral point $\hat{f}$ in $P^{\prime}$,
the (undirected) graph $\hat{G}$ induced by this point can not contain a path of length greater than 2. 
We prove this claim by contradiction. Suppose $\hat{G}$ contains a path
of length at least 3, which is represented as $L:\;v_1\to v_2 \to v_3 \to v_4 \to \ldots\to v_k$,
where $k\ge 4$. Without loss of generality, we assume $v_2>v_3$, the proof
for the case $v_2<v_3$ is similar. Given that $v_2>v_3$,
the following constraints from $\mtc{S}$ are violated by $\hat{f}$
for different cases:
\begin{displaymath}
\begin{aligned}
&\textrm{cases } & & \textrm{constraint that is violated by $\hat{f}$ } \\
&v_1<v_2 \textrm{ and } v_2\le n-1, & & f_{u,v_2}+\sum_{v\in\cV_{i,j}<v_2}f_{v_2,v}\le 1\;\;\forall u>v_2, \\
&v_1<v_2 \textrm{ and } v_2= n, & & \sum_{v\in\cV_{i,j}<v_2}f_{v_2,v}\le 1, \\
&v_1>v_2, & & f_{v_1,v_2}+\sum_{v\in\cV_{i,j}<v_2}f_{v_2,v}\le 1.
\end{aligned}
\end{displaymath}
This contradicts $\hat{f}\in P^{\prime}$.
It is easy to verify that if the length of every path in an undirected graph is no greater than 2,
the graph must be a star partition graph. 

It remains to show that every extreme point of $P^{\prime}$ is an integral point.
Suppose $\hat{f}$ is an extreme point of $P^{\prime}$ that contains fractional entries. 
We call a structured constraint from $\cS$ a \emph{pivot constraint} if it satisfies 
the following two conditions:
(1) The constraint is active at $\hat{f}$. (2) The constraint contains at least one fractional
entry of $\hat{f}$. It is easy to see that any non-zero entry of $\hat{f}$ 
that is involved in a pivot constraint must be fractional. Indeed, suppose 
the constraint $f_{u,v}+\sum_{w\in\cV_{i,j}<v}f_{v,w}\le 1$ is a pivot constraint.
If any non-zero entry involved in this constraint is not fractional, it must be one,
and then all the other entries involved in the constraint must be zero. 
This contradicts the constraint is a pivot.
For clarity of notation, we use $C_n$ to denote the structured constraint
$\sum_{v\in\cV_{i,j}<n}f_{n,v}\le 1$,
and use $C_{u,v}$ to denote the structured constraint
$f_{u,v}+\sum_{w\in\cV_{i,j}<v}f_{v,w}\le 1$. 
We create a mapping $\mtc{M}$ that maps a pivot constraint to a linear equation. 
The mapping works as follows:
For a pivot constraint $C$, we first set the inequality to be equality,
and remove the variables in $C$ that correspond to zero entries in $\hat{f}$.
For example, suppose $f_{5,4}+f_{4,1}+f_{4,2}+f_{4,3}\le 1$ is a pivot constraint,
and $\hat{f}_{54}=\hat{f}_{42}=0$, $\hat{f}_{41}=\hat{f}_{43}=1/2$.
This pivot constraint is then transformed into the linear equation:
$f_{4,1}+f_{4,3}=1$. We denote the mapping result as:
$\{f_{4,1}+f_{4,3}=1\}=\mtc{M}(f_{5,4}+f_{4,1}+f_{4,2}+f_{4,3}\le 1)$ for this case.
Let $\mtc{C}$ be the set of all pivot constraints, and define a linear equation system
\bdm
\mtc{H}=\Set*{\mtc{M}(C)}{C\in\mtc{C}}.
\edm  
So $\mtc{H}$ is the linear equation system defined by mapping all pivot constraints into
linear equations. Note that if the linear equations in $\mtc{H}$ are not linearly independent,
we can impose a post process on $\mtc{H}$ by keeping a maximum set of linear equations 
that are linearly independent. From now on, we assume that the linear equations 
in $\mtc{H}$ are linearly independent.
We define the following subsets of $\mtc{H}$:
\bdm
\begin{aligned}
&\mtc{H}_1=\Set*{\mtc{M}(C_{u,v})}{\forall\;C_{u,v}\in\mtc{C},\textrm{ s.t. }\hat{f}_{u,v}=0}
\cup\Set*{\mtc{M}(C_n)}{\forall\; C_n\in\mtc{C}}, \\
&\mtc{H}_2=\Set*{\mtc{M}(C_{u,v})}{\forall\;C_{u,v}\in\mtc{C},\textrm{ s.t. }0<\hat{f}_{u,v}<1}.
\end{aligned}
\edm 
Note that $\mtc{H}_1,\mtc{H}_2$ form a partition of $\mtc{H}$. 
Define the following set of vehicle indices:
\bdm
\begin{aligned}
&\mtc{I}_0=\Set*{(u,v)}{f_{u,v}\textrm{ is involved in a linear equation from }\mtc{H}}, \\
&\mtc{I}_1=\Set*{(u,v)}{f_{u,v}\textrm{ is involved in a linear equation from }\mtc{H}_1}, \\
&\mtc{I}_2=\Set*{(u,v)}{f_{u,v}\textrm{ is involved in a linear equation from }\mtc{H}_2},
\end{aligned}
\edm
For any linear equation $D\in\mtc{H}$, we also define 
\bdm
\mtc{I}(D)=\Set*{(u,v)}{f_{u,v}\textrm{ is involved in }D}.
\edm
We claim that the number of linear equations in $\mtc{H}$ is strictly smaller
than the total number of variables involved in all linear equations from $\mtc{H}$,
i.e., $|\mtc{H}|<|\mtc{I}_0|$. If the claim holds, the linear equation system $\mtc{H}$
can then be represented as 
\begin{equation}\label{eqn:AY=b}
AY_{\mtc{I}_0}=b,
\end{equation}
for some matrix $A$ and vector $b$, where 
$Y_{\mtc{I}_0}:=\Set*{f_{u,v}}{(u,v)\in\mtc{I}_0}$ is the sub-vector of $y$ induced by
the index set $\mtc{I}_0$. Note that $A$ is a $|\mtc{H}|\times|\mtc{I}_0|$ matrix of 
full rank, the vector $\widehat{Y}_{\mtc{I}_0}:=\Set*{\hat{f}_{u,v}}{(u,v)\in\mtc{I}_0}$
satisfies \eqref{eqn:AY=b}, and every entry in $\widehat{Y}_{\mtc{I}_0}$ is fractional.
Then by Proposition~\ref{prop:Ax=b}, there exist two fractional vectors 
$\widehat{Y}^1_{\mtc{I}_0}$ and $\widehat{Y}^2_{\mtc{I}_0}$ satisfying \eqref{eqn:AY=b},
and $\widehat{Y}_{\mtc{I}_0}=\frac{1}{2}(\widehat{Y}^1_{\mtc{I}_0}+\widehat{Y}^2_{\mtc{I}_0})$.
We can construct two vectors $\hat{f}^1$, $\hat{f}^2$ in the way that 
$\hat{f}^1_{u,v}=\widehat{Y}^1_{u,v}$, $\hat{f}^2_{u,v}=\widehat{Y}^2_{u,v}$,
and $\hat{f}^1_{u,v}=\hat{f}^2_{u,v}=\hat{f}_{u,v}$ for $(u,v)\notin\mtc{I}_0$.
By construction, we have $\hat{f}^1,\hat{f}^2\in\cP^{\prime}_1$ and 
$\hat{f}=\frac{1}{2}(\hat{f}^1+\hat{f}^2)$, which contradicts to the assumption
that $\hat{f}$ is an extreme point of $P^{\prime}$.

We now show that the claim holds. First we notice that 
for any two different linear equations $D$ and $D^{\prime}$ from $\mtc{H}_1$,
we must have $\mtc{I}(D)\cap\mtc{I}(D^{\prime})=\emptyset$. 
To see this, we realize that the only situation that $\mtc{I}(D)\cap\mtc{I}(D^{\prime})\neq\emptyset$
is in the case that there exist three distinct vehicle indices $u,u^{\prime},v\in\replacemath{[n]}{\{ 1,\ldots,n \}}$
such that 
\bdm
\begin{aligned}
\hat{f}_{u,v}=\hat{f}_{u^{\prime},v}=0, \quad D = \mtc{M}\left(C_{u,v}\right), \quad
D^{\prime} = \mtc{M}\left(C_{u^{\prime},v}\right).
\end{aligned}
\edm
Since the variables $f_{u,v}$ and $f_{u^{\prime},v}$ are removed from 
the pivot constraints $C_{u,v}$ and $C_{u^{\prime},v}$ under the mapping $\mtc{M}$, respectively,
we must have $D=D^{\prime}$, which contradicts to the assumption that
$D$ and $D^{\prime}$ are different. 
Second, we notice that $|\mtc{I}(D)|\ge 2$
for any $D\in\mtc{H}_1$. This shows that 
\begin{equation}\label{eqn:I1>2H1}
|\mtc{I}_1|=\sum_{D\in\mtc{H}_1}|\mtc{I}(D)|\ge 2|\mtc{H}_1|.
\end{equation}
Now we sequentially take elements (linear equations) from 
$\mtc{H}_2$ one at a time following Algorithm~\ref{alg:seq-take-E1}.
We show that the number returned from Algorithm~\ref{alg:seq-take-E1}
equals $|\mtc{I}_0|$. When the condition for the while loop at Line~\ref{lin:while-cond}
of Algorithm~\ref{alg:seq-take-E1} holds, the linear equations $D_1$ and $D_2$ 
must contain exactly one common variable,
and $\mtc{I}(D)\cap\mtc{I}(D_2)=\emptyset$ for all 
$D\in\mtc{H}^{\prime}_1\setminus\{D_1\}$. 
Therefore, the number of new variables introduced to $\mtc{H}^{\prime}_1$
after adding $D_2$ into $\mtc{H}^{\prime}_1$ is $|\mtc{I}(D_2)|-1$.
When Algorithm~\ref{alg:seq-take-E1} reaches Line~\ref{lin:add-arb},
it must hold that $\mtc{I}(\mtc{H}^{\prime}_2)\cap\mtc{I}(\mtc{H}^{\prime}_1)=\emptyset$,
where $\mtc{I}(\mtc{H}^{\prime}_i)=\cup_{D\in\mtc{H}^{\prime}_i}\mtc{I}(D)$ for $i=1,2$. 
Then if an arbitrary $D\in\mtc{I}(\mtc{H}^{\prime}_2)$ is added into $\mtc{H}^{\prime}_1$,
the number of new variables introduced to $\mtc{H}^{\prime}_1$ is $|\mtc{I}(D)|$.
Therefore, when the algorithm terminates, we must have $N=|\mtc{I}_0|$.
Let $D_1,D_2,\ldots,D_m$ be the sequence of linear equations from $\mtc{H}^{\prime}_2$
that is added into $\mtc{H}^{\prime}_1$ from the beginning to the end of Algorithm~\ref{alg:seq-take-E1},
and notice that $\mtc{I}(D_k)\ge 2$ for all $k\in\replacemath{[m]}{\{1,\ldots,m\}}$.
If $|\mtc{H}_1|\ge 1$, we must have
\bdm
\begin{aligned}
&|\mtc{I}_0|=\left|\mtc{I}(\mtc{H}_1)\cup\left(\cup^m_{k=1}\mtc{I}(D_k) \right) \right|
\ge|\mtc{I}(\mtc{H}_1)|+\sum^m_{k=1}(|\mtc{I}(D_k)|-1) \\
&=\sum_{D\in\mtc{H}_1}|\mtc{I}(D)|+\sum^m_{k=1}(|\mtc{I}(D_k)|-1) \\
&\ge 2|\mtc{H}_1|+m=2|\mtc{H}_1|+|\mtc{H}_2| \\
&>|\mtc{H}_1|+|\mtc{H}_2|=|\mtc{H}|,
\end{aligned}
\edm
where we use \eqref{eqn:I1>2H1} in the above inequality.
If $|\mtc{H}_1|=0$ ($\mtc{H}_1=\emptyset$), we have
\bdm
\begin{aligned}
&|\mtc{I}_0|=\left|\cup^m_{k=1}\mtc{I}(D_k)\right|\ge |\mtc{I}(D_1)|+\sum^m_{k=2}(|\mtc{I}(D_2)|-1) \\
&\ge 2+(m-1)=m+1>m=|\mtc{H}|.
\end{aligned}
\edm
This shows that we have $|\mtc{I}_0|>|\mtc{H}|$, which concludes the proof of the claim.

(b) We omit the edge index $(i,j)$ in the variables $f_{u,v,i,j}$ for the proof.
We first show that the inequality \eqref{eqn:sum_y<Q-1} is valid for $P_{3,i,j}$.
In the case that there is no leading vehicle in $U$, we will have
$\sum_{u,v\in U:u>v}f_{u,v,i,j}=0\le \lambda-1$. Let $\hat{f}$ be an arbitrary
star-partition vector.
In the case that there is only one leading vehicle in $U$, and suppose the index 
of the leading vehicle is $v_0$, then we have
\bdm
\begin{aligned}
  &\sum_{u,v\in U:u>v}\hat{f}_{u,v}=\sum_{u\in U>v_0}\hat{f}_{u,v_0}
  \le\textrm{min}\{\lambda-1,\textrm{Card}\Set*{u\in U}{u>v_0}\} \\
&\le\textrm{min}\{\lambda-1,|U|-1\}=\lambda-1.
\end{aligned}
\edm
In the case that there are multiple leading vehicles in $U$,
we denote their indices as $v_1,v_2,\ldots,v_k$ ($k\ge 2$).
The vehicle set $U$ can be partitioned into $\{U_i\}^k_{i=0}$
such that $\hat{f}_{u,v_i}=1$ $\forall u\in U_i\setminus\{v_i\},\;\forall i\in\replacemath{[k]}{ \left\{ 1, \ldots, k \right\}}$,
and $\hat{f}_{u,v}=0$ $\forall u\in U_0,\;\forall v\in U\setminus\{u\},\;u>v$.
It follows that 
\bdm
\begin{aligned}
&\sum_{u,v\in U:u>v}\hat{f}_{u,v}=\sum^k_{i=1}\sum_{u\in U_i\setminus\{v_i\}}\hat{f}_{u,v_i}
=\sum^k_{i=1}(|U_i|-1)=|U|-|U_0|-k \\
&\le|U|-2=|\lambda|-1.
\end{aligned}
\edm
Therefore, the inequality \eqref{eqn:sum_y<Q-1} is valid for $P_{3,i,j}$.

We now show that the inequality \eqref{eqn:sum_y<Q-1} defines a facet of $P_{3,i,j}$.
Instead of directly constructing $\binom{|V_2|}{2}$ affinely independent points from $P_{3,i,j}$ that 
satisfy \eqref{eqn:sum_y<Q-1} as an equality, we use an alternative way of proving that 
it is facet defining. Denote the inequality \eqref{eqn:sum_y<Q-1} as 
$a^{\top}f\le a_0$. Obviously, the set $\Set*{\hat{f}\in P_{3,i,j}}{a^{\top}\hat{f}=a_0}$ is nonempty.
Suppose the inequality $b^{\top}f\le b_0$ defines a facet of $P_{3,i,j}$ such that 
$\Set*{\hat{f}\in P_{3,i,j}}{a^{\top}\hat{f}=a_0}\subseteq\Set*{\hat{f}\in P_{3,i,j}}{b^{\top}\hat{f}=b_0}$.
It suffices to show that there exists a positive constant $\alpha>0$
such that $b=\alpha a$. If this holds, we can conclude that $a^{\top}\hat{f}=a_0$
defines a facet of $P_{3,i,j}$. 
For notational convenience, let $\cV_{i,j}=\{1,2,\ldots,n\}$.
Without loss of generality, suppose 
$U=\{v_1,v_2,\ldots,v_{\lambda+1}\}$, and $1<v_1<v_2<\ldots<v_{\lambda+1}<n$.
We need to show the following claims. \newline
\textbf{Claim 1}: $b_{u,v}=0$ for all $u,v\in\replacemath{[n]}{\{1,\ldots,n\}},\;u>v$ if
$\{u,v\}\cap U\neq\{u,v\}$. There are three cases:
(1) $\{u,v\}\cap U=\emptyset$; (2) $\{u,v\}\cap U=\{u\}$;
and (3) $\{u,v\}\cap U=\{v\}$.
 For Case (1), consider the two points $\hat{f}^1=\sum^\lambda_{i=2}e_{v_iv_1}$
 and $\hat{f}^2=\sum^\lambda_{i=2}e_{v_iv_1}+e_{u,v}$.
 Since $\hat{f}^1$ and $\hat{f}^2$ are on the plane $a^{\top}f=a_0$,
 they are also on the plane $b^{\top}f=b_0$, which implies that
 $\sum^\lambda_{i=2}b_{v_iv_1}=b_0$ and $\sum^\lambda_{i=2}b_{v_iv_1}+b_{u,v}=b_0$, 
 and hence $b_{u,v}=0$.
 For Case (2), suppose $u=v_k$, and consider the two points 
 $\hat{f}^1=\sum^{k-1}_{i=2}e_{v_iv_1}+\sum^{\lambda}_{i=k+1}e_{v_iv_1}$
 and $\hat{f}^2=\sum^{k-1}_{i=2}e_{v_iv_1}+\sum^{\lambda}_{i=k+1}e_{v_iv_1}+e_{v_kv}$
 that are on the plane $a^{\top}f=a_0$.
 Substituting them into the plane $b^{\top}f=b_0$ gives $b_{u,v}=0$.
 For Case (3), $b_{u,v}=0$ can be proved similarly. \newline
 \textbf{Claim 2}: $b_{u,v}=b_{u^{\prime},v}$ for all distinct $u,u^{\prime},v\in U$
 satisfying $u>v$ and $u^{\prime}>v$.
 We first prove for the case that $v=v_1$. 
 For this case, we consider the following two points on the plane $a^{\top}f=a_0$:
 $\hat{f}^1=\sum_{w\in U\setminus\{u,v_1\}}e_{w,v_1}$, 
 and $\hat{f}^2=\sum_{w\in U\setminus\{u^{\prime},v_1\}}e_{w,v_1}$.
 Substituting the above two points into the equation $b^{\top}f=b_0$
 gives $b_{u,v_1}=b_{u^{\prime},v_1}$.
 For the case that $v\neq v_1$,
 we consider the following two points on the plane $a^{\top}f=a_0$:
 $\hat{f}^1=e_{u,v}+\sum_{w\in U\setminus\{u,v,v_1\}}e_{w,v_1}$,
 and $\hat{f}^2=e_{u^{\prime},v}+\sum_{w\in U\setminus\{u^{\prime},v,v_1\}}e_{w,v_1}$.
 Substituting the above two points into the equation $b^{\top}f=b_0$
 gives: 
 \bdm
 b_{u,v}+\sum_{w\in U\setminus\{u,v,v_1\}}b_{w,v_1}
 =b_{u^{\prime},v}+\sum_{w\in U\setminus\{u^{\prime},v,v_1\}}b_{w,v_1},
 \edm
 which implies that $b_{u,v}-b_{u,v_1}=b_{u^{\prime},v}-b_{u^{\prime},v_1}$.
 Since we have already proved that $b_{u,v_1}=b_{u^{\prime},v_1}$,
 the above equation gives that $b_{u,v}=b_{u^{\prime},v}$.\newline
 \textbf{Claim 3}: $b_{u,v}=b_{u,v^{\prime}}$ for all distinct $u,v,v^{\prime}\in U$.
 The proof of Claim 3 is similar to the proof of Claim 2. That is we first show
 that the claim holds for $u=v_{\lambda+1}$, and then we show it holds in general. \newline
 \textbf{Claim 4}: $b_{u,v}=b_{v,w}$ for all distinct $u,v,w\in U$ with $u>v>w$.
 Assuming $w\neq v_1$ (the proof for the case $w=v_1$ is similar), 
 we consider the following two points:
 $\hat{f}^1=e_{u,v}+\sum_{u^{\prime}\in U\setminus\{u,v,v_1\}}e_{u^{\prime},v_1}$,
 and $\hat{f}^2=e_{v,w}+\sum_{u^{\prime}\in U\setminus\{w,v,v_1\}}e_{u^{\prime},v_1}$.
 Substituting the above two points into the equation $b^{\top}y=b_0$ gives
 \bdm
 b_{u,v}+\sum_{u^{\prime}\in U\setminus\{u,v,v_1\}}b_{u^{\prime},v_1}
 =b_{v,w}+\sum_{u^{\prime}\in U\setminus\{w,v,v_1\}}b_{u^{\prime},v_1},
 \edm
 which implies that $b_{u,v}-b_{u,v_1}=b_{v,w}-b_{w,v_1}$.
 Since we have proved that $b_{u,v_1}=b_{w,v_1}$, the above equation
 further implies that $b_{u,v}=b_{v,w}$.
 \newline
 \textbf{Claim 5}: $b_{u,v}=b_{u^{\prime},v^{\prime}}$ for all distinct $u,u^{\prime},v,v^{\prime}\in U$.
 Claim 5 can be implied by Claims 2, 3 and 4. 
 There are three cases: (1) $\textrm{min}\{u,u^{\prime}\}>\textrm{max}\{v,v^{\prime}\}$;
(2) $v>u^{\prime}$; and (3) $v^{\prime}>u$.
For Case (1), we have $b_{u,v}=b_{u^{\prime},v}=b_{u^{\prime},v^{\prime}}$.
For Case (2), we have $b_{u,v}=b_{v,u^{\prime}}=b_{u^{\prime},v^{\prime}}$.
For Case (3), we have $b_{u^{\prime},v^{\prime}}=b_{v^{\prime},u}=b_{u,v}$.
Therefore, Claims 1-5 show that $b=\alpha a$ for some $\alpha>0$,
and hence \eqref{eqn:sum_y<Q-1} defines a facet of $P_{3,i,j}$.
\end{proof}


\end{document}